\documentclass[letterpaper,oneside,11pt]{article}
\usepackage[english]{babel} 
\usepackage[T1]{fontenc}
\usepackage[utf8]{inputenc}
\usepackage{lmodern} 

\usepackage[abbrev]{amsrefs}
\renewcommand{\MR}[1]{}
\usepackage{hyphenat}
\usepackage{dsfont}
\usepackage{amsmath, amssymb}
\usepackage{amsthm}
\usepackage{amsfonts}
\newtheorem{teo}{Theorem}[section]
\newtheorem{prop}[teo]{Proposition}
\newtheorem{lemma}[teo]{Lemma}
\newtheorem{re}[teo]{Remark}
\newtheorem{de}[teo]{Definition}
\newtheorem{cor}[teo]{Corollary}
\newtheorem{claim}[teo]{Claim}
\numberwithin{equation}{section}
\newcommand{\Real}{\mathbb R}

\newcommand{\n}{\mathcal{N}}

\usepackage{mathtools}
\mathtoolsset{showonlyrefs}
\usepackage{enumerate}

\usepackage{parskip}
\usepackage{xcolor}
\usepackage[normalem]{ulem}

\usepackage[plainpages=false,pdfpagelabels,pdfusetitle,pdfdisplaydoctitle, pdfstartpage=1, pdfstartview={{XYZ null null 1.00}}]{hyperref}
\usepackage[top=3cm, bottom=2cm, left=3cm, right=2cm]{geometry}
\author{\textsc{Luccas Campos}, \textsc{Luiz Gustavo Farah}, \textsc{Svetlana Roudenko}}

\renewcommand{\Re}{\operatorname{Re}}
\renewcommand{\Im}{\operatorname{Im}}

\usepackage{etoolbox}
\makeatletter
\patchcmd{\l@section}{1.0em}{0.4em}{}{}
\makeatother

\begingroup 
\makeatletter
   \@for\theoremstyle:=theorem,de,re,pro,lem,coro,plain\do{%
     \expandafter\g@addto@macro\csname th@\theoremstyle\endcsname{%
        \addtolength\thm@preskip\parskip
     }%
   }
\endgroup

\begin{document}
\scrollmode
\title{Threshold solutions for the nonlinear Schrödinger equation}
\date{}
\maketitle

\makeatletter{\renewcommand*{\@makefnmark}{}
\makeatother}
\makeatletter{\renewcommand*{\@makefnmark}{}
\makeatother}
\vspace{-.5cm}
\begin{abstract}\noindent
We study the focusing NLS equation in $\mathbb{R}^N$ in the mass-supercritical and energy-subcritical (or \textit{intercritical}) regime, with $H^1$ data at the mass-energy threshold $ \mathcal{ME}(u_0)=\mathcal{ME}(Q)$, where $Q$ is the ground state. Previously,
Duyckaerts-Merle studied the behavior of threshold solutions in the $H^1$-critical case, in dimensions $N = 3, 4, 5$, later generalized by Li-Zhang for higher dimensions. In the intercritical case, Duyckaerts-Roudenko studied the threshold problem for the 3d cubic NLS equation.


In this paper, we generalize the results of Duyckaerts-Roudenko for any dimension and any power of the nonlinearity for the  \textit{entire} intecritical range. We show the existence of special solutions, $Q^\pm$, besides the standing wave $e^{it}Q$, which exponentially approach the standing wave in the positive time direction, but differ in its behavior for negative time. We classify all solutions at the threshold level, showing either blow-up occurs in finite (positive and negative) time, or scattering in both time directions, or the solution is equal to one of the three special solutions above, up to symmetries. Our proof extends to the $H^1$-critical case, thus, giving a different and more unified approach than the Li-Zhang result.

These results are obtained by studying the linearized equation around the standing wave and some tailored approximate solutions to the NLS equation. We establish important decay properties of functions associated to the spectrum of the linearized Schrödinger operator,  which, in combination with modulational stability and coercivity for the linearized operator on special subspaces, allows us to use a fixed-point argument to show the existence of special solutions. Finally, we  prove the uniqueness by studying exponentially decaying solutions to a sequence of linearized equations.
\end{abstract}
\vspace{-.4cm}
\tableofcontents




\section{Introduction}

We consider the focusing energy-subcritical nonlinear Schrödinger (NLS) equation in $H^1(\Real^N)$, $N \geq 1$,

\begin{equation}\label{sub_NLS}
	\begin{cases}
	    i\partial_t u + \Delta u +|u|^{p-1}u = 0,\\
	    u(x,0) = u_0(x) \in H^1(\Real^N),
	\end{cases}
\end{equation}
where $1+\frac{4}{N} <p < 2^*-1 := \begin{cases}1+\frac{4}{N-2}, &N \geq 3,\\ +\infty, &N =1,2.\end{cases}$

When $N \geq 3$, we also consider the focusing energy-critical nonlinear Schrödinger equation in $\dot{H}^1(\Real^N)$ with nonlinearity power $p_c := \frac{N+2}{N-2}$
\begin{equation}\label{NLS}
	\begin{cases}
	    i\partial_t u + \Delta u +|u|^{p_c-1}u = 0,\\
	    u(x,0) = u_0(x) \in \dot{H}^1(\Real^N).
	\end{cases}
\end{equation}

The equation \eqref{sub_NLS} is considered in the inhomogeneous space $H^1(\Real^N)$ with the norm $\|f\|_{H^1} := \| f\|_{L^2}+\|\nabla f\|_{L^2}$, while \eqref{NLS} is studied in the homogeneous Sobolev space $\dot{H}^1(\Real^N)$ with the norm $\|f\|_{\dot{H}^1} := \|\nabla f\|_{L^2}$.

Note that \eqref{sub_NLS} and \eqref{NLS} are invariant under scaling. Namely, if $u$ is a solution, then 
\begin{equation}
    u_{\delta}(x,t) =  \delta^\frac{2}{p-1} u(\delta x, \delta^2 t)
\end{equation}
is also a solution to the same equation, for any $\delta > 0$. Computing the homogeneous $\dot{H}^s(\Real^N)$ norm yields
\begin{equation}
    \|u_{\delta}(\cdot,0)\|_{\dot{H}^s} = \delta^{s-\left(\frac{N}{2}-\frac{2}{p-1}\right)}\|u(\cdot,0)\|_{\dot{H}^s}.
\end{equation}
Hence, the scale-invariant norm is $\dot{H}^{s_c}(\Real^N)$, where
 $$
 s_c = \frac{N}{2}-\frac{2}{p-1}.
 $$
 The conditions on $p$ are equivalent to $0 < s_c < 1$ in \eqref{sub_NLS}, and to $s_c =1$ in \eqref{NLS}.

In addition to scaling invariance, the equations \eqref{sub_NLS} and \eqref{NLS} exhibit several symmetries, such as, space translation, time translation, phase rotation and time-reversal. Indeed, if $u(x,t)$ is such a solution, so is
$$
 e^{i\theta_0}u\left(x+x_0,t+t_0 \right) \text{ or }  e^{i\theta_0}\bar{u}\left(x+x_0,\-t+t_0 \right),
$$
with $(\theta_0, x_0,t_0) \in [0,2\pi)\times  \Real^N \times \Real$.

All these symmetries leave the $\dot{H}^{s_c}$ norm invariant. Another symmetry that does not have this characteristic is the \textit{Galilean boost}, given by
\begin{equation}
    e^{ix.\xi_0}e^{-it|\xi_0|^2}u(x-2\xi_0 t, t), \quad \xi_0 \in \mathbb{R}^N.
\end{equation}

Moreover, if $u_0\in \dot{H}^1(\Real^N)$, solutions to \eqref{sub_NLS} and \eqref{NLS} conserve the \textit{energy}
\begin{equation}
    E(u(t)) := \frac{1}{2}\int |\nabla u(t)|^2 - \frac{1}{p+1}\int |u(t)|^{p+1} = E(u_0),
\end{equation}
and whenever $u_0$ belongs to $L^2(\Real^N)$, the \textit{mass}
\begin{equation}
    M(u(t)) := \int |u(t)|^2 = M(u_0)
\end{equation}
and the \textit{momentum}
\begin{equation}
    P(u(t)) := \Im\int \bar{u}(t)\nabla u(t)  = P(u_0)
\end{equation}
are also conserved.

The Cauchy problem for \eqref{sub_NLS} was first studied by Ginibre and Velo \cite{GV79}. Namely, for initial data $u_0 \in H^1(\Real^N)$, there exists a non-empty maximal interval $I$ and a unique local-in-time solution $u : \Real^N \times I \to \mathbb{C}$ that belongs to $C^0_tH^1_x(\Real^N \times J)$ for every compact interval $J \subset I$. Moreover, the map $u_0 \mapsto u$ is uniformly continuous and $u$ satisfies the Duhamel formula
\begin{equation}
    u(t) = e^{it\Delta}u_0 + \int_0^t e^{i(t-s)\Delta}|u|^{p-1}u(s)\,ds
\end{equation}
for all $t \in I$. The solution is also known to be in $ L^q_t W_x^{1,r}(\Real^N \times J)$ for any Strichartz pair $(q,r)$ (see Section \ref{Sec2}).

In the energy-critical case, the Cauchy problem for \eqref{NLS} was first considered by Cazenave and Weissler \cite{CW90}. They proved that, for initial data $u_0 \in \dot{H}^1$, there exists a unique solution defined on a maximal non-empty interval $I$, satisfying the corresponding Duhamel formula and belonging to $C^0_t\dot{H}^1_x(\Real^N \times J) \cap L_{t,x}^{2(N+2)/(N-2)}(\Real^N \times J)$ for every compact interval $J \subset I$. Later works \cite{CKSTT08,RV07,TV05} proved that the map from the initial data to the solution is also uniformly continuous.

The impossibility to extend the solution to all times is related to the concept of \textit{finite-time blow-up}. We say that a solution to \eqref{sub_NLS} blows up in finite positive time $T >0$ if 
\begin{equation}
    \lim_{t \nearrow T} \|\nabla u(t)\|_{L^2} = +\infty.
\end{equation}

For the energy-critical case, given that the scale-invariant norm is $\dot{H}^1(\Real^N)$, this criterion is not enough to preclude the possibility of continuing the solution. Rather, we say that a solution to \eqref{NLS} blows up in finite positive time $T>0$ if 
\begin{equation}
    \int_0^T\int |u(x,t)|^{\frac{2(N+2)}{N-2}} \, dx dt= +\infty.
\end{equation}
In a similar way, blow-up in finite negative time is defined.


Solutions to \eqref{sub_NLS} and \eqref{NLS} can also exhibit a \textit{scattering} behavior in the energy space. We say that a solution to \eqref{sub_NLS} scatters forward in time if it is defined for any $t \in [0,+\infty)$ and there exists $\psi \in H^1(\Real^N)$ such that
\begin{equation}
    \lim_{t \to +\infty}\|u(t)-e^{it\Delta}\psi\|_{H^1} = 0.
\end{equation}
In the energy-critical case, the definition is the same, except that the $\dot{H}^1(\Real^N)$ norm is used instead. Scattering backward in time is defined analogously.

The $L_{t,x}^{2(N+2)/(N-2s_c)}$ norm also plays an important role in the scattering theory {(see, for instance, Cazenave \cite{cazenave}*{Chapter 7})}: solutions to either \eqref{sub_NLS} or \eqref{NLS} defined on the time interval $[0,+\infty)$ scatter forward in time if
\begin{equation}
    \int_0^{+\infty}\int |u(x,t)|^{\frac{2(N+2)}{N-2s_c}} \, dx dt < +\infty.
\end{equation}
By the time-reversal symmetry a similar scattering criterion backward in time can be obtained.

{Besides} finite-time blow-up and scattering, there is the concept of \textit{standing waves}. Consider the elliptic equation
\begin{equation}\label{sub_def_ground}
    \Delta \psi - (1-s_c)\psi +|\psi|^{p-1}\psi = 0.
\end{equation}
It is known that, for $s_c < 1$, this equation admits the unique radial, positive solution in $H^1(\Real^N)$, which we call the \textit{ground state} and denote by $Q = Q_{p,N}$ (see Strauss \cite{Strauss77}, Berestycki, Lions and Peletier \cite{BLP81}, Kwong \cite{Kwong89} and also Tao \cite[Appendix B]{TaoBook} for a textbook exposition). If $Q$ solves \eqref{sub_def_ground}, then the standing wave
\begin{equation}
    u(x,t) = e^{it}Q(x)
\end{equation}
is a solution to \eqref{sub_NLS} that neither blows up in finite time, nor scatters, in any time direction. On the other hand, if $s_c = 1$, since the equation \eqref{sub_def_ground} is invariant by scaling, the radial positive solution to \eqref{sub_def_ground} is not unique. In this case, an explicit solution is given by

\begin{equation}
    Q_{\frac{2N}{N-2},N}(x) := \frac{1}{\left(1+\frac{|x|^2}{N(N-2)}\right)^\frac{N-2}{2}}.
\end{equation}
This solution is commonly denoted by $W$, and we shall often do so.

A simple calculation shows that $W \in \dot{H}^1(\Real^N)$ for any $N \geq 3$, and that $W \in L^2(\Real^N)$ if and only if $N \geq 5$. As its subcritical counterpart, the static solution $u(x,t) = W(x)$ to \eqref{NLS} neither blows up in finite time, nor scatters, in any time direction.

Also, the following {\textit{Pohozaev}} identities follow from \eqref{sub_def_ground}
\begin{equation}\label{pohozaev}
\begin{gathered}
     \int |Q|^{p+1} = \frac{2(p+1)}{N(p-1)}\int |\nabla Q|^{2},\\
    \int |Q|^2 = \frac{(N-2)(p+1)-4}{(1-s_c)N(p-1)} \int |\nabla Q|^2, \text{ if } 0 < s_c < 1.
\end{gathered}    
\end{equation}
\begin{re}
The choice of the constant $(1 - s_c)$ in \eqref{sub_def_ground} is only for
convenience. If $s_ c < 1$ we can modify $Q$ and replace $(1-s_c)$ by any positive constant by scaling. Similarly, if $s_c = 1$, the choice of $Q_{\frac{2N}{N-2},N} = W$ is arbitrary, and we could have used any rescaled version of $W$. Since we will state our results up to a constant scaling (among other symmetries), there is no loss of generality.
\end{re}

The works of Weinstein \cite{W_Nonl} in the case $0<s_c < 1$ and of Aubin \cite{Aub76} and Talenti \cite{Tal76} for $s_c=1$ give the characterization of the ground state as the minimizer of 
\begin{equation}\label{gagl_gen}
    \|f\|_{L^{p+1}}^{p+1} \leq C_{N,p} \|\nabla f\|_{L^2}^{\frac{N(p-1)}{2}}\|f\|_{L^2}^{2-\frac{(N-2)(p-1)}{2}},
\end{equation}
with equality if and only if $f(x) = e^{i\theta_0}Q(x+x_0)$, if $0 < s_c < 1$, or $f(x) = e^{i\theta_0}\lambda_0^{\frac{N-2}{2}}W(\lambda_0 x+x_0)$, if $s_c = 1$, for some $\theta_0 \in [0,2\pi)$, $x_0 \in \Real^N$ and $\lambda_0>0$. Here, $C_{N,p}$ is the sharp constant of the inequality \eqref{gagl_gen}. In the subcritical case, the inequality \eqref{gagl_gen} is known as (one version of) the Gagliardo-Nirenberg inequality, and in the energy-critical case, it reduces to the classical Sobolev inequality.

The ground state is also associated with the \textit{threshold} for a dichotomy between finite-time blow-up and scattering. The behavior of solutions below the ground state level are now well understood for both focusing energy-critical and energy-subcritical nonlinear Schrödinger equation. Indeed, for the Cauchy problem \eqref{NLS}, solutions with $E(u_0) <E(W)$ were considered by Kenig and Merle \cite{KM_Glob} in the radial setting for $N = 3$, $4$ and $5$, introducing the concentration-compactness and rigidity approach for dispersive models. They showed that if $\|\nabla u_0\|_{L^2} <\|\nabla W\|_{L^2}$, then the corresponding solution exists globally in time and scatters in both time directions. On the other hand, if $\|\nabla u_0\|_{L^2} >\|\nabla W\|_{L^2}$, then the corresponding solution blows up in finite positive and negative times (provided it's radial or of finite variance). Later, Killip and Visan \cite{KV10} extended this result for $N \geq 5$, removing the radial assumption.
 For the Cauchy problem \eqref{sub_NLS}, this problem was studied by Holmer and Roudenko \cite{HR_Scat} in the 3d cubic radial case. In \cite{HR_Blow} for $0<s_c<1$, they consider the following scale-invariant quantities
\begin{equation}
    \mathcal{ME}(u(t)) =
    \frac{
    M(u(t))^{\frac{1-s_c}{s_c}}E(u(t))}{M(Q)^{\frac{1-s_c}{s_c}}E(Q)
    }
     = \mathcal{ME}(u_0),
\end{equation}
and 
\begin{equation}
    {\mathcal{MG}(u(t))} =\frac{ \|u_0\|_{L^2}^{\frac{1-s_c}{s_c}}   \|\nabla u(t)\|_{L^2}}
    {\|Q\|_{L^2}^{\frac{1-s_c}{s_c}}   \|\nabla Q\|_{L^2}}.
\end{equation}
Based on the concentration-compactness and rigidity approach, in \cite{HR_Blow} they proved that if $\mathcal{ME}(u_0)<1$ and $\mathcal{MG}(u_0)<1$, then the corresponding solution exists globally in time and scatters in both time directions. In \cite{HR_Blow} they also proved that if $\mathcal{ME}(u_0)<1$, $\mathcal{MG}(u_0)>1$ and either $u_0$ is radial or $|x|u_0 \in L^2(\Real^N)$, then the corresponding solution blows up in both finite positive and negative times\footnote{If $u_0$ is nonradial and has infinite variance, then there exists a sequence of times $\{t_n\}$ such that $\|\nabla u(t_n)\|_{L^2} \to +\infty$, as shown by Holmer and Roudenko in \cite{HR_Div}.}, establishing the dichotomy result in the 3d cubic radial case. Later Duyckaerts, Holmer and Roudenko \cite{HR_Scat} removed the radial assumption in the scattering result. Fang, Xie and Cazenave \cite{FXC_Scat} and Guevara \cite{Guevara} (see also Guevara and Carreon \cite{GC12}) extended this result to all intercritical ranges and dimensions.

This dichotomy does not hold above the ground state mass-energy threshold. In \cite{HPR}, Holmer, Platte and Roudenko proved blow-up criteria that included solutions above the mass-energy threshold. In \cite{DR_Going}, Duyckaerts and Roudenko showed, for $0 < s_c \leq 1$, the existence of asymmetric behavior in time of solutions to the NLS equation that are above the threshold and that scatter in one time direction and blow-up in finite time in the other time direction (in fact, they showed that it suffices to multiply the ground state by a quadratic phase to produce such a result). Moreover, they proved a dichotomy-type result above the mass-energy threshold with some conditions on the variance of the initial data.

At the threshold level, there exists a richer dynamics for the asymptotic  behavior of solutions. Indeed, for the focusing energy-critical NLS equation, this problem was first considered by Duyckaerts and Merle \cite{DM_Dyn}, in the radial case for $N = 3$, $4$ and $5$. In particular, they proved the existence of special solutions that approach $W$ as $t\to +\infty$ and either blow-up or scatter as $t \to -\infty$. Later Li and Zhang \cite{higher_thre} studied the case $N \geq 6$. For the focusing energy-subcritical nonlinear Schrödinger equation, Duyckaerts and Roudenko \cite{DR_Thre} treated the 3d cubic case. The main goal of this paper is to generalize the results in \cite{DR_Thre} to the entire intercritical range $0<s_c<1$. More precisely, we prove the following.

\begin{teo}[Energy-subcritical case]\label{sub_special}
For $N \geq 1$, there exist two radial solutions $Q^+$ and $Q^-$ to \eqref{sub_NLS} in $H^1(\Real^N)$ such that
\begin{itemize}
\item $M[Q^\pm] = M[Q]$, $E[Q^\pm] = E[Q]$ , $Q^\pm$ is defined at least on $[0,+\infty)$ and there exist $C$, $e_0>0$ such that
\begin{equation}
\|Q^\pm(t) - e^{it}Q\|_{H^1} \leq C e^{-e_0t}\text{ for all } t \geq 0,
\end{equation} 
\item $\|\nabla Q^+_0\|_2 > \|\nabla Q\|_2$ and $Q^+$ blows-up in finite negative time,
\item $\|\nabla Q^-_0\|_2 < \|\nabla Q\|_2$ and $Q^-$ is globally defined and scatters backward in time.
\end{itemize}
\end{teo}

\begin{teo}[Energy-subcritical case]\label{sub_class_thresh}
For $N \geq 1$, let $u$ be a solution to \eqref{sub_NLS} such that $\mathcal{ME}(u_0) = 1$. Then, the following holds.
\begin{itemize}
    \item If $\mathcal{MG}(u_0) < 1$, then $u$ is defined for all times. Moreover, either $u$ scatters in both time directions, or $u = Q^-$  up to the symmetries of the equation.

    \item If $\mathcal{MG}(u_0) = 1$, then $u = Q$ up to the symmetries of the equation.
    
    \item If $\mathcal{MG}(u_0) > 1$ and $u_0$ is either radial or $|x|u_0 \in L^2(\Real^N)$, then either $u$ blows up in finite positive and negative time, or $u = Q^+$ up to the symmetries of the equation.
\end{itemize}
\end{teo}

There are two major difficulties in extending the previous results. The first one is to deal with low powers of the parameter $p$. If $p <3 $, then the nonlinearity $|u|^{p-1}u$ is not a smooth function of $(u,\bar{u})$. Moreover, as the power of the nonlinearity is not an odd integer, the difference $|u|^{p-1}u - |v|^{p-1}v$ cannot be written as a polynomial. Therefore, we cannot use the same estimates as in \cite{DR_Thre}, as they rely heavily on $H^s(\Real^N)$ estimates, for large values\footnotemark\footnotetext{To be precise, at least $s > N/2$, to make use of the fact that $H^s(\Real^N)$ is an algebra.} of $s$. Moreover, if $p \leq 2$, then the nonlinearity is not twice real-differentiable.  In order to perform the necessary estimates, we employ the fractional calculus tools introduced by Christ and Weinstein \cite{CW91} and Visan \cite{Visan07}.

Another problem arises from the fast decay of the ground state $Q$. When constructing the solutions $Q^\pm$, we must deal with some estimates that involve terms of the form $\|Q^{-1}f\|_{L^\infty}$. Even though $(Q^{-1}f)(x)$ is pontwise defined for any function $f$, the exponential decay of $Q$ makes it harder to obtain good estimates. Therefore, we have to carefully study the desired functions $f$ to ensure that they have the necessary decay. We establish the decay via several bootstrap arguments, and by making use of resolvent convolution kernels associated to the corresponding elliptic equations.

It is worth mentioning that, in order to prove the classification in Theorems \ref{sub_special} and \ref{sub_class_thresh}, one has to change the orthogonality conditions that were used in \cite{DR_Thre}, as in the lower dimensions they would not necessarily ensure coercivity. See Remark \ref{re_ortho} and the proof of Lemma \ref{lem_coerc} for details.

Since our proof can be readily applied to the energy-critical case, we also state and prove similar results for $s_c = 1$ and $N \geq 6$.  Note that $p_c \leq 2$ happens exactly when $N \geq 6$, which again implies that the nonlinearity is not twice real-differentiable. We have the following.

\begin{teo}[Energy-critical case]\label{special}
Let $N \geq 6$. There exist two radial solutions $W^+$ and $W^-$ to \eqref{NLS} in $\dot{H}^1(\Real^N)$ such that
\begin{itemize}
\item $E[W^\pm] = E[W]$ , $W^\pm$ is defined at least in $[0,+\infty)$ and there exist  $C$, $e_0>0$ such that
\begin{equation}
\|W^\pm(t) - W\|_{H^1} \leq C e^{-e_0t} \text{ for all } t \geq 0,
\end{equation} 
\item $\|\nabla W^+_0\|_2 > \|\nabla W\|_2$ and $W^+$ blows-up in finite negative time,
\item $\|\nabla W^-_0\|_2 < \|\nabla W\|_2$ and $W^-$ is globally defined and scatters backward in time.
\end{itemize}
\end{teo}

\begin{teo}[Energy-critical case]\label{class_thresh}
For $N \geq 6$, let $u$ be a radial solution to \eqref{NLS} such that $E(u_0) = E(W)$. Then, the following holds.
\begin{itemize}
    \item If $\|u_0\|_{\dot{H}^1} < \|W\|_{\dot{H}^1}$, then $u$ is defined for all times. Moreover, either $u$ scatters in both time directions, or $u = W^-$ up to the symmetries of the equation.
    \item If $\|u_0\|_{\dot{H}^1} = \|W\|_{\dot{H}^1}$, then $u = W$ up to the symmetries of the equation.
    \item If $\|u_0\|_{\dot{H}^1} > \|W\|_{\dot{H}^1}$, and $u_0 \in L^2$, then either $u$ blows-up in finite positive and negative time, or $u = W^+$ up to the symmetries of the equation.
\end{itemize}
\end{teo}

The last two theorems were originally proved by Li and Zhang \cite{higher_thre}, by means of weighted Sobolev estimates. Since our approach is considerably different from \cite{higher_thre}, we include our proofs in this paper.

\begin{re}\label{sub_scaling} By scaling, the condition $\mathcal{ME}(u_0) = 1$  can be read, without loss of generality, as 
\begin{equation}\label{sub_scaling_me}
    \begin{cases}
    M(u_0) = M(Q)\\
    E(u_0) = E(Q).
    \end{cases}
\end{equation}
Indeed, considering $u_{0,\delta}(x) = \delta^{\frac{2}{p-1}}u_0(\delta x)$, with $\delta = (M(u_0)/M(Q))^{\frac{1}{2s_c}}$, gives the above condition for $u_{0,\delta}$. Similarly, the condition $$\mathcal{MG}(u_0) < 1$$ (resp. ``$=$'', ``$>$'') can be read as $$\|\nabla u_0\|_{L^2} < \| \nabla Q\|_{L^2}$$ (resp. ``$=$'', ``$>$''). Unless stated otherwise, we shall adopt this simplification throughout the whole paper.
\end{re}

\smallskip

{\bf Acknowledgements.} This project was mostly done while L.C. was visiting the Department of Mathematics and Statistics at Florida International University, Miami, FL during his third year of PhD training. He thanks the department and the university for the hospitality and support. He would like to thank Thomas Duyckaerts and Nicola Visciglia for their valuable comments and suggestions, which helped improve the manuscript. L. C. was financed in part by the Coordena\c{c}\~ao de Aperfei\c{c}oamento de Pessoal de N\'ivel Superior - Brasil (CAPES) - Finance Code 001. S.R. was partially supported by the NSF DMS grant 1927258, and part of research travel of L.C. was funded by the same grant DMS-1927258 (PI:Roudenko).  L.G.F. was partially supported by Coordena\c{c}\~ao de Aperfei\c{c}oamento de Pessoal de N\'ivel Superior - Brasil (CAPES) - Finance Code 001, Conselho Nacional de Desenvolvimento Cient\'ifico e Tecnol\'ogico - CNPq and Funda\c{c}\~ao de Amparo a Pesquisa do Estado de Minas Gerais - Fapemig/Brazil.

\section{Notation and preliminaries}\label{Sec2}
We will need the following tools from Harmonic Analysis.
\begin{lemma}[Sobolev inequality, see Stein \cites{steinbook}] If $0 < \rho - \sigma < N$, $1 < q < p < \infty$, and
\begin{equation}
	\frac{1}{p} = \frac{1}{q}-\frac{\rho - \sigma}{N},
\end{equation}
then the following estimate holds
\begin{align}
	\label{sobolev}\|D^\sigma u\|_{L^p(\Real^N)} &\leq C \|D^\rho u\|_{L^q(\Real^N)}.\\
\end{align}
\end{lemma}
\begin{lemma}[Leibniz rule, see Christ and Weinstein \cite{CW91}] Let $s \in (0,1)$, $p_j, q_j \in (1,\infty)$, with $\frac{1}{p} = \frac{1}{p_j}+\frac{1}{q_j}$, $j = 1,2$. Then
\begin{equation}\label{leibiniz}
	\|D^s (fg)\|_{L^p(\Real^N)} \leq C\left(\|D^s f\|_{L^{p_1^{}}(\Real^N)} \|g\|_{L^{q_1^{}}(\Real^N)} + \|f\|_{L^{p_2^{}}(\Real^N)} \|D^s g\|_{L^{q_2^{}}(\Real^N)}\right).
\end{equation}
\end{lemma}

\begin{lemma}[Fractional chain rule for Hölder continuous functions, see Visan \cite{Visan07}] Let $F$ be a Hölder continuous function of order $0 < \alpha <1$. Then, for every $0 < s < \alpha$, $1 < p < \infty$, and $\frac{s}{\alpha} < \nu < 1$ we have
\begin{equation}\label{frac_chain}
	\|D^s F(u)\|_{L^{p^{}}(\Real^N)} \leq C \| |u|^{\alpha-\frac{s}{\nu}}\|_{L^{p_1^{}}(\Real^N)} \|D^\nu u\|_{L^{\frac{s}{\nu}q_1^{}}(\Real^N)}^\frac{s}{\nu},
\end{equation}
provided $\frac{1}{p} = \frac{1}{p_1}+\frac{1}{q_1}$ and $\left(1-\frac{s}{\nu \alpha}\right)p_1 > 1$.
\end{lemma}

\begin{lemma}[Kato-Strichartz inequalities, see Cazenave \cite{cazenave}, Kato \cite{Kato94}, Foschi \cite{Foschi05}, Keel and Tao \cites{KT98}]\label{KS_est} Let $N \geq 1$, $I \subset \Real$ and $1 \leq q_i,r_i\leq \infty$, $i = 1,2$. If the pairs $(q_1,r_1)$ and $(q_2,r_2)$ satisfy
\begin{equation}
    \frac{1}{q_i} < N\left(\frac{1}{2}-\frac{1}{r_i} \right)\text{ or } (q_i,r_i) = (\infty,2), \quad i = 1,2,
\end{equation}
\begin{equation}
	\frac{1}{q_1}+\frac{1}{q_2} = N\left(1-\frac{1}{r_1}-\frac{1}{r_2}\right),
\end{equation}
and:
\begin{itemize}
    \item If $N = 2$, we require that $r_1, r_2 < +\infty$,
    \item If $N > 2$, we consider two subcases
    \begin{itemize}
        \item \textit{non sharp case:}
            \noindent\begin{align}
    	        \frac{1}{q_1}&+\frac{1}{q_2} < 1,\\
    	        \frac{N-2}{r_1} \leq \frac{N}{r_2}, &\quad \frac{N-2}{r_1} \leq \frac{N}{r_2},	
            \end{align}
        \item \textit{sharp case:}
        \noindent\begin{align}
    	        \frac{1}{q_1}&+\frac{1}{q_2} = 1,\\
    	        \frac{N-2}{r_1} < \frac{N}{r_2}, &\quad \frac{N-2}{r_1} < \frac{N}{r_2},\\
    	        \frac{1}{r_1}\leq \frac{1}{q_1}, &\quad \frac{1}{r_2}\leq \frac{1}{q_2}.
            \end{align}
    \end{itemize}
\end{itemize}

Then the following estimate holds
\begin{equation}\label{intro_in_stric}
	\left\|\int_{s>t} e^{i(t-s)\Delta}F(s)\, ds \right\|_{L_I^{q_1^{}}L_x^{r_1^{}}} + \left\|\int_{s<t} e^{i(t-s)\Delta}F(s)\, ds \right\|_{L_I^{q_1^{}}L_x^{r_1^{}}} \lesssim \left\|F\right\|_{L_I^{q_2'}L_x^{r_2'}}.
\end{equation}

\end{lemma}

\begin{de} We say that the pair $(q,r)$ is $\dot{H}^{s}$-admissible if $2 \leq q,r \leq +\infty$, $(q,r,N) \neq (2,\infty,2)$, and
\begin{equation}
    \frac{2}{q}+\frac{N}{r} = \frac{N}{2}-s.
\end{equation}

If $s = 0$, we say that the pair $(q,r)$ is $L^2$-admissible.
\end{de}
We define Strichartz norms for the energy-critical and intercritical cases separately.
\begin{de}[Critical case]\label{stric_sp}

Let $I$ be a (possibly unbounded) time interval. Given $0 < \varepsilon \ll \frac{4}{N-2}$, $N \geq 6$, define the spaces
\begin{align}
	S(\dot{H}^1,I) &= L_I^\infty L_x^\frac{2N}{N-2}
	\\
	S(\dot{H}^{1-\varepsilon},I) &= L_I^\infty L_x^\frac{2N}{N-2+2\varepsilon} \cap L_I^{\frac{4}{\varepsilon}} L_x^\frac{2N}{N-2+\varepsilon} \cap L_I^{\frac{2(N-2)}{\varepsilon(N-4)}} L_x^\frac{2N(N-2)}{(N-2)^2+4\varepsilon},\\
	S'(\dot{H}^{-(1-\varepsilon)},I) &= L_I^{\frac{2}{\varepsilon}} L_x^\frac{2N}{N+2},\\
	S(L^2,I) &= \left\{L^q_IL^r_x \left| (q,r) \text{ is } L^2-\text{admissible} \right.\right\},\\
	S'(L^2,I) &= L_I^{2} L_x^\frac{2N}{N+2}.\\
\end{align}
\end{de}
\begin{re}
In particular, we make use of the following spaces in $S(L^2,I)$: $L_I^\infty L_x^2$, $ L_I^{\frac{4}{\varepsilon}} L_x^\frac{2N}{N-\varepsilon}$, $L_I^{\frac{2(N-2)}{\varepsilon(N-4)}} L_x^\frac{2N(N-2)}{N(N-2)-2\varepsilon(N-4)}$, $L^2_IL^\frac{2N}{N-2}_x$, $L^{\frac{2(N+2)}{N-2}}_IL^{\frac{2(N+2)}{N^2+4}}_x$,   and $L^\frac{16}{\varepsilon(N-2)}_I L^\frac{8N}{4N-\varepsilon(N-2)}_x$.
\end{re}

\begin{re}
By Sobolev embedding, if $f \in S(\dot{H}^1,I) \cap \nabla^{-1} S(L^2,I)$,
\begin{equation}
	\|f\|_{S(\dot{H}^1,I)}+  \|D^\varepsilon f\|_{S(\dot{H}^{1-\varepsilon},I)} \lesssim \|\nabla f\|_{S(L^2,I)},
\end{equation}
and by Kato-Strichartz estimates, non sharp case,
\begin{align}
	\left\|\int_{s>t} e^{i(t-s)\Delta}F(s)\, ds \right\|_{S(L^2,I)} \lesssim \left\|F\right\|_{S'(L^2,I)},\\
	\left\|\int_{s>t} e^{i(t-s)\Delta}F(s)\, ds \right\|_{S(\dot{H}^{1-\varepsilon},I)} \lesssim \left\|F\right\|_{S'(\dot{H}^{-(1-\varepsilon)},I)}.
\end{align}
\end{re}
\begin{re}
Note that the pair in $S(\dot{H}^1,I)$ is $\dot{H}^1$-admissible, the pairs in $S(\dot{H}^{1-\varepsilon},I)$ are $\dot{H}^{1-\varepsilon}$-admissible, the pairs in $S(L^2,I)$ and in the dual space of $S'(L^2,I)$ are $L^2$-admissible, and the pair corresponding to the dual space of $S'(\dot{H}^{-(1-\varepsilon)},I)$ is $\dot{H}^{-(1-\varepsilon)}$-admissible.
\end{re}

\begin{de}[Intercritical case]

Define the set

\begin{equation}
    \mathcal{A}_0 = \left\{(q,r) |(q,r) \text{ is } L^2\text{-admissible}
    \right\}.
\end{equation}

For $s \in (0,1)$, define $\mathcal{A}_s$ as the $\dot{H}^s$-admissible pairs that satisfy
\begin{equation}
    \left\{\begin{array}{ccccccc} 
    \frac{2N}{N-2s} &\leq &r &\leq &\left(\frac{2N}{N-2}\right)^-&, & N \geq 3,\\
    \frac{2}{1-s} &\leq &r &\leq &\left(\left(\frac{2}{1-s}\right)^+\right)'&, & N = 2,\\ 
    \frac{2}{1-2s} &\leq &r &\leq &\infty&, &N = 1,
    \end{array}\right.   
\end{equation}
and $\mathcal{A}_{-s}$ as the $\dot{H}^{-s}$-admissible pairs that satisfy
\begin{equation}
    \left\{\begin{array}{ccccccc} 
    \left(\frac{2N}{N-2s}\right)^+ &\leq &r &\leq &\left(\frac{2N}{N-2}\right)^-&, & N \geq 3,\\
    \left(\frac{2}{1-s}\right)^+ &\leq &r &\leq &\left(\left(\frac{2}{1-s}\right)^+\right)'&, & N = 2,\\ 
    \left(\frac{2}{1-2s}\right)^+ &\leq &r &\leq &\infty&, &N = 1,
    \end{array}\right.   
\end{equation}

where $(a^+)'$ is the number such that 
\begin{equation}
    \frac{1}{a} = \frac{1}{a^+} + \frac{1}{(a^+)'}.
\end{equation}
Let $I$ be a (possibly unbounded) time interval. For $s \in [0,1)$, we define the following Strichartz norms
\begin{equation}
	\|u\|_{S(L^2,I) } = \sup_{(q,r)\in \mathcal{A}_0}\|u\|_{L_I^qL_x^r},	
\end{equation}
\begin{equation}
	\|u\|_{S(\dot{H}^{s},I) } = \sup_{(q,r)\in \mathcal{A}_{s}}\|u\|_{L_I^qL_x^r},	
\end{equation}
and the dual Strichartz norms
\begin{equation}
	\|u\|_{S'(L^2,I) } = \inf_{(q,r)\in \mathcal{A}_{0}}\|u\|_{L_I^{q'}L_x^{r'}},	
\end{equation}
\begin{equation}
	\|u\|_{S'(\dot{H}^{-s},I) } = \inf_{(q,r)\in \mathcal{A}_{-s}}\|u\|_{L_I^{q'}L_x^{r'}}.	
\end{equation}

\end{de}
\begin{re}
By Sobolev embedding, if $f \in S(\dot{H}^{s_c},\,I) \cap \langle\nabla\rangle^{-1} S(L^2,\,I)$,
\begin{equation}
	\|f\|_{S(\dot{H}^{s_c},\,I)} \lesssim \||\nabla|^{s_c} f\|_{S(L^2,\,I)} \lesssim \|\langle\nabla\rangle f\|_{S(L^2,\,I)},
\end{equation}
and by Kato-Strichartz estimates, non sharp case,
\begin{align}
	\left\|\int_{s>t} e^{i(t-s)\Delta}F(s)\, ds \right\|_{S(L^2,\,I)} \lesssim \left\|F\right\|_{S'(L^2,\,I)},\\
	\left\|\int_{s>t} e^{i(t-s)\Delta}F(s)\, ds \right\|_{S(\dot{H}^{s_c},\,I)} \lesssim \left\|F\right\|_{S'(\dot{H}^{-s_c},\,I)}.
\end{align}
\end{re}

\section{The linearized equation}

In order to prove the main theorems of this paper, we need to carefully study \eqref{sub_NLS} and \eqref{NLS} around the ground state.
We identify the complex number $a + bi$ with the vector $\begin{pmatrix}
a\\b
\end{pmatrix}$. For a complex-valued function $f$, we write $f = f_1+if_2$.
We next introduce the following definition. 

\begin{de} For $0 < s_c \leq 1$, we define
\begin{align}\label{def_op}
L_+ &=(1-s_c)-\Delta-pQ^{p-1},\\
L_-&=(1-s_c)-\Delta - Q^{p-1},\\
\mathcal{L} &:= \begin{pmatrix}0 & -L_- \\
L_+ & 0 \\
\end{pmatrix},\\
R(f) &= |Q+f|^{p-1}(Q+f)-Q^{p}-pQ^{p-1}f_1-iQ^{p-1}f_2,\\
K(f) &= p Q^{p-1}f_1+iQ^{p-1}f_2.\\
\end{align} 

\end{de}

If $u$ is a solution to \eqref{sub_NLS}, write $u = e^{i(1-s_c)t}(Q+v)$. Then $v$ must satisfy
\begin{equation}\label{linearized_eq}
	\partial_t v + \mathcal{L}v = iR(v),
\end{equation}
or, writing it as a Schrödinger equation,
\begin{equation}\label{linearized_eq_2}
	i\partial_t v + \Delta v -(1-s_c)v +K(v) = -R(v). 
\end{equation}

In the next two sections we recall some properties of the operator $\mathcal{L}$.
\subsection{The linearized operator}
For $0 < s_c \leq 1$, we have, by a direct calculation,
{
\begin{align}
    L_-(Q) &= 0,\\
    L_+(\partial_kQ)&= 0,\quad  1 \leq k \leq N.
\end{align}
This implies
\begin{equation}
 \mathcal{L}(\partial_kQ)=\mathcal{L}(iQ) = 0,\quad  1 \leq k \leq N.
\end{equation}

Also, defining $\Lambda f$ as the scaling generator $\frac{2}{p-1}f +x \cdot \nabla f$, we have
\begin{equation}
    L_-(\Lambda Q) = -\mathcal{L}(\Lambda Q) = 2(1-s_c)Q, \,\,\, \mbox{if}\,\,\, 0 < s_c <1, \,\,\, \mbox{and}\,\,\,  L_-(\Lambda W) = 0, \,\,\, \mbox{if}\,\,\, s_c = 1.
\end{equation}
}
The above directions are obtained from $Q$ by the symmetries of the NLS equation. Indeed, defining 

\begin{equation}
	f_{[x_0,\lambda_0,\theta_0]}(x) = e^{i\theta_0}\frac{1}{\lambda_0^\frac{2}{p-1}}f\left(\frac{x}{\lambda_0}+x_0\right),
\end{equation}
we have
\begin{equation}
(\nabla Q, \Lambda Q, iQ) = \frac{\partial Q_{[x_0,\lambda_0,\theta_0]}}{\partial(x_0,\lambda_0,\theta_0)}\Bigg|_{(x_0,\lambda_0,\theta_0) = (0,1,0)}.
\end{equation}
The following result is well-known, but for completeness we provide its proof here.
\begin{lemma}[see \cite{DM_Dyn},\cite{CGNT07}]\label{spectrum_L}Let $\sigma(\mathcal{L})$ be the spectrum of the operator $\mathcal{L}$, defined in $L^2(\Real^N)\times L^2(\Real^N)$ with domain $H^2(\Real^N)\times H^2(\Real^N)$ and let $\sigma_{ess}(\mathcal{L})$ be its essential spectrum. Then
\begin{equation}
	\sigma_{ess}(\mathcal{L}) = \{i y:\,\, y \in \Real, |y| \geq 1-s_c\},\quad \sigma \cap \Real = \{-e_0,0,e_0\} \quad \text{with } e_0 >0. 
\end{equation}
Moreover, $e_0$ and $-e_0$ are simple eigenvalues  of $\mathcal{L}$ with eigenfunctions $\mathcal{Y}_+$ and $\mathcal{Y}_- = \overline{\mathcal{Y}}_+ \in \mathcal{S}$, respectively. The null space of $\mathcal{L}$ is spanned by $iQ$ and $\partial_k Q$, $1 \leq k \leq N$ (and, in the energy-critical case, also by $\Lambda W$).
\begin{re}\label{re_spectrum}
By Lemma \ref{spectrum_L}, if $\mathcal{Y}_1 = \Re(\mathcal{Y}_+)$ and $\mathcal{Y}_2 = \Im(\mathcal{Y}_+)$, then
\begin{equation}
	L_+ \mathcal{Y}_1 = e_0 \mathcal{Y}_2 \quad \text{and } L_-\mathcal{Y}_2 = -e_0 \mathcal{Y}_1.
\end{equation}
Furthermore, the null space of $L_+$ is spanned by the vectors $\partial_k Q$, $1 \leq k \leq N$, (and by $\Lambda W$, if $s_c = 1$) and the null space of $L_-$ is spanned by $Q$.
\end{re}
\end{lemma}

\begin{proof}[Proof of Lemma \ref{spectrum_L}] In this proof, for brevity, we write $V = Q^{p-1}$ for $0 < s_c \leq 1$. Note that $V$ defines a compact operator from $H^1$ to $L^2$.

\textit{Intercritical case.} For convenience, from now on we rescale $Q$ in the intercritical case as to solve
\begin{equation}\label{intercrit_Q}
    \Delta Q - Q + Q^p = 0.
\end{equation}

This is in order to simplify the exposition, avoiding unnecessary parameters in the calculations. The term $(1-s_c)$ must be replaced by $1$ in the definition of $\mathcal{L}$ and in the standing wave solution $e^{i(1-s_c)t}Q$ as well.

The operator $\mathcal{L}$ is a relatively compact perturbation of $i(1-\Delta)$, and therefore, has the same essential spectrum.
We now prove the existence of exactly one negative eigenvalue to $\mathcal{L}$.

From the proof of Lemma \ref{lem_coerc}, we see that $L_-$ on $L^2$ with domain $H^2$ is non-negative. Since it is also self-adjoint, it has a unique square root $L_-^{\frac{1}{2}}$ with domain $H^1$. It is equivalent to show that the the self-adjoint operator $P:= L_-^{\frac{1}{2}}L_+L_-^{\frac{1}{2}}$ on $L^2$ with domain $H^4$ has a unique negative eigenvalue. Indeed, consider the function
\begin{equation}
    Z = \Lambda Q-\frac{(\Lambda Q,Q)_{L^2}}{(Q,Q)_{L^2}} Q.
\end{equation}
One can check that $Z \in H^2$, $Z \in \{Q\}^\perp$ and, for $0 < s_c \leq 1$,
\begin{equation}\label{Z_positive_dir} 
    (L_+Z, Z)_{L^2} = -\frac{N^2(p-1)}{4(p+1)}\left[p-\left(1+\frac{4}{N}\right)\right]\int Q^{p+1}<0.
\end{equation}
Defining $ h:= L_-^{-\frac{1}{2}}Z \in {Q}^{\perp}$, one also has
\begin{equation}
    h = (L_-^{\frac{1}{2}}L_-^{-1})  (L_-^{-1} L_-) Z = L_-^{-1}L_-^{-\frac{1}{2}} L_- Z \in H^3.
\end{equation}

For $\varepsilon >0$, choose $\tilde{h}_\varepsilon \in H^4$ such that $\tilde{h}_\varepsilon \perp Q$ and $\|h-\tilde{h}_\varepsilon\|_{H^3} < \varepsilon$. We have
\begin{equation}
    \inf_{f \in H^4} \frac{(Pf,f)_{L^2}}{\|f\|_{L^2}^2} \leq \frac{(L_+L_-^{\frac{1}{2}}\tilde{h}_\varepsilon, L_-^{\frac{1}{2}}\tilde{h}_\varepsilon)_{L^2}}{\|\tilde{h}_\varepsilon\|_{L^2}^2}  
    < 0, 
\end{equation}
if $\varepsilon$ is small enough.

Hence, by the minimax principle, $P$ has a negative eigenvalue $-e_0^2$ and an associated eigenfunction $g$. Defining $\mathcal{Y}_1 := L_-^{\frac{1}{2}}g$, $\mathcal{Y}_2 := \frac{1}{e_0} L_+ \mathcal{Y}_1$, and $\mathcal{Y}_{\pm}:= \mathcal{Y}_{1} \pm i \mathcal{Y}_{2}$, we have $\mathcal{L}\mathcal{Y}_{\pm}=\pm e_0 \mathcal{Y}_{\pm}$.
Uniqueness of the negative eigenfunction of $P$ follows from the non-negativity of $L_+$ on $\{Q^p\}^\perp$. The assertions about the kernel of $\mathcal{L}$ follow from the coercivity given by Lemma \ref{lem_coerc}.

It remains to prove that $\mathcal{Y}_{\pm} \in S(\Real^N)$. It suffices to prove this assertion for $\mathcal{Y}_1 = \Re \mathcal{Y}_+$. The differential equation for $\mathcal{Y}_1$ is 
\begin{equation}\label{diff_eq_y1}
    [(1-\Delta)^2+e_0^2]\mathcal{Y}_1 = [pV^2+V(1-\Delta)]\mathcal{Y}_1-p(1-\Delta)[V\mathcal{Y}_1].
\end{equation}

Since the Fourier symbol of $(1-\Delta)^2+e_0^2$ is $(1+|\xi|^2)^2+e_0^2 \approx (1+|\xi|^2)^2$, and $V,\mathcal{Y}_1 \in H^2(\Real^N)$, we have that $\mathcal{Y}_1 \in H^{s}$ for all $s \geq 0$.
As in \cite{DM_Dyn}, we show that for all non-negative integers $k, s$ and all $\varphi \in C^{\infty}_c(\Real^N)$, we have
\begin{equation}\label{schw_ind}
    \|\varphi(x/R)\mathcal{Y}_1\|_{H^s} \leq \frac{C(\varphi,s,k) }{R^k}\,\,\, \text{for all} \,\,\, R \geq 1.
\end{equation}
Indeed, the inequality \eqref{schw_ind} holds if $k = 0$, for any $s \geq 0$. By induction, we show that if it holds for $(k,s)$, it also holds for $(k+1,s+1)$. Given $\varphi$, consider $\tilde{\varphi} \in C^{\infty}_c(\Real^N)$ such that $\tilde{\varphi}$ is $1$ on the support of $\varphi$, so that we have $\tilde{\varphi}\partial^{\alpha}\varphi = \partial^{\alpha}\varphi$ for any multi-index $\alpha$. Since $Q$ and its derivatives decay (more than) polynomially, \eqref{diff_eq_y1} gives, for $s \geq 3$,
\begin{equation}
    \|{\varphi(x/R)} [(1-\Delta)^2+e_0^2]\mathcal{Y}_1\|_{H^{s-3}} 
    \leq \frac{C}{R}\|\tilde{\varphi}(x/R)\mathcal{Y}_1\|_{H^{s-1}}
    \leq \frac{C}{R}\|\tilde{\varphi}(x/R)\mathcal{Y}_1\|_{H^{s}}.
\end{equation}
Using the trivial commutator estimate $\|[(1-\Delta)^2+e_0^2,\phi(x/R)]\|_{H^{s-3}\to H^{s-3}} \leq C/R$, we get
\begin{equation}
    \| \varphi(x/R)\mathcal{Y}_1\|_{H^{s+1}} \approx \|[(1-\Delta)^2+e_0^2](\varphi(x/R)\mathcal{Y}_1)\|_{H^{s-3}} 
    \leq \frac{C}{R}\|\tilde{\varphi}(x/R)\mathcal{Y}_1\|_{H^{s}}.
\end{equation}

By the induction hypothesis, we get $\| \varphi(x/R)\mathcal{Y}_1\|_{H^{s+1}} \leq C/R^{k+1}$, as desired. The same argument shows that, if $\lambda \in \Real \backslash\sigma(\mathcal{L})$, then $(\lambda - \mathcal{L})^{-1} S(\Real^N) \subset S(\Real^N)$.

\textit{Critical case.}

The range of the operator $L_-$ is no longer closed, but the operator $1+L_-$  is invertible on $\{Q\}^\perp$. Therefore, for any $\varepsilon >0$, one can take $G_\varepsilon \in H^2$ such that
\begin{equation}
    \|L_-G_\varepsilon - (1+L_-)Z\|_{L^2} < \varepsilon.
\end{equation}
Letting $h_\varepsilon := (1+L_-)^{-1}L_-^{\frac{1}{2}}G_\varepsilon = L_-^{\frac{1}{2}}(1+L_-)^{-1}G_\varepsilon = 
(1+L_-)^{-1}(1+L_-)^{-\frac{1}{2}}L_-^{\frac{1}{2}}(1+L_-)^{-\frac{1}{2}}G_\varepsilon \in H^3$, we have
\begin{align}
    \|L_-^{\frac{1}{2}}h_\varepsilon-Z\|_{H^2} &= \|(1-\Delta)(1+L_-)^{-1}[L_-G_\varepsilon - (1+L_-)Z]\|_{L^2}\\
    &\leq \varepsilon\|[1-V(1-\Delta)^{-1}]^{-1}\|_{L^2 \to L^2}.
\end{align}
Choosing $\tilde{h}_\varepsilon \in H^4$ such that $\tilde{h}_\varepsilon \perp Q$ and $\|h_\varepsilon-\tilde{h}_\varepsilon\|_{H^3} < \varepsilon$, and recalling $P = L_-^{\frac{1}{2}}L_+L_-^{\frac{1}{2}}$ and \eqref{Z_positive_dir}, we get 
\begin{equation}
    (P \tilde{h}_\varepsilon,\tilde{h}_\varepsilon)_{L^2} = (L_+Z,Z)_{L^2}+O(\varepsilon).
\end{equation}
Thus, if $\varepsilon$ is small enough, the conclusion follows. The regularity and the decay of $\mathcal{Y}_{\pm}$ follow analogously from the argument for the intercritical case.
\end{proof}

Consider the bilinear form 
\begin{align}
	B(f,g) &:= \frac{1}{2}(L_+f_1,g_1)+ \frac{1}{2}(L_-f_2,g_2)\\
	&= \frac{1-s_c}{2}\int f_1 \cdot  g_1+  \frac{1}{2}\int \nabla f_1 \cdot \nabla g_1 - \frac{p}{2}\int Q^{p-1}f_1g_1+\\
	&\quad\quad\quad +\frac{1-s_c}{2}\int f_2 \cdot  g_2+ \frac{1}{2}\int \nabla f_2 \cdot \nabla g_2 - \frac{1}{2} \int Q^{p-1}f_2 g_2,
\end{align}

and define the \textit{linearized energy}
\begin{align}
\Phi(f) &:= B(f,f) = \frac{1}{2}(L_+f_1,f_1)+\frac{1}{2}(L_-f_2,f_2)\\
&= \frac{1-s_c}{2} \int |\nabla f|^2+\frac{1}{2} \int |\nabla f|^2 -\frac{1}{2}\int Q^{p-1}(p |f_1|^2 + |f_2|^2).
\end{align}

If $0 < s_c \leq 1$, one can check directly that, for any $f$, $g \in S(\Real^N)$,
\begin{equation}\label{prop_B}
	\begin{gathered}
	B(f,g) = B(g,f),\\
	B(\mathcal{L}f,g) = - B(f,\mathcal{L}g),\\
	\quad B(iQ,f) = 0,\\
	B(\partial_k Q,f) = 0, \,k \leq N,\\
	 B(\Lambda Q,f) = -\frac{(1-s_c)(p-1)}{2}\int Q^pf_1,\\
	\Phi(\mathcal{Y}_+) = \Phi(\mathcal{Y}_-) = 0.\\
	\end{gathered}
\end{equation}
In the energy-critical case, we have
\begin{equation}
     B(\Lambda W,f) = 0.
\end{equation}

If $ 0 < s_c < 1$, consider the following orthogonality relations 
\begin{equation}\label{sub_o1}
\int Q v_2 = \int \partial_k Q v_1 = 0, \quad 1 \leq k \leq N,    
\end{equation}
\begin{equation}\label{sub_o2}
    \int  Q^p v_1 = 0,
\end{equation}
\begin{equation}\label{sub_o3}
    \int\mathcal{Y}_1 v_2 = \int\mathcal{Y}_2v_1 = 0.
\end{equation}
Denote by $G^\perp$ the set of $v \in H^1$ satisfying \eqref{sub_o1} and \eqref{sub_o2}, and $\tilde{G}^\perp$ the set of $v \in H^1$ satisfying \eqref{sub_o1} and \eqref{sub_o3}. 
\begin{re}\label{re_ortho}
Differently from \cite{DR_Thre}, we use the orthogonality condition \eqref{sub_o2} instead of $\int \Delta Q v_1 = 0$. We make this choice in order to be able to prove coercivity in all dimensions, specially in dimension $N= 1$.     
\end{re}

By direct calculations, one sees that
\begin{equation}
	\Phi_{|\text{span}\{\nabla Q, iQ\}} = 0
\end{equation}
and

\begin{equation}\label{sub_Phi_W_neg}
	\Phi(Q) = -\frac{p+1}{2}\int Q^{p+1} < 0.
\end{equation}

If $s_c = 1$, consider the directions $W$, $iW$, $\Lambda W = \frac{N-2}{2}W+x \cdot \nabla W$ and $\partial_k W$, $1\leq k \leq N$, in the Hilbert space $\dot{H}^1 = \dot{H}^1(\Real^N,\mathbb{C})$ with real inner product defined in \eqref{inner_dot}. 
Denote by $G :=$ span$\{W, \nabla W, iW, \Lambda W\}$  and by $G^\perp$ its orthogonal complement in $\dot{H}^1$ with
\begin{equation}\label{inner_dot}
	(f,g)_{\dot{H}^1} = \int \nabla f_1 \cdot \nabla g_1 + \int \nabla f_2 \cdot \nabla g_2 = \Re\int \nabla f \cdot \nabla \overline{g}.
\end{equation}
Let $\tilde{G}^\perp$ be the set $\left\{v \in \dot{H}^1; \, v \perp \text{span}\{\nabla W,iW,\Lambda W\}, \, B(\mathcal{Y}_+,v) = B(\mathcal{Y}_-,v) = 0\right\}$.

By direct calculations, one sees that
\begin{equation}
	\Phi_{|\text{span}\{\nabla W, iW, \Lambda W\}} = 0
\end{equation}
and
\begin{equation}\label{Phi_W_neg}
	\Phi(W) = -\frac{2}{(N-2)C_N^N} < 0,
\end{equation}
where $C_N$ is the sharp constant for Sobolev inequality for the embedding $\dot{H}^1 (\Real^N)\hookrightarrow L^\frac{2N}{N-2}(\Real^N)$. The following lemma shows that $\Phi$ is coercive in $G^\perp \cup \tilde{G}^\perp$. 

\begin{lemma}\label{lem_coerc}
For $0 < s_c \leq 1$, there exists a constant $\tilde{c} > 0$ such that, for any $f \in G^\perp\cup\tilde{G}^\perp$
\begin{equation}\label{Phi_coerc}
	\Phi(f) \geq \tilde{c}\|f\|_{\dot{H}^1}^2.
\end{equation}
\end{lemma}
{This result was proved in a different context in \cite{Rey90}, in the energy-critical case, and in \cite{DR_Thre} in the 3d cubic case. We give here the proof for all $0 < s_c \leq 1$, in any dimension.}

\begin{proof}[Proof of Lemma \ref{lem_coerc}, energy-critical case]$ $   \newline
\textit{Step 1. Coercivity in $G^\perp$.} We adapt here the proof in \cite{Rey90} to our context.

Let $\Pi : S^N \to \Real^N$ be the ``stretched'' stereographic projection of the sphere $S^N$ onto $\Real^N$, with respect to the North pole, defined by
\begin{equation}
    y_i = \frac{1}{N(N-2)} \frac{x_i}{1-x_{N+1}}, \quad 1 \leq i \leq N.
\end{equation}

If $y = \Pi x$ and $v$ is a real function is defined on $\Real^N$, we define a function $u$ on $S^N$ by
\begin{equation}
    u(x) = W^{-1}(y)v(y).
\end{equation}
By integration by parts, one can check that
\begin{equation}
    \int_{S^N} |\nabla_{S_N} u|^2 d\sigma = 2^{N-2}\int_{\Real^N}\left( |\nabla v|^2 - W^{p_c} v^2  \right)dy, 
\end{equation}
and
\begin{equation}
      \int_{S^N} u^2 d\sigma = \frac{2^{N}}{N(N-2)}\int_{\Real^N} W^{p_c} v^2  dy. 
\end{equation}
The spectrum of $\Delta_{S^N}$ is well-known \cite{5Rey}. Namely, for the first eigenvalues $\lambda_{k}$, with multiplicity $n_{k}$ and associated eigenfunctions $u_{k,\alpha}$, with $\alpha =(\alpha_1, \cdots, \alpha_N)\in \mathbb{Z}^N_{\geq 0}$,  we have
\begin{align}
\begin{array}{llll}
    \lambda_0 = 0,      &n_0 = 1,      &u_0 = 1,  &     \\
    \lambda_1 = N,      &n_1 = N+1,    &u_{1,j} = x_i, &1 \leq j \leq N+1,\\ 
    \lambda_2 = 2(N+1), & n_2 = 2N+3.  &
\end{array}
\end{align}

Therefore, if $v \perp W$ in $\dot{H}^1$, then $u$ is orthogonal to $u_0$, and we have
\begin{equation}
    \int_{\Real^N} \left( |\nabla v|^2 - W^{p_c} v^2  \right)  dy \geq \frac{4\lambda_1}{N(N-2)} \int_{\Real^N}  W^{p_c} v^2 dy, 
\end{equation}
which is equivalent to
\begin{equation}
    \int_{\Real^N} \left( |\nabla v|^2 - W^{p_c} v^2  \right)  dy \geq \frac{4}{N+2} \int_{\Real^N} |\nabla v|^2  dy. 
\end{equation}
Similarly, if $v \perp \text{span}\{W,\nabla W, \Lambda W\}$ in $\dot{H}^1$, then $u$ is orthogonal to $u_0$, $u_{1,i}$, $1 \leq i \leq N+1$, and thus
\begin{equation}
    \int_{\Real^N} \left( |\nabla v|^2 - W^{p_c} v^2  \right)  dy \geq \frac{4\lambda_2}{N(N-2)} \int_{\Real^N} W^{p_c} v^2  dy,
\end{equation}
which is equivalent to
\begin{equation}
    \int_{\Real^N} \left(|\nabla v|^2 - p_cW^{p_c} v^2 \right) dy \geq \frac{4}{N+2} \int_{\Real^N} |\nabla v|^2  dy. 
\end{equation}
Therefore, we proved that, for $h \in G^\perp$,
\begin{equation}
    \Phi(h) \geq \frac{4}{N+2}\|h\|_{\dot{H}^1}.
\end{equation}

\textit{Step 2. Coerciveness of $\Phi$ in $\tilde{G}^\perp$}.

We first claim that $B(\mathcal{Y}_+,\mathcal{Y}_-) \neq 0$. If $B(\mathcal{Y}_+,\mathcal{Y}_-)$ was $0$, then $\Phi$ would 
be identically $0$ on span$\{\nabla W, iW, \Lambda W, \mathcal{Y}_+, \mathcal{Y}_-\}$, a subspace of dimension $N+4$. But this cannot happen, given $\Phi$ is positive definite on $G^\perp$, which is of codimension $N+3$.

We now show that $\Phi(h) > 0$ on $\tilde{G}^\perp \backslash \{0\}$. Assume, by contradiction, that there exists $h \in \tilde{G}^\perp \backslash \{0\}$ such that $\Phi(h) \leq 0$. Recall that $\ker \mathcal{L} = \text{span}\{\nabla W, iW, \Lambda W\}$, and that, by definition of $\tilde{G}^\perp \backslash \{0\}$, $B(\mathcal{Y}_+,h) = 0$. Hence, the vectors $\partial_kW$, $1 \leq k \leq N$, $iW$, $\Lambda W$, $\mathcal{Y}_+$ and $h$ are mutually orthogonal under the symmetric form $B$. Since
\begin{equation}
	\Phi(\partial_k W) = \Phi(iW) = \Phi(\Lambda W) = \Phi(\mathcal{Y}_+) = 0,
\end{equation}
we get
\begin{equation}
	\Phi_{|\text{span}\{\nabla W, iW, \Lambda W, \mathcal{Y}_+,h\}} \leq 0.
\end{equation}
We claim that these vectors are independent. Indeed, if
\begin{equation}
	\sum_k \alpha_k \partial_k W + \beta iW + \gamma \Lambda W +\delta \mathcal{Y}_+ +\epsilon h = 0,
\end{equation}
then
\begin{equation}
	\delta B(\mathcal{Y}_+,\mathcal{Y}_-) = 0,
\end{equation}
and since $B(\mathcal{Y}_+,\mathcal{Y}_-) \neq 0$, $\delta = 0$. Therefore, the claim is proven, since $\partial_k W$, $iW$, $\Lambda W$ and $h$ are orthogonal in the real Hilbert space $\dot{H}^1$.

To prove coercivity, we rely on a compactness argument. Suppose, by contradiction, that there exists $\{h_n\} \subset \tilde{G}^\perp$ such that
\begin{equation}
	\lim \Phi(h_n) = 0, \quad \|h_n\|_{\dot{H}^1} = 1.
\end{equation}

Up to a subsequence, we may assume $h_n \rightharpoonup h^*$ weakly in $\dot{H}^1$. This implies $h^* \in \tilde{G}^\perp$. Since the operator $\int W^{p_c-1}|\cdot|^2$ is compact, we have $\int W^{p_c-1}|h_*|^2 > 0$ and 
\begin{equation}
	\Phi(h^*)\leq 
	\lim\inf \Phi(h_n) = 0.
\end{equation}
This contradicts the strict positivity of $\Phi$ on $\tilde{G}^\perp\backslash\{0\}$.

\end{proof}

\begin{proof}[Proof of Lemma \ref{lem_coerc}, intercritical case] Since the explicit formula for $Q$ in the intercritical case is not available, we cannot proceed as in the energy-critical case. We follow here \cite{W_1989} and \cite{DR_Thre}.

\textit{Step 1. Non-negativity on $G^\perp$}. Define the functional 
\begin{equation}
J(u) = \frac{\left(\displaystyle\int|\nabla u|^2\right)^{a}\left(\displaystyle\int|u|^2\right)^{b}}{\displaystyle\int |u|^{p+1}}, 
\end{equation}
where
\begin{equation}
    a = \frac{N(p-1)}{4}, \quad b = \frac{2p+2-N(p-1)}{4}.
\end{equation}
By the sharp Gagliardo-Nirenberg inequality, one can see that this functional achieves an absolute min\-\\
i\-mum at $Q$. Therefore, the minimization condition $\frac{d^2}{d\varepsilon^2}J(Q+\varepsilon h)_{|\varepsilon = 0} \geq 0$ for all functions $h \in H^1$ gives
\begin{equation}
    \Phi(h) \geq\frac{b}{\left(\displaystyle\int Q^2\right)}\left[ \frac{1}{a}\left(\int \Delta Q h_1\right)^2
    +\frac{1}{b}\left(\int Q h_1\right)^2 - \left(\int Q^p h_1\right)^2\right].
\end{equation}

Since $a$ and $b$ are positive if $0 < s_c < 1$, we have that $\Phi(h) \geq 0$ if $\int Q^p h_1 = 0 $. Therefore, $\Phi$ must be non-negative on $G^\perp$.

\textit{Step 2. Coercivity on $G^\perp$}. We now employ compactness to show that, for every real function $h \in G^\perp$, 
\begin{equation}
    (L_+h,h)_{L^2} \gtrsim \|h\|_{L^2}, \text{ and } (L_-h,h)_{L^2} \gtrsim \|h\|_{L^2}.
\end{equation}
If we prove both inequalities, then (again) by compactness, the coercivity follows.
Suppose that there is a sequence of real $H^1$ functions $\{h_n\}$ in $G^\perp$ such that
\begin{equation}
    \lim_{n \to \infty} (L_+h_n,h_n)_{L^2} = \Phi(h_n) = 0, \text{ and } \|h_n\|_{L^2} = 1.
\end{equation}
This implies
\begin{equation}\label{sub_aux}
    0 \leq \frac{1}{2}\int|\nabla h_n|^2  = -\frac{1}{2}+\frac{p}{2}\int Q^{p-1}h_n^2 +\Phi(h_n) \lesssim 1.
\end{equation}
Therefore, $\|\nabla h_n\| \lesssim 1$ and, for large $n$, $\int Q^{p-1}h_n^2 \gtrsim 1$. Passing to a subsequence, and recalling that $Q$ decays at infinity, we get that there exists $h_* \in G^\perp$ such that 
\begin{equation}
    h_n \rightharpoonup h_* \text{ weakly in } H^1, \text{ and } \int Q^{p-1}h_*^2 >0.
\end{equation}
In particular, $h_* \neq 0$. Moreover,
\begin{equation}
    \Phi(h_*) \leq \frac{1}{2}\liminf_{n+\to \infty}\|h_n\|_{H^1}-\frac{p}{2} \lim_{n \to \infty}\int Q^{p-1}h_n^2 = \liminf_{n+\to \infty} \Phi(h_n) = 0.  
\end{equation}
Recall that $\Phi(h_*) \geq 0$ by Step 1. Therefore, $\Phi(h_*) = 0$ and $h_*$ is the solution to the minimization problem 
\begin{equation}
    0 =  (L_+h_*,h_*)_{L^2} = \min_{f\in E}  (L_+h,h)_{L^2}, \text{ where}
\end{equation}
\begin{equation}
    E:= \left\{ h \in H^1: \|h\|_{L^2} = \|h_*\|_{L^2} \text{ and } h \in G^{\perp}\right\}.
\end{equation}

Thus, there exist Lagrange multipliers $\lambda_0, \cdots, \lambda_{N+1}$ such that
\begin{equation}
    L_+h_* = \lambda_0 Q^p + \sum_{j=1}^N \lambda_j \partial_{j}Q + \lambda_{N+1}h_*.
\end{equation}

Since $h_* \in G^{\perp}\backslash\{0\}$ and $(L_+h_*,h_*)_{L^2} = 0$, we have $\lambda_{N+1}=0$. By testing the last equation against $\partial_j Q$ and using that $L_+(\partial_k Q) = 0$, for all $k \leq N$, we conclude that 
\begin{equation}
    L_+h_* = \lambda_0 Q^p.
\end{equation}
Recalling that $L_+Q = -\frac{p-1}{2}Q^p$ and $\ker(L_+) = \text{span}\{\nabla Q\}$, we conclude that there exist $\mu_1, \cdots, \mu_N$ such that
\begin{equation}
    h_* = -\frac{2\lambda_0}{p-1}Q +  \sum_{j=1}^N \mu_j \partial_{j}Q.
\end{equation}
Noting that $\int Q \partial_j Q = \frac{1}{2}\int \partial_j(Q^2) = 0$, and recalling that $h_* \in G^\perp$ gives $\mu_j = 0$ for all $j$. Therefore,
\begin{equation}
    h_* = \frac{2\lambda_0}{p-1}Q.
\end{equation}
And, by direct calculation,
\begin{equation}
    (L_+h_*,h_*)_{L^2} = -\left(\frac{2\lambda_0}{p-1}\right)^2\int Q^{p+1} < 0.
\end{equation}
This contradicts $(L_+h_*,h_*)_{L^2} \geq 0$ and $h_* \neq 0$, and proves that 
\begin{equation}
    (L_+h,h)_{L^2} \gtrsim \|h\|_{L^2}
\end{equation}
for any real function $h \in G^{\perp}$. The proof for $L_-$ is analogous. In particular, we have strict positivity of $\Phi$ on $G^\perp\backslash\{0\}$ and, by compactness, the coercivity follows on $G^\perp$.

\textit{Step 3. Coercivity on $\tilde{G}^\perp$.} The proof relies on a (co)dimensional argument, together with compactness, as in the energy-critical case.
\end{proof}

Unlike the energy-critical case, the ground state decays exponentially if $s_c < 1$. In the next sections, we need sharp bounds on the decay of $Q$ and its derivatives. We start recalling the following result, proved by Gidas, Ni and Nirenberg \cite{Gidas81}. Recall that $2^*=2N/(N-2)$, if $N\geq 3$ and $2^*=+\infty$, if $N=1, 2$.

\begin{lemma}[Gidas et al. {\cite{Gidas81}*{Theorem 2, p. 370}}]\label{sub_GNN}
For $1+\frac{4}{N} \leq p < 2^*-1$, let $Q \in \mathcal{S}(\Real^N)$, Schwartz space, be the unique radial positive solution of the equation \eqref{intercrit_Q}.
Then there exists $C > 0$ such that
\begin{equation}
    \lim_{|x| \to +\infty} |x|^{\frac{N-1}{2}}e^{|x|} Q(x) = C.
\end{equation}
\end{lemma}
We next study the decay of solutions to the equation \eqref{intercrit_Q}. {The next lemma, as well as its corollary, might be known in the theory of elliptic equations. However, we could not find a specific reference, and for convenience provide a proof here.}
\begin{lemma}\label{sub_boot_decay} Let $f \in S(\mathbb{R}^N)$ and $\lambda \in \Real$. If $f$ solves
\begin{equation}
    (1-\Delta+\lambda i)f = G
\end{equation}
with 
\begin{equation}
|G(x)| \lesssim \frac{e^{-a|x|}}{\left(1+|x|^\frac{N-1}{2}\right)^b},
\end{equation}
for $0<a \neq \Re\sqrt{1+\lambda i}$, $0<b \neq 1$, then
\begin{equation}
    |f(x)| \lesssim \frac{1}{\left(1+|x|^\frac{N-1}{2}\right)^{\min\{b,1\}}}\left(e^{-|x|}\right)^{\min\{a,\Re\sqrt{1+\lambda i}\}}.
\end{equation}
\end{lemma}
\begin{proof}
Let $c = \Re\sqrt{1+\lambda i} \geq 1$. We recall the integral form of the resolvent (see \cite{AS61})
\begin{equation}
    (1-\Delta+\lambda i)^{-1}G = K * G,
\end{equation}
where $K \in L^1(\Real^N)$ is such that, for $|x| \gg 1$,
\begin{equation}\label{sub_decay_K_inf}
     K(x) \lesssim \frac{e^{-c|x|}}{1+|x|^{\frac{N-1}{2}}},
\end{equation}
and, for $|x|\ll 1$,
\begin{equation}\label{sub_decay_K_zero}
    K(x) \lesssim 
    \begin{cases} 
        \frac{1}{|x|^{\frac{N-1}{2}}} &\text{ for } N > 2,\\
        \ln\frac{1}{|x|} &\text{ for } N = 2,\\
        1 &\text{ for } N < 2.
    \end{cases}
\end{equation}
Consider first the case $0<a < c$. We estimate
\begin{align}
    |K*G(x)| &\lesssim \int K(y)\frac{e^{-a|x-y|}}{\left(1+|x-y|^\frac{N-1}{2}\right)^b}\,dy\\
    &\lesssim \frac{e^{-a|x|}}{\left(1+|x|^\frac{N-1}{2}\right)^{\min\{b,1\}}} \int K(y)e^{a|y|}\frac{\left(1+|x-y|^{\frac{N-1}{2}}+|y|^\frac{N-1}{2}\right)^{\min\{b,1\}}}{\left(1+|x-y|^\frac{N-1}{2}\right)^b}\,dy\\
    &\lesssim \frac{e^{-a|x|}}{\left(1+|x|^\frac{N-1}{2}\right)^{\min\{b,1\}}} \int K(y)e^{a|y|}\left(1+|y|^{\frac{N-1}{2}}\right)^{\min\{b,1\}}\,dy.\\
\end{align}
By \eqref{sub_decay_K_inf} and \eqref{sub_decay_K_zero}, the integral in the last inequality is $O(1)$.
For $a > c$, the estimate is
\begin{align}
   |K*G(x)|&\lesssim\frac{e^{-c|x|}}{\left(1+|x|^\frac{N-1}{2}\right)^{\min\{1,b\}}} \int K(y)e^{c|y|}\left(1+|y|^{\frac{N-1}{2}}\right)^{\min\{1,b\}}e^{-(a-c)|x-y|}\,dy\\
   &\lesssim \frac{e^{-c|x|}}{\left(1+|x|^\frac{N-1}{2}\right)^{\min\{1,b\}}} \left[\int e^{-(a-c)|x-y|}\,dy + \int_{|y|\leq 1}K(y)e^{c|y|}\left(1+|y|^{\frac{N-1}{2}}\right)\,dy\right]. 
\end{align}
   
Since the first integral in the last inequality is bounded uniformly in $x$, the lemma is proved.
\end{proof}
\begin{cor}\label{sub_decay_eigen} For any multi-index $\alpha\in \mathbb{Z}^N_+$, the following estimates hold
\begin{enumerate}[(i)]
    \item\label{sub_a_inf} $\|Q^{-1}\partial^\alpha Q\|_{L^\infty} < +\infty$,
    \item\label{sub_b_inf} $\|Q^{-1}e^{\eta|x|}\partial^\alpha\mathcal{Y}_\pm\|_{L^\infty} < +\infty$, for some $0<\eta\ll1$,
    \item\label{sub_c_inf} $\|Q^{-1}e^{\eta|x|}\partial^\alpha[(\mathcal{L}-\lambda)^{-1}f]\|_{L^\infty} < +\infty$, for every  $\lambda \in \Real \backslash \sigma(\mathcal{L})$ and every $f\in S(\Real^N)$ such that $\|Q^{-1}e^{\eta|x|}\partial^\beta f\|_{L^\infty}<+\infty$ for some $0<\eta<\Re(\sqrt{1+\lambda i}) $ and any $\beta \in \mathbb{Z}^N_+$.
\end{enumerate}
\end{cor}
\begin{proof}
We first remark that $Q$ is strictly positive, and thus, $Q^{-1}$ is well-defined. Recalling Lemma \ref{sub_GNN}, we have, for all $x$,
\begin{equation}
    Q(x) \approx \frac{e^{-|x|}}{1+|x|^\frac{N-1}{2}}.
\end{equation}
We differentiate \eqref{sub_def_ground} to obtain
\begin{equation}
    (1-\Delta)\nabla Q = pQ^{p-1}\nabla Q.
\end{equation}
Since $Q \in S$, by Lemma \ref{sub_boot_decay} and a bootstrap argument, we conclude \eqref{sub_a_inf} for $|\alpha|=1$. Differentiating \eqref{sub_def_ground} and repeating the argument, we conclude \eqref{sub_a_inf}, by induction, for any multi-index $\alpha$.

To prove \eqref{sub_b_inf}, recall the differential equation for $\mathcal{Y}_1 = \Re(\mathcal{Y}_+)$ 
\begin{equation}
    (1-\Delta -pQ^{p-1})(1-\Delta -Q^{p-1})\mathcal{Y}_1 = -e_0^2 \mathcal{Y}_1.
\end{equation}

By factoring $[(1-\Delta)^2+e_0^2] = (1-\Delta+ie_0)(1-\Delta-ie_0)$ and using the item \eqref{sub_a_inf}, this equation can be rewritten as
\begin{equation}
     (1-\Delta+ie_0)(1-\Delta-ie_0)\mathcal{Y}_1 = G_2(\mathcal{Y}_1),
\end{equation}
where we define $G_k(f)$ as a linear function on $f$ and its derivatives up to order $k$ that satisfies, for any $k \geq 1$,
\begin{equation}
    |G_k(f)| \lesssim Q^{p-1}\sum_{|\alpha|\leq k}|\partial^\alpha f|.
\end{equation}
Writing $g=(1-\Delta-ie_0)\mathcal{Y}_1$, we have, for any multi-indices $\alpha, \beta$,
\begin{equation}
    \begin{cases}
        (1-\Delta+ie_0)\partial^\alpha g = G_{|\alpha|+2}(\mathcal{Y}_1)\\
        (1-\Delta-ie_0)\partial^\beta\mathcal{Y}_1 = \partial^\beta g.
    \end{cases}
\end{equation}
Therefore, using Lemma \ref{sub_boot_decay} and bootstrapping, we prove that $Q^{-1}e^{\eta|x|}\partial^\alpha\mathcal{Y}_1 \in L^\infty$ for any multi-index $\alpha$, where $0<\eta\ll \Re(\sqrt{1+ie_0})-1$. The estimate on $\mathcal{Y}_2$ is analogous, and hence, \eqref{sub_b_inf} holds.
We now turn to estimate \eqref{sub_c_inf}. If $g = (\mathcal{L}-\lambda)^{-1}f$, then, for any $\alpha$,
\begin{equation}
    \begin{cases}
    -\partial^\alpha(1-\Delta-Q^{p-1}) g_2-\lambda\partial^\alpha g_1 = \partial^\alpha f_1\\
    \partial^\alpha(1-\Delta-pQ^{p-1}) g_1-\lambda\partial^\alpha g_2 = \partial^\alpha f_2.
    \end{cases}
\end{equation}
We can rewrite this system as 
\begin{equation}
    [(1-\Delta)^2+\lambda^2]\partial^\alpha g_1 = G_{|\alpha|+2}(g_1)+H_{|\alpha|+2}(f),
\end{equation}
where we define $H_k(f)$ as a linear function on $f$ and its derivatives up to order $k$ that satisfies, for any $k \geq 1$,
\begin{equation}
    |H_k(f)| \lesssim \sum_{|\alpha|\leq k}|\partial^\alpha f|.
\end{equation}
Bootstrapping similarly to the previous items, and noting that the argument to $g_2$ is analogous, completes the proof of Lemma \ref{sub_boot_decay}.

\end{proof}

\subsection{Estimates on the linearized equation}

We now prove some estimates that will be used in the next sections. We start with estimates for the energy-critical case.

\begin{lemma}[Preliminary estimates]\label{lem_nonl} Let $s_c = 1$, $N \geq 6$ (and then $p_c-1 = \frac{4}{N-2}\leq 1$), $0 < \epsilon \ll \frac{4}{N-2}$ and $I$ be a bounded time interval with $|I|\leq 1$, and consider $f,g \in S(\dot{H}^1,I)$ such that $\nabla f, \nabla g \in S(L^2,I)$. The following estimates hold

\begin{enumerate}[(i)]
\item \label{linear1} $\|\nabla K(f)\|_{S'(L^2,I)} \lesssim |I|^\frac{1}{2} \|\nabla f\|_{S(L^2,I)}$,
\item \label{grad_r_point} $\begin{aligned}[t]{\|\nabla(R(f)-R(g))\|_{S'(L^2,I)}} \lesssim 
&\|\nabla(f-g)\|_{S(L^2,I)}\left(\|\nabla f\|^{p_c-1}_{S(L^2,I)} +\|\nabla g\|^{p_c-1}_{S(L^2,I)}\right)\\
&+\|D^\varepsilon(f-g)\|_{S(\dot{H}^{1-\varepsilon},I)}^{p_c-1}\left(
\|\nabla f\|_{S(L^2,I)}+\|\nabla g\|_{S(L^2,I)}\right).\end{aligned}$
\end{enumerate}
If $N >6$, then also
\begin{enumerate}[(i)]
\setcounter{enumi}{2}
\item \label{linear2} $\| D^\varepsilon K(f)\|_{S'(\dot{H}^{-(1-\varepsilon)},I)} \lesssim |I|^\frac{\varepsilon}{N-2} \|D^\varepsilon f\|_{S(\dot{H}^{1-\varepsilon},I)}$,
\item \label{r_point} $\|D^\varepsilon (R(f)-R(g))\|_{S'(\dot{H}^{-(1-\varepsilon)},I)} \lesssim \|D^\varepsilon(f-g)\|_{S(\dot{H}^{1-\varepsilon},I)}\left(\|\nabla f\|_{S(L^2,I)}^{p_c-1}+\|\nabla g\|_{S(L^2,I)}^{p_c-1}\right)$.
\end{enumerate}\end{lemma}

\begin{re}
It is necessary to treat the case $N >6$ differently due to the low power of the nonlinearity. If $N \leq 6$, then it is possible to estimate $\|\nabla (R(f)-R(g))\|_{S'(L^2,I)}$ at least linearly in terms of $\|\nabla(f-g)\|_{S(L^2,I)}$. In higher dimensions, one of the terms must be in the form $\|\nabla(f-g)\|_{S(L^2,I)}^{p_c-1}$, which is not good enough for the fixed-point argument carried on in the next section. The use of less than one derivative enables us to keep the desired linearity.
\end{re}
\begin{proof}[Proof of Lemma \ref{lem_nonl}] We start by proving the following claim:
\begin{claim}\label{claim1} Let $H$ be a map such that $H(0) = 0$ and $|H(f) - H(g)| \leq C |f-g|^\frac{4}{N-2}$ for all functions $f, g : \Real^N \to \mathbb{C}$, $N > 6$. Then, for all $f, g \in S(\dot{H}^1,I)\cap \nabla^{-1}S(L^2,I)$,
\begin{align}
	\left\| D^\varepsilon(H(f) g) \right\|_{L_I^{\frac{2}{\varepsilon}}L_x^{\frac{2N}{N+2}}} \lesssim \|\nabla f\|_{\frac{16}{\varepsilon(N-2)},\frac{8N}{4N-\varepsilon(N-2)}}^\frac{4}{N-2}\|D^\varepsilon g\|_{\frac{4}{\varepsilon},\frac{2N}{N-2+\varepsilon}}\\
	+	\|\nabla f\|_{\frac{4}{\varepsilon},\frac{2N}{N-\varepsilon}}^\frac{4}{N-2}\|D^\varepsilon g\|_{\frac{2(N-2)}{\varepsilon(N-4)},\frac{2N(N-2)}{(N-2)^2+4\varepsilon}}.
\end{align}
In other words,
\begin{equation}
    \left\| D^\varepsilon(H(f) g) \right\|_{S'(\dot{H}^{-(1-\varepsilon)},I)} \lesssim \|\nabla f\|_{S(L^2,I)}^\frac{4}{N-2}\|D^\varepsilon g\|_{S(\dot{H}^{1-\varepsilon},I)}.
\end{equation}
\end{claim}
\begin{proof}[Proof of Claim \ref{claim1}] By Leibiniz Rule \eqref{leibiniz} and Holder's inequality, we can write
\begin{align}
	 \left\| D^\varepsilon(H(f) g) \right\|_{L_I^{\frac{2}{\varepsilon}}L_x^{\frac{2N}{N+2}}} &\lesssim \left\| \| D^\varepsilon H(f)\|_{L_x^\frac{2N}{4+\varepsilon}} \|g\|_{L_x^\frac{2N}{N-2-\varepsilon}} \right\|_{L_I^\frac{2}{\varepsilon}} + \left\| \| H(f)\|_{L_x^\frac{N(N-2)}{2(N-2-\varepsilon)}} \|D^\varepsilon g\|_{L_x^\frac{2N(N-2)}{(N-2)^2+4\varepsilon}} \right\|_{L_I^\frac{2}{\varepsilon}}\\
	&\lesssim \| D^\varepsilon H(f)\|_{L_I^\frac{4}{\varepsilon}L_x^\frac{2N}{4+\varepsilon}} \|g\|_{L_I^\frac{4}{\varepsilon}L_x^\frac{2N}{N-2-\varepsilon}}  \\ &\quad+\| H(f)\|_{L_I^\frac{N-2}{\varepsilon}L_x^\frac{N(N-2)}{2(N-2-\varepsilon)}} \|D^\varepsilon g\|_{L_I^\frac{2(N-2)}{\varepsilon(N-4)}L_x^\frac{2N(N-2)}{(N-2)^2+4\varepsilon}}.
\end{align}

By Sobolev inequality,
\begin{equation}
   \|g\|_{L_I^\frac{4}{\varepsilon}L_x^\frac{2N}{N-2-\varepsilon}} \lesssim \|D^\varepsilon g\|_{L_I^\frac{4}{\varepsilon}L_x^\frac{2N}{N-2+\varepsilon}}, 
\end{equation}
and since, by the assumption on $H$, $|H(f)| \lesssim |f|^\frac{4}{N-2}$, we have, by Sobolev,
\begin{equation}
	\| H(f)\|_{L_I^\frac{N-2}{\varepsilon}L_x^\frac{N(N-2)}{2(N-2-\varepsilon)}} \lesssim \|f\|_{\frac{4}{\varepsilon},\frac{2N}{N-2-\varepsilon}}^\frac{4}{N-2} 
	\lesssim \|\nabla f\|_{\frac{4}{\varepsilon},\frac{2N}{N-\varepsilon}}^\frac{4}{N-2}.
\end{equation}
It remains to estimate $\| D^\varepsilon H(f)\|_{L_x^\frac{2N}{4+\varepsilon}}$. Choosing $\nu$ such that $\frac{(N-2)\varepsilon}{4} < \nu < 1$, by fractional chain rule\footnote{This is where the hypothesis $N>6$ is used, as the fractional chain rule requires $0 < 4/(N-2) < 1$. } \eqref{frac_chain} and Sobolev embeddings, we have
\begin{align}
	\| D^\varepsilon H(f)\|_{L_x^\frac{2N}{4+\varepsilon}} &\lesssim \|f\|^{\frac{4}{N-2}-\frac{\varepsilon}{\nu}}_{L_x^{\left(\frac{4}{N-2}-\frac{\varepsilon}{\nu}\right)p_1^{}}} \|D^\nu f\|_{L_x^{\frac{\varepsilon}{\nu}q_1^{}}}^{\frac{\varepsilon}{\nu}}\\
	&\lesssim \|\nabla f\|^{\frac{4}{N-2}-\frac{\varepsilon}{\nu}}_{L_x^{p_2^{}}}\|\nabla f\|_{L_x^{q_2^{ }}}^\frac{\varepsilon}{\nu},
\end{align}

where we choose $p_2 = q_2 = \frac{8N}{4N-\varepsilon(N-2)} \in (1,+\infty)$, and $p_1$ and $q_1$ must satisfy 
\begin{equation}
	1 < p_1, q_1 < \infty,
	\end{equation}\begin{equation}
	\frac{1}{p_2} = \frac{1}{\left(\frac{4}{N-2}-\frac{\varepsilon}{\nu}\right)p_1}+\frac{1}{N}, \quad
	\frac{1}{q_2} = \frac{1}{\frac{\varepsilon}{\nu}q_1}+\frac{1-\nu}{N},\\
\end{equation}
and
\begin{equation}
	\left(1-\frac{\varepsilon(N-2)}{4\nu}\right)p_1 > 1.
\end{equation}


These conditions can be easily satisfied if $\varepsilon$ is small enough (depending only on the dimension). The claim is now proved.
\end{proof}

The estimate \eqref{linear2} of Lemma \ref{lem_nonl} follows directly from Sobolev inequality and Claim \ref{claim1}, by taking $H(W) = |W|^\frac{4}{N-2}$, and from the fact that $|\nabla W| \in L^2_x\cap L^\infty_x$, if $N > 6$. 

To prove \eqref{r_point}, note that $R(f) = W^{p_c}J(W^{-1}f)$, where $J(z) = |1+z|^{p_c-1}(1+z)-1-\frac{p_c+1}{2}z -\frac{p_c-1}{2}\bar{z}$ is $C^1(\mathbb{C})$. Its derivatives $J_z$ and $J_{\bar{z}}$ satisfy $J_z(0) = J_{\bar{z}}(0) = 0$ and, if $N > 6$, are Hölder continuous of order $p_c-1 < 1$. Therefore, writing
\begin{equation}\label{R_diff}
	R(f) - R(g) = W^{p_c-1}\int_0^1 J_z(W^{-1}(sf+(1-s)g))(f-g) + J_{\bar{z}}(W^{-1}(sf+(1-s)g))(\overline{f-g}) ds,
\end{equation}

we can apply Claim \ref{claim1} to estimate each term in \eqref{R_diff}, taking $H(f) = W^{p_c-1}J_z(W^{-1}f)$ or $H(f) = W^{p_c-1}J_{\bar{z}}(W^{-1}f)$. Estimate \eqref{r_point} then follows directly.

To prove \eqref{linear1}, we write
\begin{equation}
	|\nabla K(f)| \lesssim |W|^{p_c-2}|\nabla W| |f| + |W|^{p_c-1}|\nabla f|.
\end{equation}
 Using the fact that $|\partial^{\alpha} W(x)| \leq C_{\alpha} |W(x)|$ for every multi-index $\alpha \in \mathbb{Z}_+^N$ and all $x$, we have, by Hölder inequality
 \begin{align}
	\|\nabla K(f)\|_{L_I^2L_x^{\frac{2N}{N+2}}} &\lesssim  \left\| \|W\|_{L^\frac{2N}{N-2}_x}^{\frac{4}{N-2}} \|f\|_{L^\frac{2N}{N-2}_x} \right\|_{L^2_t}+\left\| \|W\|_{L^\frac{4N}{N-2}_x}^{\frac{4}{N-2}} \|\nabla f\|_{L^2_x} \right\|_{L^2_t} \\
	&\lesssim |I|^\frac{1}{2} \left(\|f\|_{L^\infty_t L^\frac{2N}{N-2}_x}+\|\nabla f\|_{L^\infty_t L^2_x}\right).
 \end{align}

Note that we used that $W \in L_x^\frac{2N}{N-2} \cap L_x^\frac{4N}{N-2}$, which follows from the fact that $W \in L^2_x\cap L^\infty_x$, if $N > 6$. The inequality follows from Sobolev embedding.

We finally turn to estimate \eqref{grad_r_point}. Write
\begin{align}
	\nabla (R(f) - R(g)) = &\underbrace{p_c W^{p_c-1} \nabla W (J(W^{-1} f)-J(W^{-1}g))}_{(a)} \\
	&+ \underbrace{W^{p_c -1}J_z(W^{-1}f)\nabla f - W^{p_c -1}J_z(W^{-1}g)\nabla g}_{(b)} \\
	&+ \underbrace{W^{p_c -1}J_{\bar{z}}(W^{-1}f)\nabla \bar{f} - W^{p_c -1}J_{\bar{z}}(W^{-1}g)\nabla \bar{g}}_{(c)}\\
	&+\underbrace{W^{p_c -2} \nabla  W J_{z}(W^{-1}f) f - W^{p_c -2}\nabla W J_{z}(W^{-1}g) g}_{(d)}\\
	&+\underbrace{W^{p_c -2} \nabla  W J_{\bar{z}}(W^{-1}f) \bar{f} - W^{p_c -2}\nabla W J_{\bar{z}}(W^{-1}g) \bar{g}}_{(e)}
\end{align}
To estimate $(a)$, note that
\begin{align}
	|(a)| &\lesssim W^{p_c-1}\int_0^1 |J_z(W^{-1}(sf+(1-s)g))(f-g) + J_{\bar{z}}(W^{-1}(sf+(1-s)g))(\overline{f-g})| ds \\
	&\lesssim \left(|f|^{p_c-1}+|g|^{p_c-1}\right)|f-g|.
\end{align}
Thus, by Hölder and Sobolev inequalities,
\begin{align}
	\|(a)\|_{L_I^2L_x^{\frac{2N}{N+2}}} &\lesssim |I|^\frac{1}{2}\left(\|f\|^{p_c-1}_{L_I^\infty L_x^\frac{2N}{N-2}}+\|g\|^{p_c-1}_{L_I^\infty L_x^\frac{2N}{N-2}} \right)\|f-g\|_{L_I^\infty L_x^\frac{2N}{N-2}}\\
	& \lesssim |I|^\frac{1}{2}\left(\|\nabla f\|^{p_c-1}_{L_I^\infty L_x^2}+\|\nabla g\|^{p_c-1}_{L_I^\infty L_x^2} \right)\|\nabla (f-g)\|_{L_I^\infty L_x^2}.\\
\end{align}

We now estimate $(b)$. By triangle inequality,
\begin{align}
	|(b)| &\leq W^{p_c -1}|J_z(W^{-1}f)||\nabla f-\nabla g| + W^{p_c -1}|J_z(W^{-1}f)-J_z(W^{-1}g)||\nabla g|\\
	& \leq |f|^{p_c-1} |\nabla (f- g)| + |f-g|^{p_c-1}|\nabla g|.\\
\end{align}
So that, by Hölder and Sobolev inequalities,
\begin{align}
	\|(b)\|_{L_I^2L_x^{\frac{2N}{N+2}}} &\lesssim \|f\|^{p_c-1}_{L_I^\infty L_x^\frac{2N}{N-2}} \|\nabla (f- g)\|_{L_I^2 L_x^\frac{2N}{N-2}} 
	+ \|f-g\|^{p_c-1}_{L_I^\infty L_x^\frac{2N}{N-2}}\|\nabla g\|_{L_I^2 L_x^\frac{2N}{N-2}}\\
	&\lesssim \|\nabla f\|^{p_c-1}_{L_I^\infty L_x^2} \|\nabla (f- g)\|_{L_I^2 L_x^\frac{2N}{N-2}} +\|D^\varepsilon (f-g)\|^{p_c-1}_{L_I^\infty L_x^\frac{2N}{N-2+2\varepsilon}}\|\nabla g\|_{L_I^2 L_x^\frac{2N}{N-2}}.\\
\end{align}
The estimate for $(c)$ is analogous. To estimate $(d)$, we write
\begin{align}
	|(d)| &\leq W^{p_c -1}|J_z(W^{-1}f)|| f- g| + W^{p_c -1}|J_z(W^{-1}f)-J_z(W^{-1}g)|| g|\\
	& \leq |f|^{p_c-1} |f- g| + |f-g|^{p_c-1}|g|.
\end{align}

Therefore, by H\"older and Sobolev,
\begin{align}
	\|(d)\|_{L_I^2L_x^{\frac{2N}{N+2}}} &\lesssim |I|^{\frac{1}{2}}\|f\|^{p_c-1}_{L_I^\infty L_x^\frac{2N}{N-2}} \|f- g\|_{L_I^\infty L_x^\frac{2N}{N-2}} 
	+ |I|^{\frac{1}{2}}\|f-g\|^{p_c-1}_{L_I^\infty L_x^\frac{2N}{N-2}}\| g\|_{L_I^\infty L_x^\frac{2N}{N-2}}\\
	&\lesssim |I|^{\frac{1}{2}}\|\nabla f\|^{p_c-1}_{L_I^\infty L_x^2} \|\nabla (f- g)\|_{L_I^\infty L_x^2} +|I|^{\frac{1}{2}}\|D^\varepsilon (f-g)\|^{p_c-1}_{L_I^\infty L_x^\frac{2N}{N-2+2\varepsilon}}\|\nabla g\|_{L_I^\infty L_x^2}.\\
\end{align}

Since the estimate for $(e)$ is analogous, the proof of Lemma \ref{lem_nonl} is complete.

\end{proof}

The following Strichartz-type continuity argument follows from Lemma \ref{lem_nonl} and will be useful in proving the main theorems of this paper.

\begin{lemma}\label{lem_stric_exp}
Let $h$ be a solution to \eqref{linearized_eq_2}. If, for some $c>0$, and all $t > 0$,
\begin{equation}\label{exp_h}
    \|h(t)\|_{\dot{H}^1} \lesssim e^{-c t},
\end{equation}
then, for all $t > 0$,
\begin{equation}\label{stric_exp}
    \|\nabla h\|_{S(L^2,\,[t,+\infty))} \lesssim e^{-ct}.
\end{equation}
\end{lemma}
\begin{proof}
Differentiating \eqref{linearized_eq_2}, we get
\begin{equation}
    i\partial_t (\nabla h) + \Delta (\nabla h) +\nabla(K(h)+R(h)) = 0.
\end{equation}
By Duhamel formula, Strichartz estimates and items \eqref{linear1} and \eqref{grad_r_point} of Lemma \ref{lem_nonl}, if $0 < \tau < 1$,
\begin{equation}
    \|\nabla h\|_{S(L^2,\,[t,t+\tau])} \lesssim \|h(t)\|_{\dot{H}^1}+ \tau^\frac{1}{2}\|\nabla h\|_{S(L^2,\,[t,t+\tau])}+\|\nabla h\|_{S(L^2,\,[t,t+\tau])}^{p_c}.
\end{equation}
By \eqref{exp_h}, we get, for some $K>0$,
\begin{equation}\label{stric_K}
    \|\nabla h\|_{S(L^2,\,[t,t+\tau])} \leq K( e^{-ct}+ \tau^\frac{1}{2}\|\nabla h\|_{S(L^2,\,[t,t+\tau])}+\|\nabla h\|_{S(L^2,\,[t,t+\tau])}^{p_c}).
\end{equation}
This implies, for large $t$, 
\begin{equation}
    \|\nabla h\|_{S(L^2,\,[t,t+\tau_0])} < 2K e^{-ct},\quad \tau_0 = \frac{1}{9K^2}.
\end{equation}
Indeed, assume by contradiction that there exists $\tau \in (0,\tau_0]$ such that $\|h\|_{S(L^2,\,[t,t+\tau])} = 2K e^{-ct}$, for fixed $t>0$. Then, by \eqref{stric_K}, 
\begin{equation}
2Ke^{-ct} \leq Ke^{-ct}+2K^2\tau^\frac{1}{2}e^{-ct}+(2K)^{p_c}Ke^{-cp_ct} \leq   \frac{5}{3} Ke^{-ct}+(2K)^{p_c}Ke^{-cp_ct} , 
\end{equation}
which is a contradiction if $t$ is large. Therefore, by decomposing $[t,+\infty) = \displaystyle\bigcup_{j=0}^{\infty}[t+j\tau_0,t+(j+1)\tau_0]$ and using the triangle inequality, we see that \eqref{stric_exp} holds.
\end{proof}

The following lemma is the intercritical version of Lemma \ref{lem_nonl}, and its proof is analogous.

\begin{lemma}[Preliminary estimates, subcritical case]\label{sub_lem_nonl} Let $0 < s_c < 1$ and $I$ be a bounded time interval such that $|I|\leq 1$, and consider $f,g \in S(L^2,\,I)$ such that $\nabla f, \nabla g \in S(L^2,\,I)$. There exists $\alpha >0$ such that the following estimates hold.

For $p > 1$:
\begin{enumerate}[(i)]
\item \label{sub_linear1} $\|\langle\nabla\rangle K(f)\|_{S'(L^2,\,I)} \lesssim |I|^{\alpha} \|\langle\nabla\rangle f\|_{S(L^2,\,I)}$,
\item \label{sub_linear2} $\| K(f)\|_{S'(\dot{H}^{-s_c},\,I)} \lesssim |I|^{\alpha} \|f\|_{S(\dot{H}^{s_c},\,I)}$.
%
\end{enumerate}
For $p > 2$:
\begin{enumerate}[(i)]
\setcounter{enumi}{2}
\item \label{sub_P_grad_r_point} $\begin{aligned}[t]\|\langle\nabla\rangle(R(f)-R(g))\|_{S'(L^2,\,I)} \lesssim 
\|\langle\nabla\rangle(f-g)\|_{S(L^2,\,I)}&\left[\|\langle\nabla\rangle f\|_{S(L^2,\,I)} +\| \langle\nabla\rangle g\|_{S(L^2,\,I)}\right.\\
&+\left.\|\langle\nabla\rangle f\|^{p-1}_{S(L^2,\,I)} +\|\langle\nabla\rangle g\|^{p-1}_{S(L^2,\,I)}\right],
\end{aligned}$
\item \label{sub_P_r_s_point}$\begin{aligned}[t]\|R(f)-R(g)\|_{S'(\dot{H}^{-s_c},\,I)} \lesssim 
\|f-g\|_{S(\dot{H}^{s_c},\,I)}&\left[\|f\|_{S(\dot{H}^{s_c},\,I)} +\| g\|_{S(\dot{H}^{s_c},\,I)}\right.\\
&\left.+\|f\|^{p-1}_{S(\dot{H}^{s_c},\,I)} +\| g\|^{p-1}_{S(\dot{H}^{s_c},\,I)}\right],\end{aligned}$
\end{enumerate}
For $1 < p \leq 2$:
\begin{enumerate}[(i)]
\setcounter{enumi}{4}
%
\item \label{sub_p_grad_r_point} $\begin{aligned}[t]\|\langle\nabla\rangle(R(f)-R(g))\|_{S'(L^2,\,I)} \lesssim 
&\|\langle\nabla\rangle(f-g)\|_{S(L^2,\,I)}\left(\|f\|^{p-1}_{S(\dot{H}^{s_c},\,I)} +\| g\|^{p-1}_{S(\dot{H}^{s_c},\,I)}\right)\\
&+\|f-g\|_{S(\dot{H}^{s_c},\,I)}^{p-1}\left(
\|\langle\nabla\rangle f\|_{S(L^2,\,I)}+\|\langle\nabla\rangle g\|_{S(L^2,\,I)}\right),\end{aligned}$
\item \label{sub_p_r_s_point}$\|R(f)-R(g)\|_{S'(\dot{H}^{-s_c},\,I)} \lesssim 
\|f-g\|_{S(\dot{H}^{s_c},\,I)}\left(\|f\|^{p-1}_{S(\dot{H}^{s_c},\,I)} +\| g\|^{p-1}_{S(\dot{H}^{s_c},\,I)}\right)$.
\end{enumerate}
\begin{proof}
The estimates are very similar as the ones in the proof of the energy-critical case. We use the following classical inequalities
\begin{equation}
    \||a|^{p-1}b\|_{S'(L^2)} \leq \|a\|_{S(\dot{H}^{s_c})}^{p-1}\|b\|_{S(L^2)} \lesssim \|\langle \nabla \rangle a \|_{S(L^2)}^{p-1}\|b\|_{S(L^2)},
\end{equation}
and 
\begin{equation}
    \||a|^{p-1}b\|_{S'(\dot{H}^{-s_c})} \leq \|a\|_{S(\dot{H}^{s_c})}^{p-1}\|b\|_{S(\dot{H}^{s_c})} \lesssim \|\langle \nabla \rangle a\|_{S(L^2)}^{p-1}\|\langle \nabla \rangle b\|_{S(L^2)},
\end{equation}
which can be verified using the pairs $\left(\frac{4(p+1)}{N(p-1)},p+1\right)\in \mathcal{A}_0$, $\left(\frac{2(p-1)(p+1)}{4-(N-2)(p-1)}, p+1\right)\in \mathcal{A}_{s_c}$, and $\left(\frac{2(p-1)(p+1)}{(p-1)(Np-2)-4},p+1\right) \in \mathcal{A}_{-s_c}$, together with Sobolev inequality.
Let us estimate, for example, $\|\nabla (R(f)-R(g))\|_{S'(L^2,I)}$. Write

\begin{align}
	\nabla (R(f) - R(g)) = &\underbrace{p Q^{p-1} \nabla Q (J(Q^{-1} f)-J(Q^{-1}g))}_{(a)} \\
	&+ \underbrace{Q^{p -1}J_z(Q^{-1}f)\nabla f - Q^{p -1}J_z(Q^{-1}g)\nabla g}_{(b)} \\
	&+ \underbrace{Q^{p -1}J_{\bar{z}}(Q^{-1}f)\nabla \bar{f} - Q^{p -1}J_{\bar{z}}(Q^{-1}g)\nabla \bar{g}}_{(c)}\\
	&+\underbrace{Q^{p -2} \nabla  Q J_{z}(Q^{-1}f) f - Q^{p_c -2}\nabla Q J_{z}(Q^{-1}g) g}_{(d)}\\
	&+\underbrace{Q^{p -2} \nabla  Q J_{\bar{z}}(Q^{-1}f) \bar{f} - Q^{p -2}\nabla Q J_{\bar{z}}(Q^{-1}g) \bar{g}}_{(e)}.
\end{align}

Making use of $|\nabla Q| \lesssim Q$ (which follows from Corollary \ref{sub_decay_eigen}), we write $(a)$ as
\begin{equation}
    |(a)| \lesssim Q^{p-1}\int_0^1 |J_z(Q^{-1}(sf+(1-s)g))(f-g) + J_{\bar{z}}(Q^{-1}(sf+(1-s)g))(\overline{f-g})| ds.
\end{equation}
Now, since
\begin{equation}
    |J_z(z_1)-J_z(z_2)| + |J_{\bar{z}}(z_1)-J_{\bar{z}}(z_2)| \lesssim \begin{cases}
    |z_1-z_2|(1 + |z_1|^{p-2}+|z_2|^{p-2}), & p \geq 2,\\
    |z_1-z_2|^{p-1}, & 1 < p < 2,
    \end{cases}
\end{equation}
we have
\begin{equation}
|(a)| \lesssim \begin{cases}
(Q^{p-2}|f|+Q^{p-2}|g|+|f|^{p-1}+|g|^{p-1})|f-g|, & p \geq 2,\\
(|f|^{p-1}+|g|^{p-1})|f-g|, & 1 < p < 2.
\end{cases}
\end{equation}

Thus, since $Q \in \mathcal{S}(\Real^N)$ and $|I|\leq 1$,
\begin{equation}
    \|(a)\|_{S'(L^2,I)} \lesssim \begin{cases}
    (\|f\|_{S(\dot{H}^{s_c}, I)} + \|g\|_{S(\dot{H}^{s_c}, I)}+
    \|f\|_{S(\dot{H}^{s_c},I)}^{p-1}+\|g\|_{S(\dot{H}^{s_c}, I)}^{p-1})\|f-g\|_{S'(L^2, I)}, & p \geq 2,\\
    (\|f\|_{S(\dot{H}^{s_c},I)}^{p-1}+\|g\|_{S(\dot{H}^{s_c}, I)}^{p-1})\|f-g\|_{S'(L^2, I)}, & 1 < p < 2.
    \end{cases}
\end{equation}
We also have
\begin{equation}
    \|(b)\|_{S'(L^2,I)} + \|(d)\|_{S'(L^2,I)} \lesssim
\begin{cases}
\|\nabla(f-g)\|_{S(L^2,I)}(\|\langle\nabla\rangle f\|_{S(L^2,I)}+\|\langle\nabla\rangle f\|_{S(L^2,I)}^{p-1}\\
\quad +\|\langle\nabla\rangle g\|_{S(L^2,I)}+\|\langle\nabla\rangle g\|_{S(L^2,I)}^{p-1}), &p \geq 2,\\
\|\langle\nabla\rangle(f-g)\|_{S(L^2,\,I)}\left(\|f\|^{p-1}_{S(\dot{H}^{s_c},\,I)} +\| g\|^{p-1}_{S(\dot{H}^{s_c},\,I)}\right)\\
\quad+\|f-g\|_{S(\dot{H}^{s_c},\,I)}^{p-1}\left(
\|\langle\nabla\rangle f\|_{S(L^2,\,I)}+\|\langle\nabla\rangle g\|_{S(L^2,\,I)}\right), & 1 < p  < 2,
\end{cases}
\end{equation}
with the same bounds for $(d)$ and $(e)$.
\end{proof}
\begin{re}
We do not employ the same estimates as in \cite{DR_Thre}, since the nonlinearity $|u|^{p-1}u$ is not a polynomial in $(u,\bar{u})$ if $p$ is not an odd integer. Therefore, instead of using $H^s$ estimates, we rely on $S(L^2)$ and $S(\dot{H}^{s_c})$ estimates, that are more suitable for generalizing the result to all dimensions and powers of the nonlinearity.
\end{re}
\begin{re}
We employ a different approach than Li and Zhang \cite{higher_thre}, that divide all estimates in regions where $|f|> W$ or $|f| \leq W$. Instead, we use fractional derivatives to avoid some sublinear estimates, resulting in a simpler proof.
\end{re}
\end{lemma}



\section{Construction of special solutions}

In this section, we construct \textit{special} solutions to the NLS equation  \eqref{NLS}, in the sense that they are on the same energy level of the ground state, converge to the standing wave in $\dot{H}^1$ as $t \to +\infty$, but have kinetic energy different from $\|\nabla Q\|_{L^2}$.

\subsection{Construction of a family of approximate solutions}

We start with a proposition that was first proved by Duyckaerts and Merle in \cite{DM_Dyn}, for $s_c = 1$, and later in \cite{DR_Thre} for the 3d cubic case. We extend here those proofs to the intercritical case. The main difference from the energy-critical case is that $Q$ decays exponentially if $0 < s_c < 1$, so we need to be careful with its spatial decay, as we make use of the estimates of the type $\|Q^{-1}f\|_{L^{\infty}}$. To this end, we make use of the sharp decay estimate for $Q$ given by \eqref{sub_GNN} and of the control on the spatial decay given by Corollary \ref{sub_decay_eigen}.

\begin{prop}\label{family_appr}  Let $0 < s_c \leq 1$ and $A \in \Real$. There exists a sequence ($Z_k^A)_{k \geq 1}$ of functions in $\mathcal{S}(\Real^N)$ such that $Z_1^A = A \mathcal{Y}_+$ and, if $k \geq 1$ and $\mathcal{V}_k^A = \sum_{j=1}^k e^{-je_0 t}Z_j^A$, then as $t \to +\infty$ we have
\begin{equation}\label{eq_sol_v}
\partial_t\mathcal{V}_k^A + \mathcal{LV}_k^A = iR(\mathcal{V}_k^A)+O\left(e^{-\left(k+1\right)e_0t}\right) \text{ in }\mathcal{S}(\Real^N),
\end{equation}
where $\mathcal{L}$ and $R$ are given in Definition \ref{def_op}.
\end{prop}

\begin{proof}
We prove this proposition by induction. For simplicity, we often omit the superscript $A$.

Define $Z_1 = A \mathcal{Y}_+ $ and $\mathcal{V}_1 =\ e^{-e_0t}Z_1$. Thus,
\begin{equation}
\partial_t\mathcal{V}_1 + \mathcal{LV}_1 - iR(\mathcal{V}_1) =- iR(\mathcal{V}_1).
\end{equation}

Note now that $R(f) = Q^{p}J(Q^{-1}f)$, where $J(z) = |1+z|^{p-1}(1+z)-1-\frac{p+1}{2}z -\frac{p-1}{2}\bar{z}$ is real-analytic in the disc $\{z: |z| < 1\}$, and satisfies $J(0) = \partial_{z}J(0) = \partial_{\bar{z}}J(0) = 0$. Write its Taylor expansion as 
\begin{equation}\label{sub_J_exp}
	J(z) = \sum_{i+j \geq 2} a_{ij} z^i \bar{z}^j
\end{equation}

with the uniform convergence of the series and all of its derivatives in the compact disc $\{z: |z| \leq \frac{1}{2}\}$.

Now, if $s_c = 1$, since $Z_1 \in \mathcal{S}(\Real^N)$ and $W$ decays polynomially, we have that $\|W^{-1}Z_1\|_{L^\infty} < +\infty$. For $0 < s_c < 1$, we make use of Corollary \ref{sub_decay_eigen} \eqref{sub_b_inf}, to conclude that $\|Q^{-1}Z_1\|_{L^\infty} < +\infty$. In any case, we can choose $t_0$ such that $|\mathcal{V}_1(t)| \leq \frac{1}{2} Q$, for any $t \geq t_0$. Therefore, for large $t$, we have 
\begin{equation}
	|R(\mathcal{V}_1)| \leq \|Q\|_{L^\infty}^{p} \left( \sum_{i+j \geq 0} |a_{ij}| \frac{1}{2^{i+j}}\right) |Q^{-1}\mathcal{V}_1|^2 = C |Q^{-1}\mathcal{V}_1|^2 . 
\end{equation}

In the same fashion, we can use Leibiniz rule, equation \eqref{sub_J_exp} and items \eqref{sub_a_inf} and \eqref{sub_b_inf} of Corollary \ref{sub_decay_eigen} to bound all the derivatives of $R(\mathcal{V}_1)$. Using that $\mathcal{V}_1 = e^{-e_0t}Z_1$, we conclude that $R(\mathcal{V}_1) = O(e^{-2e_0t})$ in $\mathcal{S}(\Real^N)$. Moreover, by Corollary \ref{sub_decay_eigen}.\eqref{sub_b_inf}, we have 
$ \|Q^{-1}e^{\eta|x|}\partial^\alpha Z_1\|_{L^\infty} < +\infty$.

Now let $k \geq 1$ and assume that $\mathcal{V}_i$ is defined and satisfy \eqref{eq_sol_v} for all $i \leq k$. For $0 < s_c <1$, assume furthermore that, for all $i \leq k$, and all multi-indices $\alpha$,
\begin{equation}\label{sub_Z_ind_hyp}
    \|Q^{-1}e^{\eta|x|}\partial^\alpha Z_i\|_{L^\infty} < +\infty.
\end{equation} Defining
\begin{equation}\label{sub_e_k}
	\epsilon_k = \partial_t \mathcal{V}_k +\mathcal{L}\mathcal{V}_k-iR(\mathcal{V}_k), 
\end{equation}
note that 
\begin{equation}
	\partial_t \mathcal{V}_k = \sum_{j=1}^k (-je_0) e^{-je_0 t}Z_k,
\end{equation}

so that \eqref{sub_e_k} can be written as

\begin{equation}\label{sub_e_k2}
	\epsilon_k(x,t) = \sum_{j=1}^k e^{-je_0 t}\left(-je_0 Z_k(x) +\mathcal{L}Z_k(x)\right)-iR(\mathcal{V}_k(x,t)). 
\end{equation}

Recall that, for all $k$, $Z_k \in \mathcal{S}(\Real^N)$. If $0 < s_c < 1$, we also have \eqref{sub_Z_ind_hyp}. Therefore, for large $t$, and all $x$, $|\mathcal{V}_k(x,t)| \leq \frac{1}{2} Q(x)$. Writing $R(\mathcal{V}_k) = Q^{p}J(Q^{-1}\mathcal{V}_k)$ and using again the expansion \eqref{sub_J_exp}, we get by \eqref{sub_e_k2} that there exist functions $F_j \in \mathcal{S}(\Real^N)$ such that for large $t$
\begin{equation}
	\epsilon_k(x,t) = \sum_{j=1}^{k+1}e^{-je_0 t}F_j(x)  + O(e^{-e_0(k+2)t}) \text{ in }\mathcal{S}(\Real^N).
\end{equation}

By \eqref{eq_sol_v}, we conclude that $F_j = 0$ for $j \leq k$, which shows
\begin{equation}\label{decay_e}
	\epsilon_k (x,t) = e^{-(k+1)e_0 t}F_{k+1} + O(e^{-(k+2)e_0 t}).
\end{equation}

Noting that $(k+1)e_0$ is not in the spectrum of $\mathcal{L}$, define $Z_{k+1} = -(\mathcal{L}+(k+1)e_0)^{-1}F_{k+1}$, which belongs to $\mathcal{S}$ (see Appendix). Moreover, if $0 < s_c < 1$, $Z_{k+1}$ satisfies \eqref{sub_Z_ind_hyp} with $k$ replaced by $k+1$. By definition, we have $\mathcal{V}_{k+1} = \mathcal{V}_k + e^{-(k+1)e_0 t}Z_{k+1}$. Furthermore,

\begin{equation}
\epsilon_{k+1} = \epsilon_k - e^{-(k+1)e_0 t} F_{k+1}-i(R(\mathcal{V}_{k+1})-R(\mathcal{V}_k)).
\end{equation}

By \eqref{decay_e}, $\epsilon_k - e^{-(k+1)e_0 t} F_{k+1} =  O(e^{-(k+2)e_0 t})$. Writing again $R(f) = Q^{p}J(Q^{-1}f)$, and using the expansion \eqref{sub_J_exp}, we conclude that $R(\mathcal{V}_{k+1})-R(\mathcal{V}_k) =  O(e^{-(k+2)e_0 t})$, completing the proof.
\end{proof}

\subsection{Contraction argument near an approximate solution}

We now prove the key result of this {section}. The propositions for the energy-critical and for the intercritical cases are stated separately.

\subsubsection{Energy-critical case}

We only treat here the case $N \geq 6$, as in the lower-dimensional cases this result is proved in \cite{DM_Dyn}. The main difference here from \cite{DM_Dyn} is that $0 < p_c-1 < 1$ if $N > 6$, so that the nonlinearity is no longer $C^2$, and its derivative is only H\"older-continuous of order $p_c-1$. This introduces difficulties, as the control of the convergence of $\nabla U^A$ to $\nabla W$ is not enough to close the contraction argument, and we need to ensure that the higher order terms $D^{\varepsilon}(U^A-W-\mathcal{V}_k)$ converge faster to 0, for small $\varepsilon>0$. The fractional derivative $D^\varepsilon$ is needed here to avoid certain end-point Strichartz estimates, which are not available for any combination of $\dot{H}^1$-admissible and $\dot{H}^{-1}$-admissible pairs.

\begin{prop}\label{exist_UA} Let $N \geq 6$. There exists $k_0 > 0$ such that for any $k \geq k_0$, there exists $t_k \geq 0$ and a solution $U^A$ to \eqref{NLS} such that for $t \geq t_k$ and $l( k) = {\left\lceil \frac{N-2}{4}k +\frac{N-6}{4}\right\rceil}$ (where $\lceil x \rceil$ denotes the least integer bigger or equal to x),
\begin{equation}
\begin{aligned}\label{bound_UA}
	\|D^\varepsilon(U^A-W-\mathcal{V}_{l(k)}^A)\|_{S(\dot{H}^{1-\varepsilon},\,[t,+\infty))} &\leq e^{-(k+\frac{1}{2})\frac{N-2}{4}  e_0 t}\text{ and}\\
	\|\nabla(U^A-W-\mathcal{V}_{l(k)}^A)\|_{S(L^2,\,[t,+\infty))} &\leq e^{-(k+\frac{1}{2})e_0 t}.
\end{aligned}
\end{equation}
Furthermore, $U^A$ is the unique solution to \eqref{NLS} satisfying \eqref{bound_UA} for large $t$. Finally, $U^A$ is independent of $k$ and satisfies for large $t$,
\begin{equation}\label{bound_UA2}
	\|U^A(t) - W - Ae^{-e_0t}\mathcal{Y}_+\|_{\dot{H}^1} \leq e^{-2e_0 t}.
\end{equation}
\end{prop}
\begin{proof}
Since $A \in \Real$ is fixed in the proof, we omit the superscripts $A$. Define
\begin{equation}
	h = U^A - W - \mathcal{V}_{l(k)}^A,
\end{equation}
so that $U^A$ is a solution to \eqref{NLS}, if and only if, $h$ satisfies
\begin{equation}
	i\partial_t h + \Delta h = - K(h) - (R(\mathcal{V}_{l(k)}+h)-R(\mathcal{V}_{l(k)}))+i\epsilon_{l(k)}, 
\end{equation}
where $\epsilon_{l(k)} = O\left(e^{-\left(l(k)+1\right)e_0t}\right) \text{ in }\mathcal{S}(\Real^N)$ for all $k \geq 0$. Therefore, the existence of $U^A$ can be written as the fixed-point problem
\begin{equation}
	h(t) = \mathcal{M}(h)(t),
\end{equation}
where
\begin{equation}
	\mathcal{M}(h)(t) = -i\int_t^{+\infty} e^{i(t-s)\Delta}\left[- K(h) - (R(\mathcal{V}_{l(k)}+h)-R(\mathcal{V}_{l(k)}))+i\epsilon_{l(k)}\right]\,ds.
\end{equation}

Let first $N > 6$. We will show that $\mathcal{M} $ is a contraction on $B$ defined by
\begin{align}
	    B = B(k,t_k) := &\left\{h \in E:\|h\|_E \leq 1\right\},\\
		E = E(k,t_k) := &\left\{
		h \in 
		S(\dot{H}^1,\,[t_k,+\infty)), D^\varepsilon h \in S(\dot{H}^{1-\varepsilon}\,[t_k, +\infty)), 
		\right.\\ 
		&\quad\quad\quad\left.\nabla h \in S(L^2,\,[t_k,+\infty)):\|h\|_E < +\infty\right\},\\
		\|h\|_E := &\sup_{t\geq t_k} e^{\left(k+\frac{1}{2}\right)\frac{N-2}{4} e_0t}\|D^\varepsilon h\|_{S(\dot{H}^{1-\varepsilon},\,[t,+\infty))}+\sup_{t\geq t_k} e^{\left(k+\frac{1}{2}\right)e_0 t}\|\nabla h\|_{S(L^2,\,[t,+\infty))},
	\end{align}
	equipped with the metric
	\begin{equation}
	\rho(u,v) = \sup_{t\geq t_k} e^{\left(k+\frac{1}{2}\right)\frac{N-2}{4}e_0t}\|D^\varepsilon(u-v)	\|_{S(\dot{H}^{1-\varepsilon},\,[t,+\infty)).}
	\end{equation}
	

Let $\{h_n\} \subset B$ and $h \in S(\dot{H}^{1-\varepsilon},I)$ with $\rho(h_n,h) \to 0$. By reflexivity and uniqueness between the weak and strong limits, $h \in B$. This shows that $(B,\rho)$ is a complete metric space. We will show that $\mathcal{M}(B) \subset B$ and that $\mathcal{M}$ is a contraction.

By Strichartz estimates, there exists $C^*>0$ such that
\begin{align}\label{contract1}
	\| \nabla\mathcal{M}(h)\|_{S(L^2,\,[t,+\infty))} \leq C^*&\left[\|\nabla K(h)\|_{S'(L^2,\,[t,+\infty))}
	\right.\\
 &\left.+ \| \nabla(R(\mathcal{V}_{l(k)}+h)-R(\mathcal{V}_{l(k)}))\|_{S'(L^2,\,[t,+\infty))}
	+ \|\nabla \epsilon_{l(k)}\|_{S'(L^2,\,[t,+\infty))}\right],
\end{align}

\begin{align}\label{contract2}
	\| D^\varepsilon\mathcal{M}(h)\|_{S(\dot{H}^{1-\varepsilon},\,[t,+\infty))} \leq C^*&\left[\| D^\varepsilon K(h)\|_{S'(\dot{H}^{-(1-\varepsilon)},\,[t,+\infty))}\right.\\
	&+ \|D^\varepsilon (R(\mathcal{V}_{l(k)}+h)-R(\mathcal{V}_{l(k)}))\|_{S'(\dot{H}^{-(1-\varepsilon)},\,[t,+\infty))}\\
	&\quad\quad\quad\quad\quad\quad\quad\quad\quad+ \left.\|D^\varepsilon \epsilon_{l(k)}\|_{S'(\dot{H}^{-(1-\varepsilon)},\,[t,+\infty))}\right],
	\end{align}

and
\begin{align}\label{contract3}
	\| D^\varepsilon(\mathcal{M}(g)-\mathcal{M}(h))\|_{S(\dot{H}^{1-\varepsilon},\,[t,+\infty))} \leq C^*&\left[\| D^\varepsilon K(g-h)\|_{S'(\dot{H}^{-(1-\varepsilon)},\,[t,+\infty))}\right.\\
	&+ \left.\|D^\varepsilon (R(\mathcal{V}_{l(k)}+g)-R(\mathcal{V}_{l(k)}+h))\|_{S'(\dot{H}^{-(1-\varepsilon)},\,[t,+\infty))}\right].
	\end{align}
To finish the argument, we need the following estimates.
\begin{lemma}\label{contr_estim} For every $\eta >0$, there exists $\tilde{k}(\eta) >0$ such that, if $k \geq \tilde{k}(\eta)$, then for any $g$, $h \in B$ the following inequalities hold for all $t \geq t_k$
\begin{enumerate}[(i)]
	\item\label{contr_lin} $\|\nabla K(h)\|_{S'(L^2,\,[t,+\infty))} \leq \eta e^{-\left(k+\frac{1}{2}\right)e_0 t}\|h\|_E$,
	\item\label{contr_R} $\| \nabla (R(\mathcal{V}_{l(k)}+h)-R(\mathcal{V}_{l(k)}))\|_{S'(L^2,\,[t,+\infty))} \leq C_k
	e^{-\left(k+\frac{1}{2}+\frac{4}{N-2}\right)e_0 t}$,
	\item\label{contr_lin2} $\|D^\varepsilon K(h)\|_{S'(\dot{H}^{-(1-\varepsilon)},\,[t,+\infty))} \leq \eta e^{-\left(k+\frac{1}{2}\right)\frac{N-2}{4}e_0 t}\rho(h,0)$,
	\item\label{contr_R2} $\| D^\varepsilon (R(\mathcal{V}_{l(k)}+g)-R(\mathcal{V}_{l(k)}+h))\|_{S'(\dot{H}^{-(1-\varepsilon)},\,[t,+\infty))} \leq C_k
	e^{-\left(k+\frac{1}{2}+\frac{16}{(N-2)^2}\right)\frac{N-2}{4}e_0 t}\rho(g,h)$,
	\item\label{contr_e} $\|\nabla\epsilon_{l(k)}\|_{S'(L^2,\,[t,+\infty))}+\|D^\varepsilon\epsilon_{l(k)}\|_{S'(\dot{H}^{-(1-\varepsilon)},\,[t,+\infty))} \leq C_k e^{-\left({k}+1\right){ \frac{N-2}{4} }e_0 t}$.
\end{enumerate}
\end{lemma}

Indeed, assuming this lemma, choosing first $\eta > 0$ small enough, and then a large enough $t_k$, we see by \eqref{contract1} and \eqref{contract2} that $\mathcal{M}(B) \subset B$. Moreover, by \eqref{contract3}, $\mathcal{M}$ is a contraction on $B$. Thus, for every $k \geq k_0 = \tilde{k}(\eta)$, there is a unique solution $U^A$ to \eqref{NLS} satisfying \eqref{bound_UA} for $t \geq t_k$. Note that the uniqueness still holds in the class of solutions to \eqref{NLS} satisfying \eqref{bound_UA} for $t \geq t'_k$, where $t_k' \geq t_k$. Thus, uniqueness of the solutions to \eqref{NLS} shows that $U^A$ does not depend on $k$.

It remains to show Lemma \ref{contr_estim}. By Lemma \ref{lem_nonl}.\eqref{linear1}, if $\tau_0 > 0$ and $t \geq t_k$, then
\begin{equation}
	\|\nabla K(h)\|_{S'(L^2,\,[t,t+\tau_0])} \leq C \tau_0^{\frac{1}{2}}e^{-\left(k+\frac{1}{2}\right)e_0 t}\|h\|_E.
\end{equation}
Summing up this equation at times $t_j = t+j\tau_0$, and using the triangle inequality, we get a geometric series, whose sum is
\begin{equation}
	\|\nabla K(h)\|_{S'(L^2,\,[t,+\infty))} \leq C \frac{\tau_0^{\frac{1}{2}}}{1-e^{-\left(k+\frac{1}{2}\right)\frac{4}{N-2}e_0 \tau_0}} e^{-\left(k+\frac{1}{2}\right)e_0 t}\|h\|_E.
\end{equation}
Choosing $\tau_0$ and $k_0$ such that $\tau_0^\frac{1}{2} = \frac{\eta}{2 C}$  and $e^{-\left(k+\frac{1}{2}\right)e_0 \tau_0} \leq \frac{1}{2}$ , we get the estimate (\ref{contr_lin}) of Lemma \ref{contr_estim}. The estimate \eqref{contr_lin2} follows analogously from Lemma \ref{lem_nonl} \eqref{linear2}.

We now turn to the item (\ref{contr_R}). By Lemma \ref{lem_nonl} \eqref{grad_r_point}, we have
\begin{align}
	\|\nabla (R(\mathcal{V}_{l(k)}+h)-R(\mathcal{V}_{l(k)}))\|_{S'(L^2,\,[t,t+1])} \leq (A)\|\nabla h\|_{S(L^2,\,[t,t+1])}+(B)\|D^\varepsilon h\|^{\frac{4}{N-2}}_{S(\dot{H}^{1-\varepsilon},\,[t,t+1])},	
\end{align}
where 
$$
(A) \lesssim \|\nabla \mathcal{V}_{l(k)}\|_{S(L^2,\,[t,t+1])}^{\frac{4}{N-2}}+\|\nabla h\|^{\frac{4}{N-2}}_{S(L^2,\,[t,t+1])}
$$ 
and 
$$
(B) \lesssim \|\nabla \mathcal{V}_{l(k)}\|_{S(L^2,\,[t,t+1])}+\|\nabla h\|_{S(L^2,\,[t,t+1])}.
$$

By the explicit form of $\mathcal{V}_k$ and the fact that $h \in B$, we get
\begin{equation}
	(A)+(B) \leq C_k e^{-e_0\frac{4}{N-2}t} .
\end{equation}
Therefore, 
\begin{align}
\|(R(\mathcal{V}_{l(k)}+h)-R(\mathcal{V}_{l(k)}))\|_{S'(L^2,\,[t,t+1])} 
&\leq C_k e^{-e_0\frac{4}{N-2}t} (\|\nabla h\|_{S(L^2,\,[t,t+1])}+ 
\|D^\varepsilon h\|^{\frac{4}{N-2}}_{S(\dot{H}^{1-\varepsilon},\,[t,t+1])}) \\&\leq C_k 
e^{-\left(k+\frac{1}{2}+\frac{4}{N-2}\right)e_0t}.
\end{align}

Since $h \in B$, the triangle inequality and the sum of the resulting geometric series gives us \eqref{contr_R}. As for {the} item \eqref{contr_R2}, it follows analogously from Lemma \ref{lem_nonl} \eqref{r_point}. The estimate (\ref{contr_e}) of Lemma \ref{contr_estim} follows immediately from \eqref{decay_e} and from the bound
\begin{equation}
    l(k) +1\geq \left(k+1\right)\frac{N-2}{4}.
\end{equation}

Finally, given that $U^A = W+\mathcal{V}_k +h$ with $h \in B$, we see that, for large $k$,  
\begin{equation}
	\|\nabla h(t)\|_{L^2} \leq e^{-\frac{5}{2}e_0t}\|h\|_E \leq e^{-\frac{5}{2}e_0t},
\end{equation}
and recalling the definition of $\mathcal{V}_k$ given in Proposition \ref{family_appr}, we have, for all $k$,
\begin{equation}
	\mathcal{V}_{l(k)} = Ae^{-e_0 t}\mathcal{Y}_+ + O(e^{-2e_0t}) \text{ in } \mathcal{S}(\Real^N),
\end{equation}
which proves \eqref{bound_UA2}, and finishes the case $N > 6$.

For the case $N = 6$, note that $\frac{N-2}{4} = 1$, so that, by Sobolev embedding, only the space $S(L^2,I)$ is enough for the contraction argument. Therefore, defining the space $B$ as
\begin{align}
        B = B(k,t_k) := &\left\{h \in E:\,\,\|h\|_E \leq 1\right\},\\
		E = E(k,t_k) := &\left\{h \in C_t\dot{H}^1([t_k,+\infty))\cap S(\dot{H}^1,\,[t_k,+\infty)),\right.\\ &\left.\quad\quad\quad\nabla h \in S(L^2,\,[t_k,+\infty)):\,\, \|h\|_E < +\infty\right\},\\	
		\|h\|_E := &\sup_{t\geq t_k} e^{\left(k+\frac{1}{2}\right)e_0 t}\|\nabla h\|_{S(L^2,\,[t,+\infty))},
	\end{align}
	equipped with the metric
	\begin{equation}
	\rho(u,v) = \sup_{t\geq t_k} e^{\left(k+\frac{1}{2}\right)e_0t}\|\nabla(u-v)	\|_{S(L^2,\,[t,+\infty))}.
	\end{equation}
By Lemma \ref{lem_nonl}, the analogous estimates of Lemma \ref{contr_estim} hold. Hence, the conclusion of Proposition \ref{exist_UA} also holds for $N = 6$, finishing its proof.

\end{proof}

\subsubsection{Intercritical case}

\begin{prop}\label{sub_exist_UA} There exists $k_0 > 0$ such that for any $k \geq k_0$, there exists $t_k \geq 0$ and a solution $U^A$ to \eqref{sub_NLS} such that for $t \geq t_k$ and $l(k) = \max\{\left\lceil \frac{k+1}{p-1}-1\right\rceil,k\}$, we have
\begin{equation}\label{sub_bound_UA_S}
\|U^A-e^{it}Q-e^{it}\mathcal{V}_{l(k)}^A\|_{S(\dot{H}^{s_c},[t,+\infty))} \leq e^{-(k+\frac{1}{2})\max\{\frac{1}{p-1},1\}  e_0 t}
\end{equation}
and
\begin{equation}\label{sub_bound_UA}
\|\langle\nabla\rangle(U^A-e^{it}Q-e^{it}\mathcal{V}_{l(k)}^A)\|_{Z(t,+\infty)} \leq e^{-(k+\frac{1}{2})e_0 t}.
\end{equation}

Furthermore, $U^A$ is the unique solution to the NLS equation \eqref{sub_NLS} satisfying \eqref{sub_bound_UA} for large $t$. Finally, $U^A$ is independent of $k$ and satisfies for large $t$,
\begin{equation}\label{sub_bound_UA2}
	\|U^A(t) - e^{it}Q - Ae^{-e_0t+it}\mathcal{Y}_+\|_{H^1} \leq e^{-2e_0 t}.
\end{equation}
\end{prop}
In view of Lemma \ref{sub_lem_nonl}, the proof of Proposition \ref{sub_exist_UA} is essentially the same as in the energy-critical case, and we state Lemma \ref{sub_contr_estim} below for completeness. Note that \eqref{sub_bound_UA_S} is a consequence of \eqref{sub_bound_UA} in the case $p \geq 2$, due to the Sobolev inequalities. 

\begin{lemma}\label{sub_contr_estim} For every $\eta >0$, there exists $\tilde{k}(\eta) >0$ such that, if $k \geq \tilde{k}(\eta)$, then for any $g$, $h \in B$ the following inequalities hold for all $t \geq t_k$    
\begin{enumerate}[(i)]
	\item\label{sub_contr_lin} $\|\nabla K(h)\|_{S'(L^2,\,[t,+\infty))} \leq \eta e^{-\left(k+\frac{1}{2}\right)e_0 t}\|h\|_E$,
	\item\label{sub_contr_R} $\| \nabla (R(\mathcal{V}_{l(k)}+h)-R(\mathcal{V}_{l(k)}))\|_{S'(L^2,\,[t,+\infty))} \leq C_k e^{-\min\{p-1,1\}e_0 t}
	e^{-\left(k+\frac{1}{2}\right)e_0 t}$,
	\item\label{sub_contr_lin2} $\| K(h)\|_{S'(\dot{H}^{-s_c},\,[t,+\infty))} \leq \eta e^{-\left(k+\frac{1}{2}\right)\max\{\frac{1}{p-1},1\}e_0 t}\rho(h,0)$,
	\item\label{sub_contr_R2} $\| R(\mathcal{V}_{l(k)}+g)-R(\mathcal{V}_{l(k)}+h)\|_{S'(\dot{H}^{-s_c},\,[t,+\infty))} \leq C_k e^{-\min\{p-1,1\}e_0t}
	e^{-\left(k+\frac{1}{2}\right)\max\{\frac{1}{p-1},1\}e_0 t}\rho(g,h)$,
	\item\label{sub_contr_e} $\|\nabla\epsilon_{l(k)}\|_{S'(L^2,\,[t,+\infty))}+\|\epsilon_{l(k)}\|_{S'(\dot{H}^{-s_c},\,[t,+\infty))} \leq C_k e^{-\left({k}+1\right)\max\{\frac{1}{p-1},1\}e_0 t}$.
\end{enumerate}
\end{lemma}

\section{Modulation}

Throughout the rest of the paper, we write, for $0 < s_c \leq 1$,

\begin{equation}
	d(f) = \Big|\|\nabla f\|_{L^2}-\|\nabla Q\|_{L^2}\Big|.
\end{equation}

If $u$ is a solution to $\eqref{sub_NLS}$ or to $\eqref{NLS}$ and if there is no risk of ambiguity, we also write 
\begin{equation}
    d(t) = d(u(t)).
\end{equation}

\subsection{Energy-critical case}
The variational characterization of $W$, see \cites{Aub76,Tal76,Lio85}, shows that, if $E(f) = E(W)$, then
\begin{equation}
	\inf_{\substack{x \in \Real^N\\ \lambda >0 \\ \theta \in \Real}}\|f_{[x,\lambda,\theta]}-W\|_{\dot{H}^1} \leq \epsilon(d(f))
\end{equation}
with $\epsilon = \epsilon(\delta)$ such that
\begin{equation}
	\lim_{\delta \to 0^+}\epsilon(\delta) = 0.
\end{equation}

The goal of this section is to construct modulation parameters $x_0, \lambda_0$ and $\theta_0$ such that the quantity $d(f)$ controls linearly $\|f_{[x_0,\lambda_0,\theta_0]}-W\|_{\dot{H}^1}$ as well as the parameters and its derivatives.

By making use of the Implicit Function Theorem, we can construct appropriate modulation parameters. The proof of the next two lemmas is very similar to the ones in \cite{DM_Dyn}, with the introduction of a translation parameter, and will be given in Appendix.

\begin{lemma}\label{mod_1}
There exist $\delta_0 > 0$ and a positive function $\epsilon(d)$ defined for $0 < d < \delta_0$, with $\epsilon(d)\to 0$ as $d\to 0$, such that, for all $f \in \dot{H}^1$ satisfying $E(f) = E(W)$ and $d(f) < \delta_0$, there exist $(x,\lambda,\theta)$ such that 
\begin{align}
	\|f_{[x,\lambda,\theta]}-W\|_{\dot{H}^1} \leq \epsilon(d(u)),\\
	f_{[x,\lambda,\theta]} \perp \text{span}\{\nabla W, iW, \Lambda W\}.
\end{align}

The parameters $(x,\lambda,\theta)$ are unique in $\Real^N \times \Real/2\pi \mathbb{Z} \times \Real_+$ and the mapping $u \mapsto (x,\lambda,\theta)$ is $C^1$.
\end{lemma}

Let $u$ be a solution to \eqref{NLS} and $I$ be a time interval such that 
\begin{equation}
    d(t) < \delta_0 \text{ for all } t \in I.
\end{equation}

For each $t \in I$, choose the parameters $(x(t),\lambda(t),\theta(t))$ according to Lemma \ref{mod_1}. Write
\begin{equation}\label{decomp_alpha}
	u_{[x(t),\lambda(t),\theta(t)]}(t) = (1+\alpha(t))W+ h(t),
\end{equation}
where
\begin{equation}
	\alpha(t) = \frac{1}{\|W\|_{\dot{H}^1}}(u_{[x(t),\lambda(t),\theta(t)]},W)_{\dot{H}^1}-1.
\end{equation}
Note that $\alpha(t)$ is chosen so that $h(t) \in G^\perp$. By Lemma \ref{mod_1} and a standard regularization argument (see, for instance \cite{MM01}*{Lemma 4} for details in a similar context), the map $t \mapsto (x(t),\lambda(t),\theta(t))$ is $C^1$. We are now able to prove estimates on the modulation.

\begin{lemma}\label{mod_2}
Let $u$ be a solution to \eqref{NLS} satifying $E(u_0) = E(W)$. Taking a smaller $\delta_0$, if necessary, the following estimates hold on $I$:
\begin{equation}
	\label{lem_equiv_1}|\alpha(t)| \approx \|h(t)\|_{\dot{H}^1} \approx d(t) < \delta_0
\end{equation}
\begin{equation}
	\label{lem_equiv_2}|\alpha'(t)| + |x'(t)| + |\theta'(t)| + \left|\frac{\lambda'(t)}{\lambda(t)}\right| \lesssim \lambda^2(t) d(t).
\end{equation}
\end{lemma}

In the next two sections, we mainly consider \textit{radial} solutions to \eqref{NLS}. Since the translation parameter is not needed, we write
\begin{equation}
    f_{[\lambda_0,\theta_0]}(x) = e^{i\theta_0}\frac{1}{\lambda_0^\frac{N-2}{2}}f\left(\frac{x}{\lambda_0}\right).
\end{equation}
In some results regarding compactness, the parameter $\theta_0$ can also be omitted. In this case, we write
\begin{equation}
    f_{[\lambda_0]}(x) = \frac{1}{\lambda_0^\frac{N-2}{2}}f\left(\frac{x}{\lambda_0}\right).
\end{equation}

For solutions to \eqref{sub_NLS}, since the scaling parameter is fixed (see Remark \ref{sub_scaling}), we use the notation
\begin{equation}
    f_{[x_0,\theta_0]}(x) = e^{i\theta_0}f(x+x_0).
\end{equation}

When $\theta_0$ can be omitted, we write
\begin{equation}
    f_{[x_0]}(x) = f(x+x_0).
\end{equation}

\subsection{Intercritical case}

For $0 < s_c < 1$, the variational characterization of $Q$ \cite{Lio84} shows that, if $M(f)= M(Q)$ and $E(f) = E(Q)$, then
\begin{equation}\label{sub_var_car_H1}
	\inf_{\substack{x \in \Real^N\\ \theta \in \Real}}\|f_{[x,\theta]}-Q\|_{H^1} \leq \epsilon(d(f)),
\end{equation}

with
\begin{equation}
	\lim_{\delta \to 0^+}\epsilon(\delta) = 0.
\end{equation}

As before, the goal here is to construct modulation parameters $x_0$ and $\theta_0$ such that the quantity $d(f)$ controls linearly $\|f_{[x_0,\theta_0]}-Q\|_{H^1}$, as well as the parameters and its derivatives. We follow mainly \cite{DR_Thre} here, and the proofs for the next two lemmas are almost identical, thus, we omit them.

\begin{lemma}\label{sub_mod_1}
There exist $\delta_0 > 0$ and a positive function $\epsilon(d)$ defined for $0 < d < \delta_0$, with $\epsilon(d)\to 0$ as $d\to 0$, such that, for all $f \in H^1$ satisfying $M(f) = M(Q)$, $E(f) = E(Q)$ and $d(f) < \delta_0$, there exist $(x,\theta)$ such that 
\begin{align}
	\|f_{[x,\theta]}-Q\|_{H^1} \leq \epsilon(d(f)),\\
	f_{[x,\theta]} \perp \text{span}\{\nabla Q, iQ\}.
\end{align}

The parameters $(x,\theta)$ are unique in $\Real^N \times \Real/2\pi \mathbb{Z}$ and the mapping $u \mapsto (x,\theta)$ is $C^1$.
\end{lemma}

Let $u$ be a solution to \eqref{sub_NLS} and $I$ be a time interval such that $d(t) := d(u(t)) < \delta_0$ for all $t \in I$. For each $t \in I$, choose the parameters $(x(t),\theta(t))$ according to Lemma \ref{sub_mod_1}. Write
\begin{equation}\label{sub_decomp_alpha}
	e^{-it}u_{[x(t),\theta(t)]}(t) = (1+\alpha(t))Q+ h(t),
\end{equation}
where
\begin{equation}
	\alpha(t) = \frac{\Re (e^{-it}\int Q^p u_{[x(t),\theta(t)]} ) }{\|Q\|^{p+1}_{L^{p+1}}}-1.
\end{equation}
Note that $\alpha(t)$ is chosen so that $h(t) \in G^\perp$. Recall that the parameters $(x,\theta)$ are $C^1$.
We are now able to prove estimates  on the modulation.

\begin{lemma}\label{sub_mod_2}
Let $u$ be a solution to \eqref{sub_NLS} satisfying $M(u_0) = M(Q)$ and $E(u_0) = E(Q)$. Taking a smaller $\delta_0$, if necessary, the following estimates hold on $I$:
\begin{align}
	\label{sub_lem_equiv_1}|\alpha(t)| \approx \|h(t)\|_{H^1} \approx \left|\int Q h_1\right|  \approx d(t)\\
	\label{sub_lem_equiv_2}|\alpha'(t)| \approx |x'(t)| \approx |\theta'(t)| \lesssim d(t).
\end{align}
\end{lemma}


We finish this section with a lemma that will be useful in later sections.
\begin{lemma}\label{sub_lem_exp_delta}
Let $u$ be a solution to \eqref{sub_NLS} such that $M(u_0) = M(Q)$ and $E(u_0) = E(Q)$. Assume that $u$ is defined on $[0,+\infty)$ and that there exists $c>0$ such that
\begin{equation}\label{sub_exp_int_delta}
	\int_t^{+\infty} d(s) ds \lesssim e^{-ct}.
\end{equation}
Then there exist $(x_0,\theta_0)$ such that
\begin{equation}
	\|u_{[x_0,\theta_0]} - e^{it}Q\|_{H^1} \lesssim e^{-ct}.
\end{equation}
\end{lemma}
\begin{proof}\textit{Step 1. Convergence of $\delta(t)$.} We claim that 
\begin{equation}\label{sub_delta_to_0}
\lim_{t \to +\infty} d(t) = 0.
\end{equation}
To prove this, first note that \eqref{sub_exp_int_delta} implies that there exists a sequence $\{t_n\}$ with $t_n \to +\infty$ such that
\begin{equation}\label{sub_dtn_0}
	\lim_{n \to +\infty} d(t_n) = 0.
\end{equation}
Suppose now, by contradiction, that \eqref{sub_delta_to_0} does not hold. In this case, we can find another sequence $\{t_n'\}$ and $0 < \epsilon_1 < \delta_0$ such that
\begin{equation}
	\begin{gathered}
		t_n < t_n' \quad \forall n\\
		\label{sub_dtnl} d(t_n') = \epsilon_1,\\ 
	\end{gathered}
\end{equation}
and
\begin{equation}
	d(t) < \epsilon_1 \quad \forall t \in [t_n, t_n').
\end{equation}

Since $\epsilon_1 < \delta_0$, the parameter $\alpha(t)$ is well-defined on $[t_n, t_n')$. By Lemma \ref{sub_mod_2}, 
$|\alpha'(t)| \lesssim d(t)$, so that $\int_{t_n}^{t_{n}'}|\alpha'(t)|dt \lesssim e^{-ct_n}$, by \eqref{sub_exp_int_delta}. Therefore,
\begin{equation}
    \lim_{n\to \infty}|\alpha(t_n)-\alpha(t_n')| = 0.
\end{equation}
Since, by Lemma \eqref{sub_mod_2}, $|\alpha(t)| \approx d(t)$, we get that $\alpha(t_n)$ tends to $0$, which contradicts \eqref{sub_dtnl} and proves \eqref{sub_delta_to_0}.
Recalling the decomposition \eqref{sub_decomp_alpha}, to conclude the proof of Lemma \ref{sub_lem_exp_delta}, it is sufficient to prove that there exists $(x_0,\theta_0)$ such that
\begin{equation}
    d(t)+|\alpha(t)|+\|h(t)\|_{\dot{H}^1}+|x(t)-x_0|+|\theta(t)-\theta_0| \lesssim e^{-ct}.
\end{equation}

By Lemma \ref{sub_mod_2}, $|\alpha(t)| \approx d(t) \to 0$ as $t \to +\infty$, and hence,  
\begin{equation}
    |\alpha(t)| \leq \int_t^{+\infty}|\alpha'(s)|ds \lesssim \int_t^{+\infty}d(s)ds \lesssim e^{-ct},
\end{equation}
since $|\alpha'(t)| \approx d(t)$. Again by Lemma \ref{sub_mod_2}, $d(t)\approx \|h(t)\|_{\dot{H}^1} \approx |\alpha(t)|$, we get the bounds on $d(t)$ and $\|h(t)\|_{\dot{H}^1}$.
To obtain the bounds on $x(t)$ and $\theta(t)$, it is sufficient to recall that Lemma \ref{sub_mod_2} says $|x'(t)|+|\theta'(t)| \lesssim d(t) \lesssim e^{-ct}$.
\end{proof}



\section{Convergence to the standing wave above the ground state}

\subsection{Energy-critical case}

In this and the next sections, we prove that radial solutions to \eqref{NLS} on the same energy level as $W$ that do not blow-up in finite positive time (and have finite mass), and that do not scatter forward in time, must converge exponentially to $W$ as $t \to +\infty$. We follow closely \cite{DM_Dyn}, and give the proofs in Appendix.

\begin{prop}\label{prop_supercrit}
Let $u$ be a solution to \eqref{NLS} defined on $[0,+\infty)$ satisfying
\begin{equation}\label{supercrtit}
    E(u_0) = E(W) \quad\text{and}\quad \|u_0\|_{\dot{H}^1} > \|W\|_{\dot{H}^1}.
\end{equation}
Assume, furthermore, that 
$u_0$ is radial and belongs to $L^2(\Real^N)$. Then there exist $(\lambda_0,\theta_0)$ and $c>0$ such that 
\begin{equation}\label{exp_conv_W_super}
    \|u-W_{[\lambda_0,\theta_0]}\|_{\dot{H}^1} \lesssim  e^{-ct}.
\end{equation}
Moreover, $u$ blows up in finite negative time.
\end{prop}

We will work with a truncated variance. Consider a radial function $\phi \in C^\infty_0(\Real^N)$ such that
\begin{equation}
  \quad \phi(r) \geq 0  \quad \forall r \geq 0,
\end{equation}
\begin{equation}
    \phi(r) = 
    \begin{cases} 
    r^2, &r \leq 1,\\
    0, & r \geq 3,
    \end{cases}
\end{equation}

and
\begin{equation}
    \frac{d^2\phi}{dr^2}(r)\leq 2, r \geq 0.
\end{equation}

Define $\phi_R(x) = R^2\phi(\frac{x}{R})$ and 
\begin{equation}
    F_R(t) = \int \phi_R |u(t)|^2.
\end{equation}
By virial identities, if $E(u_0) = E(W)$, we have
\begin{align}
\label{truncated_virial_1}F_R'(t) = 2 \Im \int \nabla \phi \cdot \nabla u \,\bar{u}\\
\label{truncated_virial_2}F''_R(t)= -\frac{16}{N-2}d(t)+A_R(u(t)),
\end{align}
where
\begin{equation}\label{truncated_virial_error}
    A_R(u(t)) = \int_{|x|\geq R}|\nabla u(t)|^2\left(4\partial_r^2\phi_R-8\right)+\int_{|x|\geq R}|u|^{2^*}\left(8-\frac{4}{N}\Delta \phi_R\right)-\int_{|x|\geq R} |u|^2 \Delta^2\phi_R.
\end{equation}

Recall that, if $\|\nabla u_0\|_{\dot{H}^1}> \|\nabla W\|_{\dot{H}^1}$, then, for all $t$ in the interval of definition of $u$,
\begin{equation}
d(t) = \Big|\|\nabla u(t)\|_{\dot{H}^1} - \|\nabla W\|_{\dot{H}^1}\Big| = \|\nabla u(t)\|_{\dot{H}^1} - \|\nabla W\|_{\dot{H}^1}.    
\end{equation}
In order to prove Proposition \ref{prop_supercrit}, we need the following lemma, which is also proved in Appendix.

\begin{lemma}\label{lem_int_exp_decay_d_2}
Let $u$ be a radial solution to \eqref{NLS} defined on $[0,+\infty)$ and satisfying \eqref{supercrtit}. Assume, furthermore, that the mass $M(u_0)$ is finite. Then, there exists a constant $R_0 >0$ such that for all $t$ in the interval of existence of $u$ and all $R\geq R_0$ 
\begin{equation}\label{positivity_2}
F_R'(t) > 0,    
\end{equation}
and there exists $c > 0$ such that 
\begin{equation}\label{int_exp_decay_d_2}
    \int_t^{+\infty}d(s)ds \lesssim e^{-ct},\quad \forall t \geq 0.
\end{equation}
\end{lemma}

\subsection{Intercritical case}
Here, we state the corresponding result for the intercritical case. Since the proof is very similar to the energy-critical case, we mainly sketch it in Appendix.

\begin{prop}\label{sub_prop_supercrit}
Let $u$ be a solution to \eqref{sub_NLS} defined on $[0,+\infty)$ satisfying
\begin{equation}\label{sub_supercrtit}
    M(u_0)= M(Q), \quad E(u_0) = E(Q), \quad\text{and}\quad \|\nabla u_0\|_{L^2} > \|\nabla Q\|_{L^2}.
\end{equation}
Assume, furthermore, that 
either $u_0$ is radial or $|x|u_0\in L^2(\mathbb{R}^N)$ . Then there exist $(x_0,\theta_0)$ and $c>0$ such that 
\begin{equation}\label{sub_exp_conv_W_super}
    \|u-e^{it}Q_{[x_0,\theta_0]}\|_{H^1} \lesssim  e^{-ct}.
\end{equation}
Moreover, $u$ blows up in finite negative time.
\end{prop}

\section{Convergence to the standing wave below the ground state}
\subsection{Energy-critical case}
In this section, we consider solutions such that
\begin{equation}\label{subcrtit}
    E(u_0) = E(W) \text{ and } \|u_0\|_{\dot{H}^1} < \|W\|_{\dot{H}^1}.
\end{equation}

\begin{de}
A solution $u$ to \eqref{NLS} with lifespan $I$ is said to be \textit{almost periodic modulo symmetries} on $J \subset I$ if there exist functions $x: J \to \Real^N$, $\lambda: J\to \Real_+^*$ and $C: \Real^+_* \to \Real^+_*$ such that for all $t \in I$ and all $\eta>0$
\begin{equation}
    \int_{|x-x(t)| \geq C(\eta)\slash\lambda(t) }|\nabla u(x,t)|^2 \, dx \leq \eta
\end{equation}
and
\begin{equation}
    \int_{|\xi| \geq C(\eta)\lambda(t) }|\xi|^2| \hat{u}(\xi,t)|^2 \, d\xi\leq \eta.
\end{equation}
\end{de}
\begin{re} By Arzel{\`a}-Ascoli's theorem, almost periodicity modulo symmetries is equivalent to the set
\begin{equation}
    \left\{u_{[x(t),\lambda(t),0]}: \, t \in J \right\}
\end{equation}
being precompact in $\dot{H}^1$.
\end{re}

\begin{re} If the solution is radial, $x(t)$ can be chosen to be zero.

\end{re}
\begin{prop}\label{sub_global}
Let $u$ be a solution to \eqref{NLS} and $I=(T^-,T^+)$ be its maximal interval of existence. If $u$ satisfies \eqref{subcrtit}, then
\begin{equation}
    I = \Real.
\end{equation}
Furthermore, if 
\begin{equation}\label{noscatter}
\int_0^{+\infty}\int_{\Real^N} |u(x,t)|^{\frac{2(N+2)}{N-2}} dx \, dt= +\infty,    
\end{equation}
 then $u$ is almost periodic modulo symmetries on $[0,+\infty)$.
\end{prop}
The proof of Proposition \ref{sub_global} is essentially contained in the proof in \cite{KV10}*{Proposition 3.1}, which extended the work in \cite{KM_Glob} to dimensions $N \geq 6$.

\begin{re}
By time-reversal symmetry, the analogous version of \eqref{noscatter} for the interval $(-\infty,0]$ holds.
\end{re}
The next theorem is the main result proved in \cite{KV10}*{Theorem 1.7}.
\begin{teo} For $N \geq 5$, let $u: I\times\Real^N \to \mathbb C$ be a solution to \eqref{NLS} satisfying
\begin{equation}
    E_*:=\sup_{t \in I} \|u(t)\|_{\dot{H}^1} < \|W\|_{\dot{H}^1}.
\end{equation}
 Then,
 \begin{equation}
\int_I\int_{\Real^N} |u(x,t)|^{\frac{2(N+2)}{N-2}} dx \, dt = C(E_*)< +\infty.    
\end{equation}
\end{teo}
In particular, by uniqueness of solutions and continuity of the flow of \eqref{NLS}, we have the following consequence.
\begin{cor}\label{cor_mean_subcrit} For $N \geq 5$, let $u$ be a solution to \eqref{NLS} satisfying \eqref{subcrtit} and \eqref{noscatter}. Then there exists a sequence $t_n \to +\infty$ such that
\begin{equation}
 \lim_{n \to +\infty} d(u(t_n)) = 0.
\end{equation}

\end{cor}
The main aim of this section is to prove the following proposition.
\begin{prop}\label{prop_subcrit}
Let $u$ be a radial solution to \eqref{NLS} satisfying \eqref{subcrtit} and \eqref{noscatter}. Then there exist $(\lambda_0,\theta_0)$ and $c>0$ such that, for all $t\geq0$,
\begin{equation}\label{exp_conv_sub}
    \|u-W_{[  \lambda_0,\theta_0]}\|_{\dot{H}^1} \lesssim  e^{-ct}.
\end{equation}
Moreover, $u$ scatters backward in time.
\end{prop}

The proof follows the same lines as in \cite{DM_Dyn}, and is given in Appendix.

\subsection{Intercritical case}

Here, we consider solutions such that
\begin{equation}\label{sub_subcrtit}
    M(u_0) = M(Q), \, E(u_0) = E(Q), \text{ and } \|\nabla u_0\|_{L^2} < \|\nabla Q\|_{L^2}.
\end{equation}

Since the scaling parameter is fixed \textit{a priori} in the intercritical regime, controlling scaling is no longer an issue. We can then use the fact that the solution has finite mass, together with information given by the virial-type and compactness arguments, to control the translation parameter, allowing us to prove results in the non-radial setting. We start with a definition.

\begin{de}
A solution $u$ to \eqref{sub_NLS} with lifespan $I$ is said to be \textit{almost periodic modulo symmetries} on $J \subset I$ if there exist functions $x: J \to \Real^N$ and $C: \Real^+_* \to \Real^+_*$ such that for all $t \in J$ and all $\eta>0$
\begin{equation}
    \int_{|x-x(t)| \geq C(\eta) }|\nabla u(x,t)|^2+|u(x,t)|^2 \, dx \leq \eta.
\end{equation}

\end{de}
\begin{re} By Arzel{\`a}-Ascoli's theorem, almost periodicity modulo symmetries is equivalent to the set
\begin{equation}
    \left\{u_{[x(t)]}: \, t \in J \right\}
\end{equation}
being precompact in $H^1$.
\end{re}


\begin{prop}\label{sub_sub_global}
Let $u$ be a solution to \eqref{sub_NLS} and $I=(T^-,T^+)$ be its maximal interval of existence. If $u$ satisfies \eqref{sub_subcrtit}, then
\begin{equation}
    I = \Real.
\end{equation}
Furthermore, if 
\begin{equation}\label{sub_noscatter}
\|u\|_{S(0,+\infty)} = +\infty,    
\end{equation}
 then $u$ is almost periodic modulo symmetries on $[0,+\infty)$, and we have
 \begin{equation}
     P(u) = \Im\int \bar{u}\nabla u = 0, \text{ and }
 \end{equation}
 \begin{equation}
     \lim_{t \to \infty} \frac{x(t)}{t} = 0.
 \end{equation}
\end{prop}
The proof of Proposition \ref{sub_sub_global} is now classical, and it is essentially the same as in \cite{DR_Thre}*{Lemma 6.2, Corollary 6.3 and Lemma 6.4}.

\begin{re}
By time-reversal symmetry, the analogous version of \eqref{sub_noscatter} for the interval $(-\infty,0]$ holds.
\end{re}
\begin{re}\label{sub_re_x}
As in \cite{DR_Thre}*{Lemma 6.2}, the function $x(t)$ can be chosen to be continuous on $\Real$ and the same as the one given in Lemmas \ref{sub_mod_1} and \ref{sub_mod_2}, if $d(t) < \delta_0$.
\end{re}

\begin{prop}\label{sub_prop_subcrit}
Let $u$ be a solution to \eqref{sub_NLS} satisfying \eqref{sub_subcrtit} and \eqref{sub_noscatter}. Then there exist $(x_0,\theta_0)$ and $c>0$ such that, for all $t\geq0$,
\begin{equation}\label{sub_exp_conv}
    \|u-e^{it}Q_{[ x_0,\theta_0]}\|_{H^1} \lesssim  e^{-ct}.
\end{equation}
Moreover, $u$ scatters backward in time.
\end{prop}

\section{Estimates on exponentially decaying solutions}
According to the previous sections, we must study the behavior of solutions approaching $e^{it}Q$ exponentially fast in time.  We start with the energy-critical setting.
\subsection{Energy-critical case}

In contrast to the previous two sections, the radiality assumption is not needed to prove the results in this subsection. We consider the linearized approximate equation
\begin{equation}\label{linear_exp}
	\partial_t h +\mathcal{L}h = \epsilon
\end{equation}
with $h$ and $\epsilon$ such that, for $t \geq 0$,
\begin{equation}\label{exp_decay}
	\begin{gathered}
	\|h(t)\|_{\dot{H}^1} \lesssim e^{-c_0t},\\
	\|\epsilon(t)\|_{L^\frac{2N}{N+2}} + \|\nabla \epsilon\|_{S'(L^2,\,[t,+\infty))} \lesssim e^{-c_1t},
	\end{gathered}
\end{equation}
where $c_1 > c_0 > 0$. The following self-improving estimate was proved for radial data in \cite{DM_Dyn}. We give the proof without the radial assumption in Appendix.

Using the notation $a^-$ for $a - \delta$ with arbitrarily small $\delta > 0$, we have the following lemma.
\begin{lemma}\label{lem_self}Under the assumptions \eqref{exp_decay},
\begin{enumerate}[(i)]
	\item if $e_0 \notin [c_0,c_1)$, then 
	\begin{equation}\label{lem_self_1}
		\|h(t)\|_{\dot{H}^1} \lesssim e^{-c_1^- t}, 
	\end{equation}
	\item if $e_0 \in [c_0,c_1)$, then there exists $A \in \Real$ such that 
	\begin{equation}\label{lem_self_2}
			\|h(t)-Ae^{-e_0t}\mathcal{Y}_+\|_{\dot{H}^1} \lesssim e^{-c_1^- t }.
	\end{equation}
\end{enumerate}
\end{lemma}

To further improve the convergence in the case $N \geq 6$, we study the linearized equation around $W+\mathcal{V}_k^A$, for $A \neq 0$, which was defined in \eqref{eq_sol_v}. For simplicity, we omit the superscript $A$. Defining, for every $k$,
\begin{align}
    \tilde{\mathcal{L}_k} &=\begin{pmatrix}0 & \Delta \\
-\Delta & 0 \\
\end{pmatrix}+
\frac{(p_c+1)}{2}|W+\mathcal{V}_k|^{p_c-1}\begin{pmatrix}0 & 1 \\
-1 & 0 \\
\end{pmatrix}\\
&\quad+\frac{(p_c-1)}{2}|W+\mathcal{V}_k|^{p_c-3}\begin{pmatrix}\Im (W+\mathcal{V}_k)^2 & -\Re (W+\mathcal{V}_k)^2 \\
-\Re (W+\mathcal{V}_k)^2 & - \Im (W+\mathcal{V}_k)^2 \\
\end{pmatrix},\\
\end{align}
\begin{equation}
    \tilde{K}_k(h) = \frac{(p_c+1)}{2}|W +\mathcal{V}_k|^{p_c-1} h + \frac{(p_c-1)}{2}|W+\mathcal{V}_k|^{p-3}(W+\mathcal{V}_k)^2 \bar{h},
\end{equation}
and
\begin{equation}
    \tilde{R}_k(h) = |W+\mathcal{V}_k|^{p_c-1}(W+\mathcal{V}_k) J((W+\mathcal{V}_k)^{-1}h), 
\end{equation}
where
\begin{equation}
    J(z) = |1+z|^{p_c-1}(1+z)-1-\frac{(p_c+1)}{2}z-\frac{(p_c-1)}{2}\bar{z},
\end{equation}
we have that if, $u = W +\mathcal{V}_k + h$ satisfies \eqref{NLS}, then $h$ satisfies

\begin{equation}\label{linearized_eq_tilde}
	\partial_t h +\tilde{\mathcal{L}}_k h = i \tilde{R}_k(h) +\epsilon_k,
\end{equation}
or in the form of a Schr\"odinger equation,
\begin{equation}
    i\partial_t h + \Delta h + \tilde{K}_k h = - \tilde{R}_k(h)+ i\epsilon_k,
\end{equation}
where $\epsilon_k$ are $O(e^{-(k+1)e_0 t})$ in $\mathcal{S}(\Real^N)$ . Note that the operator $\tilde{\mathcal{L}}_k$ is time-dependent and that, by the construction of $\mathcal{V}_k$, we have, for all $t \geq 0$,
\begin{equation}
    |\mathcal{V}_k(t)| \lesssim e^{-e_0t }|W|,
\end{equation}
and
\begin{equation}
    |\nabla \mathcal{V}_k(t)| \lesssim e^{-e_0t }|\nabla W|\lesssim e^{-e_0t }|W|.
\end{equation}
This implies that the estimates in Lemmas \ref{lem_nonl} and \ref{lem_stric_exp} hold with the same proof if we replace $K$ by $\tilde{K}_k$ and $R$ by $\tilde{R}_k$. Therefore, we have the following results.
\begin{lemma}\label{lem_nonl_tilde} Let $N \geq 6$, $k \geq 1$ and $I$ be a bounded time interval, and consider $f \in S(\dot{H}^1,I)$ such that $\nabla f \in S(L^2,I)$. The following estimates hold
\begin{enumerate}[(i)]
\item \label{linear1_tilde} $\|\nabla \tilde{K}_k(f)\|_{S'(L^2,I)} \lesssim |I|^\frac{1}{2} \|\nabla f\|_{S(L^2,I)}$,
\item \label{grad_r_point_tilde} $\begin{aligned}[t]\|\nabla \tilde{R}_k(f)\|_{S'(L^2,I)}+\|\tilde{R}_k(f)\|_{L^\frac{2N}{N+2}} \lesssim 
&\left(1+|I|^\frac{1}{2}\right)\|\nabla f\|_{S(L^2,I)}^{p_c}.\end{aligned}$
\end{enumerate}
\end{lemma}

\begin{lemma}\label{lem_stric_exp_tilde}
Let $h$ be a solution to \eqref{linearized_eq_tilde}.  If, for some $c>0$ and for any $t \geq 0$,
\begin{equation}
    \|h(t)\|_{\dot{H}^1} \lesssim e^{-c t},
\end{equation}
then
\begin{equation}\label{stric_exp_tilde}
    \|\nabla h\|_{S(L^2,\,[t,+\infty))} \lesssim e^{-\min\{c\,,(k+1^-)e_0\}t}.
\end{equation}
\end{lemma}

In the spirit of Lemma \ref{lem_self}, we prove the following estimate.

\begin{lemma}\label{lem_self_tilde} For $N \geq 6$, let $h$ be a solution to 
\begin{equation}\label{linear_exp_tilde}
	\partial_t h +\tilde{\mathcal{L}}_kh = \epsilon,
\end{equation}
with $h$ and $\epsilon$ such that, for $t \geq 0$,
\begin{equation}\label{exp_decay_tilde}
	\begin{gathered}
	\|h(t)\|_{\dot{H}^1} \lesssim e^{-c_0t},\\
	\|\epsilon(t)\|_{L^\frac{2N}{N+2}} + \|\nabla \epsilon\|_{S'(L^2,\,[t,+\infty))} \lesssim e^{-c_1t},
	\end{gathered}
\end{equation}
where $(k+1)e_0 > c_1 > c_0 > e_0$. Then,
	\begin{equation}\label{self_tilde_1}
		\|h(t)\|_{\dot{H}^1} \lesssim e^{-c_1^- t}. 
	\end{equation}
\end{lemma}

\begin{proof}[Proof of Lemma \ref{lem_self_tilde}]
Since the subscript $k$ will be fixed in this proof, we omit it. By Lemma \ref{lem_stric_exp_tilde}, we have
\begin{equation}
    \|\nabla h(t)\|_{S(L^2,\,[t,+\infty))} \lesssim e^{-c_0t}. 
\end{equation}

We first note that \eqref{linear_exp_tilde} can be written as
\begin{equation}
    \partial_t h +\mathcal{L}h = \epsilon+(\mathcal{L}-\tilde{\mathcal{L}}_k)h.
\end{equation}

Now, if $N > 6$ and $h \in \dot{H}^1$,
\begin{equation}
    |(\mathcal{L}-\tilde{\mathcal{L}})h| \lesssim |\mathcal{V}(t)|^{p_c-1}|h|
    \lesssim e^{-(p_c-1)e_0t}|W|^{p_c-1}|h|,
\end{equation}
and
\begin{align}
    |\nabla [(\mathcal{L}-\tilde{\mathcal{L}})h]| &\lesssim |W|^{p_c-2}\left[e^{-e_0t}|\nabla W|+|\nabla \mathcal{V}(t)|\right]|h|+|\mathcal{V}(t)|^{p_c-1}|\nabla h|\\
    &\lesssim e^{-(p_c-1)e_0t}\left[|W|^{p_c-1}|h|+|\nabla h|\right],
\end{align}
where we used the fact that $\mathcal{V}(t) \in \mathcal{S}(\Real^N)$, $\|\mathcal{V}(t)\|_{L^\infty} \lesssim e^{-e_0t}$ and $|\nabla W| \lesssim |W|$.

Thus, 
\begin{equation}
\|(\mathcal{L}-\tilde{\mathcal{L}})h\|_{L^\frac{2N}{N+2}} + \|\nabla [(\mathcal{L}-\tilde{\mathcal{L}})h]\|_{S'(L^2,\,[t,+\infty))} \lesssim e^{-[\min\{c_0,\,(k+1^-)e_0\}+(p_c-1)e_0]t}.
\end{equation}

Therefore, by Lemma \ref{lem_self}, since $c_0 > e_0$ by hypothesis,
\begin{equation}
    \|h\|_{\dot{H}^1} \lesssim e^{-t\min\{[c_0+(p_c-1)e_0],\,c_1\}^-}.
\end{equation}

By iterating this argument, we get \eqref{self_tilde_1}.

\end{proof}

We now improve the convergence given by Propositions \ref{prop_supercrit} and \ref{prop_subcrit}.
\begin{lemma}\label{improv_conv} For $N \geq 6$, if $u$ is a solution to \eqref{NLS} satisfying, for all $t \geq 0$,
\begin{equation}\label{exp_conv_7}
    \|u(t) - W\|_{\dot{H}^1} \lesssim e^{-ct}, \quad E(u_0) = E(W),
\end{equation}

then there exists a unique $A \in \Real$ such that $u = U^A$.
\end{lemma}
\begin{proof}

\textit{Step 1. Linearize around $W$ to improve the decay on time.}

If $u$ is a solution to \eqref{NLS}, write $u = h+W$. Recall that $h$ is a solution to \eqref{linearized_eq}. We first show the bound
\begin{equation}\label{boot_Z}
    \|\nabla(R(h))\|_{S'(L^2,\,[t,+\infty))}+\|R(h(t))\|_{L^\frac{2N}{N+2}} \lesssim e^{-c\,p_ct} \text{ for all } t \geq 0.
\end{equation}
Indeed, by Lemmas \ref{lem_stric_exp} and \ref{lem_nonl} \eqref{grad_r_point},
\begin{equation}
    \|\nabla(R(h))\|_{S'(L^2,\,[t,t+1])} \lesssim e^{-c\,p_ct}.
\end{equation}
Therefore, triangle inequality gives 
\begin{equation}
        \|\nabla(R(h))\|_{S'(L^2,\,[t,+\infty))}\lesssim e^{-c\,p_ct}. 
\end{equation}

Now, by \eqref{R_diff}, we have 
\begin{equation}
    |R(h(t))| \leq |h(t)|^{p_c},
\end{equation}
so that, by Sobolev inequality, 
\begin{equation}
    \|R(h(t))\|_{L^\frac{2N}{N+2}} \lesssim \|h(t)\|^{p_c}_{L^{2^*}} \lesssim e^{-c\,p_ct}.
\end{equation}
Therefore, the bound \eqref{boot_Z} is proved.

We are now under the hypotheses of Lemma \ref{lem_self}, with $c_0 = c$ and $c_1 = c p_c > c$. The conclusion of this Lemma gives
\begin{equation}
    \|h(t)\|_{\dot{H}^1} \lesssim e^{-e_0t}+e^{-cp_c^-t}.
\end{equation}
If $c> e_0/p_c$, we get
\begin{equation}
    \|h(t)\|_{\dot{H}^1} \lesssim e^{-e_0t}, 
\end{equation}

and, by the same argument used to prove \eqref{boot_Z},
\begin{equation}
\|\nabla(R(h))\|_{S'(L^2,\,[t,+\infty))}+\|R(h(t))\|_{L^\frac{2N}{N+2}} \lesssim e^{-e_0 p_ct}.     
\end{equation}
Thus, \eqref{lem_self_2} gives 
\begin{equation}\label{best_boot}
   \|h(t) - Ae^{-e_0t}\mathcal{Y}_+\|_{\dot{H}^1} \lesssim e^{-p_c^-e_0t}.
\end{equation}

If, however, $c \leq e_0/p_c$, then assumption \eqref{exp_conv_7} holds with $\frac{1+p_c}{2}c > c$ instead of $c$. By iteration, we get \eqref{best_boot}.

\textit{Step 2. Linearize around $W+\mathcal{V}_k$ to improve higher order convergence to $U^A$.}

For $k \geq 2$ to be chosen later, write $\tilde{h} = h - \mathcal{V}_k$, so that $\tilde{h}$ is a solution to \eqref{linearized_eq_tilde}. Since $k$ is fixed throughout the proof, it will be omitted. By Lemma \ref{lem_nonl_tilde}, we have
\begin{equation}
    \|\nabla(\tilde{R}(\tilde{h}))\|_{S'(L^2,\,[t,+\infty))}+\|\tilde{R}(\tilde{h}(t))\|_{L^\frac{2N}{N+2}} \lesssim_k  \|\nabla \tilde{h}\|_{S(L^2,\,[t,+\infty))}^{p_c}.     
\end{equation}
Therefore, by \eqref{linearized_eq_tilde},
\begin{equation}
    \partial_t h +\tilde{\mathcal{L}}_kh = \eta,
\end{equation}
with 
\begin{equation}\label{boot_tilde}
    \|\nabla \eta\|_{S'(L^2,\,[t,+\infty))}+\|\eta\|_{L^\frac{2N}{N+2}} \lesssim_k  \|\nabla \tilde{h}\|_{S(L^2,\,[t,+\infty))}^{p_c} + e^{-(k+1)e_0t}.     
\end{equation}

By \eqref{best_boot} and the definition of $\mathcal{V}_k$, we have
\begin{equation}\label{base_ind_tilde}
    \|\tilde{h}\|_{\dot{H}^1} \lesssim \|h - Ae^{-e_0t}\mathcal{Y}_+\|_{\dot{H}^1} + O(e^{-2e_0t}) \lesssim_k e^{-p_c^-e_0t}.
\end{equation}
By iteration, starting with \eqref{base_ind_tilde}, and repeatedly applying Lemmas \ref{lem_self_tilde} and \ref{stric_exp_tilde}, as well as estimate \eqref{boot_tilde}, we have, for any $k \geq 2$,
\begin{equation}
    \|\tilde{h}\|_{\dot{H}^1} \lesssim_{k} e^{-(k+1^-)e_0t}.
\end{equation}
Therefore, choosing $k = l(k_0)$, where $k_0$ and $l$ are defined in Proposition \ref{exist_UA}, we have, for $N \geq 6$ and $t \geq 0$,

\begin{equation}
	\|D^\varepsilon(u-W-\mathcal{V}_{l(k_0)})\|_{S(\dot{H}^{1-\varepsilon},\,[t,+\infty))} \lesssim
	\|\nabla(u-W-\mathcal{V}_{l(k_0)})\|_{S(L^2,\,[t,+\infty))} \lesssim e^{-(k_0+\frac{3}{4})\frac{N-2}{4}  e_0 t}.
\end{equation}
Hence, by the uniqueness in Proposition \ref{exist_UA}, we get that $u = U^A$.

\end{proof}
\begin{cor}\label{unique_A} Let $N \geq 6$. For any $A \neq 0$, there exists $T_A \in \Real$ such that either
\begin{equation}
    U^A(t) = W^+(t+T_A), \text{ if } A > 0,
\end{equation}
or
\begin{equation}
    U^A(t) = W^-(t+T_A), \text{ if } A < 0.
\end{equation}
\end{cor}
\begin{proof}
Choose $T_A$ such that $|A|e^{-e_0T_A} = 1$. We have, by \eqref{bound_UA2},
\begin{equation}\label{A_pm}
    \|U^A(t+T_A) - W \mp e^{-e_0t}\mathcal{Y}_{+}\|_{\dot{H}^1} \lesssim e^{-2e_0t}.
\end{equation}
Note that $U^A(t+T_A)$ satisfies the hypotheses of Lemma \ref{improv_conv}. Thus, there exists $\bar{A}\in \Real$ such that $U^A(t+T_A) = U^{\bar{A}}$. But \eqref{A_pm} implies that $\bar{A} = 1$, if $A > 0$, and $\bar{A} = -1$, if $A < 0$, finishing the proof of the corollary. 
\end{proof}

\subsection{Intercritical case}

For $0 < s_c<1$, we study the linearized approximate equation
\begin{equation}\label{sub_linear_exp}
	\partial_t h +\mathcal{L}h = \epsilon
\end{equation}
with $h$ and $\epsilon$ such that, for $t \geq 0$,
\begin{equation}\label{sub_exp_decay}
	\begin{gathered}
	\|h(t)\|_{H^1} \lesssim e^{-c_0t},\\
	\|\langle\nabla\rangle \epsilon\|_{S'(L^2, [t,+\infty))}  \lesssim e^{-c_1t},
	\end{gathered}
\end{equation}
where $c_1 > c_0 > 0$. We merely state the results in this case, as their proofs are very close to the energy-critical case (in fact, some proofs are easier, since the $L^2$ norm of the solution is finite).

\begin{lemma}\label{sub_lem_self}Under the assumptions \eqref{sub_exp_decay},
\begin{enumerate}[(i)]
	\item if $e_0 \notin [c_0,c_1)$, then 
	\begin{equation}\label{sub_lem_self_1}
		\| h(t)\|_{H^1} \lesssim e^{-c_1^- t}, 
	\end{equation}
	\item if $e_0 \in [c_0,c_1)$, then there exists $A \in \Real$ such that 
	\begin{equation}\label{sub_lem_self_2}
			\|h(t)-Ae^{-e_0t}\mathcal{Y}_+\|_{H^1} \lesssim e^{-c_1^- t }.
	\end{equation}
\end{enumerate}
\end{lemma}


Omitting $A$ for simplicity and defining, for every $k$,
\begin{align}
    \tilde{\mathcal{L}_k} &=\begin{pmatrix}0 & 1-\Delta \\
-(1-\Delta) & 0 \\
\end{pmatrix}+
\frac{(p+1)}{2}|Q+\mathcal{V}_k|^{p-1}\begin{pmatrix}0 & 1 \\
-1 & 0 \\
\end{pmatrix}\\
&\quad+\frac{(p-1)}{2}|Q+\mathcal{V}_k|^{p-3}\begin{pmatrix}\Im (Q+\mathcal{V}_k)^2 & -\Re (Q+\mathcal{V}_k)^2 \\
-\Re (Q+\mathcal{V}_k)^2 & - \Im (Q+\mathcal{V}_k)^2 \\
\end{pmatrix},\\
\end{align}
\begin{equation}
    \tilde{K}_k(h) = \frac{(p+1)}{2}|Q +\mathcal{V}_k|^{p-1} h + \frac{(p-1)}{2}|Q+\mathcal{V}_k|^{p-3}(Q+\mathcal{V}_k)^2 \bar{h},
\end{equation}
and
\begin{equation}
    \tilde{R}_k(h) = |Q+\mathcal{V}_k|^{p-1}(Q+\mathcal{V}_k) J((Q+\mathcal{V}_k)^{-1}h), 
\end{equation}
where
\begin{equation}
    J(z) = |1+z|^{p-1}(1+z)-1-\frac{(p+1)}{2}z-\frac{(p-1)}{2}\bar{z},
\end{equation}
we have that if $u = e^{it}(Q +\mathcal{V}_k + h)$ is a solution of the NLS equation \eqref{sub_NLS}, then $h$ satisfies

\begin{equation}\label{sub_linearized_eq_tilde}
	\partial_t h +\tilde{\mathcal{L}}_k h = i \tilde{R}_k(h) +\epsilon_k,
\end{equation}
or in the form of a Schr\"odinger equation,
\begin{equation}
    i\partial_t h + \Delta h - h + \tilde{K}_k h = - \tilde{R}_k(h)+ i\epsilon_k,
\end{equation}
where $\epsilon_k$ are $O(e^{-(k+1)e_0 t})$ in $\mathcal{S}(\Real^N)$. By the construction of $\mathcal{V}_k$, we have, for all $t \geq 0$,
\begin{equation}
    |\mathcal{V}_k(t)| \lesssim e^{-e_0t }|Q|,
\end{equation}
and
\begin{equation}
    |\nabla \mathcal{V}_k(t)| \lesssim e^{-e_0t }|\nabla Q|\lesssim e^{-e_0t }|Q|.
\end{equation}
Therefore, as in the energy-critical case, we have the following results.
\begin{lemma}\label{sub_lem_nonl_tilde} Let $p>1$, $k \geq 1$ and $I$ be a bounded time interval, and consider $f \in S(L^2,I)$ such that $\nabla f \in S(L^2,I)$. There exists $\alpha >0$ such that the following estimates hold.

For all $p > 1$:
\begin{enumerate}[(i)]
\item \label{sub_linear1_tilde} $\|\langle\nabla\rangle \tilde{K}_k(f)\|_{S'(L^2,\,I)} \lesssim |I|^{\alpha} \|\langle\nabla\rangle f\|_{S(L^2,\,I)}$.
%
\end{enumerate}
For $p > 2$:
\begin{enumerate}[(i)]
\setcounter{enumi}{1}
\item \label{sub_P_grad_r_point_tilde} $\begin{aligned}[t]\|\langle\nabla\rangle \tilde{R}_k(f))\|_{S'(L^2,\,I)} \lesssim 
\|\langle\nabla\rangle f)\|_{S(L^2,\,I)}\left(|I|^\alpha\|\langle\nabla\rangle f\|_{S(L^2,\,I)} +\|\langle\nabla\rangle f\|^{p-1}_{S(L^2,\,I)} \right).
\end{aligned}$
\end{enumerate}
For $1 < p \leq 2$:
\begin{enumerate}[(i)]
\setcounter{enumi}{2}
%
\item \label{sub_p_grad_r_point_tilde} $\begin{aligned}[t]\|\langle\nabla\rangle \tilde{R}_k(f)\|_{S'(L^2,\,I)} \lesssim 
&\left(1+|I|^\alpha\right)\|\langle\nabla\rangle f\|_{S(L^2,\,I)}^{p}\end{aligned}$.
\end{enumerate}

\end{lemma}

\begin{lemma}\label{sub_lem_stric_exp_tilde}
Let $h$ be a solution to \eqref{sub_linearized_eq_tilde}. If, for some $c>0$ and for any $t \geq 0$,
\begin{equation}
    \|h(t)\|_{H^1} \lesssim e^{-c t},
\end{equation}
then
\begin{equation}\label{sub_stric_exp_tilde}
    \|\langle\nabla\rangle h\|_{S(L^2,\,[t,+\infty))} \lesssim e^{-\min\{c\,,(k+1^-)e_0\}t}.
\end{equation}
\end{lemma}

\begin{lemma}\label{sub_lem_self_tilde} Let $h$ be a solution to 
\begin{equation}\label{sub_linear_exp_tilde}
	\partial_t h +\tilde{\mathcal{L}}_kh = \epsilon,
\end{equation}
with $h$ and $\epsilon$ such that, for $t \geq 0$,
\begin{equation}\label{sub_exp_decay_tilde}
	\begin{gathered}
	\|h(t)\|_{H^1} \lesssim e^{-c_0t},\\
	\|\langle\nabla\rangle \epsilon\|_{S'(L^2,\,[t,+\infty))} \lesssim e^{-c_1t},
	\end{gathered}
\end{equation}
where $(k+1)e_0 > c_1 > c_0 > e_0$. Then,
	\begin{equation}\label{sub_self_tilde_1}
		\|h(t)\|_{H^1} \lesssim e^{-c_1^- t}. 
	\end{equation}
\end{lemma}

\begin{lemma}\label{sub_improv_conv} If $u$ is a solution to \eqref{sub_NLS} satisfying
\begin{equation}\label{sub_exp_conv_7}
    \|u(t) - e^{it}Q\|_{H^1} \lesssim e^{-ct}, \quad M(u_0) = M(Q),\quad E(u_0) = E(Q),
\end{equation}
then there exists a unique $A \in \Real$ such that $u = U^A$.
\end{lemma}

\begin{cor}\label{sub_unique_A} Let $1+4/N < p < 2^*-1$. For any $A \neq 0$, there exists $T_A \in \Real$ such that either
\begin{equation}
    U^A(t) = Q^+(t+T_A), \text{ if } A > 0,
\end{equation}
or
\begin{equation}
    U^A(t) = Q^-(t+T_A), \text{ if } A < 0.
\end{equation}
\end{cor}


\section{Closure of the main theorems}
Having proved Propositions \ref{prop_supercrit} and \ref{prop_subcrit}, and Lemma \ref{improv_conv}, we can proceed as in \cite{DM_Dyn}.
\begin{proof} [Proof of Theorem \ref{special}]
Recall the notation $\mathcal{Y}_1 = \Re \mathcal{Y}_+ = \Re \mathcal{Y}_-$. We claim that $(W, \mathcal{Y}_1)_{\dot{H}^1} \neq 0$. If not, since $W$ solves the equation $\Delta W = -W^{p_c}$, we would have
\begin{equation}
    B(W,\mathcal{Y}_\pm) = \frac{1}{2}\int \nabla W \cdot \nabla \mathcal{Y}_1 - \frac{p_c}{2}\int W^{p_c}\mathcal{Y}_1 = \frac{p_c}{2}\int\Delta W \mathcal{Y}_1 = 0,
\end{equation}
so that $W \in G^\perp$. But, by Lemma \ref{lem_coerc}, $\Phi$ is nonnegative (in fact, it is coercive) on  $G^\perp$, which contradicts \eqref{Phi_W_neg}.

Replacing $\mathcal{Y}_\pm$, if necessary, we may assume
\begin{equation}
(W, \mathcal{Y}_1)_{\dot{H}^1} > 0.
\end{equation}
Defining
\begin{equation}
    W^+ := U^{1}\,\,\,\mbox{and}\,\,\,W^- := U^{-1},
\end{equation}
we claim that the conclusions of Theorem \ref{special} hold. By the strong convergence $U^A(t) \to W$ in $\dot{H}^1$ and energy conservation, we conclude $E(W^\pm) = E(W)$. Moreover, by \eqref{bound_UA},
\begin{equation}
    \|U^A(t)\|^2_{\dot{H}^1} = \|W\|^2_{\dot{H}^1} +2A e^{-2e_0t}(W,\mathcal{Y}_1)_{\dot{H}^1}+O(e^{-3e_0t}),
\end{equation}
which shows that $\|U^A(t)\|_{\dot{H}^1} - \|W\|_{\dot{H}^1}$ has the same sign as $A$, for large $t$. By uniqueness and continuity of the flow, this sign must remain the same for every $t$ in the intervals of existence of $W^\pm$. By Proposition \ref{sub_global}, $W^-$ is defined on $\Real$, and by Proposition \ref{prop_subcrit}, $W^-$ scatters backward in time. 

We now show that $U^A$ has finite mass for $N \geq 5$. Let, as in the proof of Proposition \ref{sub_global}, $\phi$ be a smooth, positive, radial cutoff to the set $\{|x| \leq 1\}$. Define, for $R>0$ and large $t$,
\begin{equation}
    F_R(t) = \int |U^A(x,t)|^2 \phi(\frac{x}{R}) \, dx.
\end{equation}
Since $U^A$ is a solution to \eqref{NLS}, by Lemma \ref{cauchy_banica} and Hardy's inequality, we have
\begin{equation}
    |F'_R(t)| \lesssim \|U^A(t)-W\|_{\dot{H}^1} \left( \int \frac{1}{|x|^2} |U^A(t)|^2\right)^\frac{1}{2} \lesssim \|U^A(t)-W\|_{\dot{H}^1} \|U^A(t)\|_{\dot{H}^1} \lesssim e^{-e_0t}.
\end{equation}
Hence, integrating from a large $t$ to $+\infty$, we get
\begin{equation}
\left| F_R(t) - \int |W|^2 \phi(\frac{x}{R}) \, dx\right| \lesssim e^{-e_0t}.
\end{equation}
Recalling that $W \in L^2(\Real^N)$ if $N \geq 5$, we can make $R \to +\infty$ to obtain $M(U^A) = M(W) < +\infty$. In particular, $W^\pm \in L^2(\Real^N)$ and, by Proposition \ref{prop_supercrit}, $W^+$ blows up in finite negative time. This finishes the proof of Theorem \ref{special}.
\end{proof}
\begin{proof}[Proof of Theorem \ref{class_thresh}]
The case $\|u_0\|_{\dot{H}^1} <  \|W\|_{\dot{H}^1}$ follows immediately from Proposition \ref{prop_subcrit}, Lem\-ma \ref{improv_conv} and Corollary \ref{unique_A}. Case $\|u_0\|_{\dot{H}^1} =  \|W\|_{\dot{H}^1}$ is a consequence of the variational characterization of $W$. Finally, $\|u_0\|_{\dot{H}^1} >  \|W\|_{\dot{H}^1}$ follows from Proposition \ref{prop_supercrit}, Lemma \ref{improv_conv} and Corollary \ref{unique_A}.
\end{proof}

\begin{proof} [Proof of Theorem \ref{sub_special}]
Recall the notation $\mathcal{Y}_1 = \Re \mathcal{Y}_+ = \Re \mathcal{Y}_-$. We claim that $(Q, \mathcal{Y}_1)_{H^1} \neq 0$. If not, since $Q$ solves the equation $Q - \Delta Q = -Q^{p}$, we would have
\begin{equation}
    B(Q,\mathcal{Y}_\pm) = \frac{1}{2}\int Q  \mathcal{Y}_1+ \frac{1}{2}\int \nabla Q \cdot \nabla \mathcal{Y}_1 - \frac{p}{2}\int Q^{p_c}\mathcal{Y}_1 = \frac{p+1}{2}(Q, \mathcal{Y}_1)_{H^1} = 0,
\end{equation}
so that $Q \in \tilde{G}^\perp$. But, by Lemma \ref{lem_coerc}, $\Phi$ is nonnegative (in fact, it is coercive) on  $\tilde{G}^\perp$, which contradicts \eqref{sub_Phi_W_neg}.

Replacing $\mathcal{Y}_\pm$, if necessary, we may assume
\begin{equation}
(Q, \mathcal{Y}_1)_{H^1} > 0.
\end{equation}
Defining
\begin{equation}
    Q^+ := U^{1}\,\,\,\mbox{and}\,\,\,Q^- := U^{-1}, 
\end{equation}
we claim that the conclusions of Theorem \ref{sub_special} hold. By the strong convergence $e^{-it}U^A(t) \to Q$ in $H^1$ and energy conservation, we conclude $M(Q^\pm) = M(Q)$ and $E(Q^\pm) = E(Q)$. Moreover, by \eqref{sub_bound_UA},
\begin{equation}
    \|U^A(t)\|^2_{H^1} = \|Q\|^2_{H^1} +2A e^{-2e_0t}(Q,\mathcal{Y}_1)_{H^1}+O(e^{-2p_c^-e_0t}),
\end{equation}
which shows that $\|U^A(t)\|_{H^1} - \|Q\|_{H^1}$ has the same sign as $A$, for large $t$. By uniqueness and continuity of the flow, this sign must remain the same for every $t$ in the intervals of existence of $Q^\pm$. By Proposition \ref{sub_sub_global}, $Q^-$ is defined on $\Real$, and by Proposition \ref{sub_prop_subcrit}, $Q^-$ scatters backward in time. Finally, by Proposition \ref{sub_prop_supercrit}, $Q^+$ blows up in finite negative time. This finishes the proof of Theorem \ref{sub_special}.
\end{proof}
\begin{proof}[Proof of Theorem \ref{sub_class_thresh}]
The case $\|\nabla u_0\|_{L^2} <  \|\nabla Q\|_{L^2}$ follows immediately from Proposition \ref{sub_prop_subcrit}, Lemma \ref{sub_improv_conv} and Corollary \ref{sub_unique_A}. The case $\|\nabla u_0\|_{L^2} =  \|\nabla Q\|_{L^2}$ is a consequence of the variational characterization of $Q$. Finally, $\|\nabla u_0\|_{L^2} > \|\nabla Q\|_{L^2}$ follows from Proposition \ref{sub_prop_supercrit}, Lemma \ref{sub_improv_conv} and Corollary \ref{sub_unique_A}.
\end{proof}

\section{Appendix}


\subsection{Proof of modulation results}

\begin{proof}[Proof of Lemma \ref{mod_1}] The idea of the proof is already well-known (see, for instance, \cite{DR_Thre}*{Section 7.1} for the 3d cubic NLS equation, \cite{DR_Thre}*{Section 3} for the energy-critical NLS equation, in the radial case, for dimensions $N = 3$, $4$ and $5$, and \cite{MM01}*{Section 2} in the context of the Korteweg-de Vries equation), and we extend the proofs here to any dimension, any $0 < s_c \leq 1$, not assuming radiality of the solution. The case $s_c = 1$ is the Lemma \ref{mod_1}, and the case $0 < s_c < 1$ is the Lemma \ref{sub_mod_1}. We first show the Lemma \eqref{mod_1} when $u$ is close to $W$. Define the functionals $J = (J_0, \cdots, J_{N+1})$ on $\Real^N\times \Real_+ \times \Real\times \dot{H}^1$ as
\begin{align}
    J_0&: (\theta,x,\lambda,u) \mapsto (f_{[x,\lambda,\theta]},iW)_{\dot{H}^1}, \\\ 
    J_k&: (\theta,x,\lambda,u) \mapsto (f_{[x,\lambda,\theta]},\partial_k W)_{\dot{H}^1},\, 1 \leq k \leq N, \\
    J_{N+1}&: (\theta,x,\lambda,u) \mapsto (f_{[x,\lambda,\theta]},\Lambda W)_{\dot{H}^1}.
\end{align}
By direct calculation, one can check that
\begin{equation}
    \det\left(\frac{\partial J}{\partial (\theta,x,\lambda) }\right) =   
        \left(\displaystyle\int |W|^2\right)
         \left(\prod_k\displaystyle\int |\partial_k W|^2\right)
         \left(-\displaystyle\int |\Lambda W|^2\right) \neq 0,
\end{equation}
and that $J(0,1,0,W) = 0$. Hence, by the Implicit Function Theorem, there exist $\epsilon_0, \eta_0$ such that, if $f \in \dot{H}^1$ and $\|f-W\|_{\dot{H}^1} < \epsilon_0$, then there exists a unique $n$-tuple $(x, \lambda, \theta)$ such that
\begin{equation}
    |x| + |\lambda| +|\theta-1| \leq \eta_0, \text{ and } J(\theta, x, \lambda, f) = 0.
\end{equation}
Now, if $u$ is as in the lemma, by the variational characterization of $W$, if $d(u)$ is small, then there exists $(x_0, \lambda_0, \theta_0)$ such that $u_{[x_0, \lambda_0, \theta_0]} = W + f$, with $\|f\|_{\dot{H}^1} \leq \epsilon(d(f))$. We are thus back to the preceding case. Existence, local uniqueness and regularity follow again from the Implicit Function Theorem.
\end{proof}

\begin{proof}[Proof of Lemma \ref{mod_2}]
For a fixed $t$, write $v = u_{[x(t),\lambda(t),\theta(t)]}(t)- W = \alpha(t)W+ h(t)$ as in \eqref{decomp_alpha}. Since $h \in G^\perp$, we have
\begin{equation}\label{equiv_1}
    \|v\|_{\dot{H}^1}^2 = \alpha^2 \|W\|_{\dot{H}^1}^2 +  \|h\|_{\dot{H}^1}^2 .
\end{equation}

Since $h \in G^\perp$, and $W$ satisfies the equation $\Delta W + W^{p_c}=0$, we have 
\begin{equation}
    B(W,h) = \frac{1}{2}\int \nabla W \cdot \nabla h_1 +\frac{p}{2}\int \Delta W h_1 = 0.
\end{equation}
Therefore, $\Phi(v) = \Phi(\alpha W + h) = \alpha^2\Phi(W) +\Phi(h)$. Recalling that $W$ is a critical point for the energy functional $E$, we have $E(W+v)=E(W)+\Phi(v)+O(\|v\|_{\dot{H}^1}^3)$. Since $E(W+v) = E(W)$, and by the coercivity given by Lemma \ref{lem_coerc}, we have $\Phi(h) \approx \|h\|_{\dot{H}^1}^2$. Thus, we have
\begin{equation}\label{equiv_3}
    \left|\alpha^2\Phi(W) +\Phi(h) \right| = O(\|v\|_{\dot{H}^1}^3)
\end{equation}

Since $\|v\|_{\dot{H}^1}$ is small when $d(u)$ is small, estimates \eqref{equiv_1} and \eqref{equiv_3} give $|\alpha| \approx \|h\|_{\dot{H}^1} \approx \|v\|_{\dot{H}^1}$. Finally, since 
\begin{equation}
    d(u) = \left|\|W+v\|_{\dot{H}^1}^2 - \|W\|_{\dot{H}^1}^2\right| = \left|\|v\|_{\dot{H}^1}^2+2\alpha \|W\|_{\dot{H}^1}^2\right|,
\end{equation}
we have $d(u) \approx |\alpha|$, and \eqref{lem_equiv_1} is proved.

It remains to prove \eqref{lem_equiv_2}. Consider the variables $y$ and $s$ given by 
\begin{equation}
    y = \frac{x}{\mu(t)} \text{ and } dt = \frac{1}{\mu^2(t)}ds.
\end{equation}
In view of \eqref{lem_equiv_1} and the decomposition \eqref{decomp_alpha}, we can rewrite \eqref{NLS} as 
\begin{equation}\label{NLS_mod}
    i \partial_s h + \Delta h -i\alpha_s W + \theta_s W - i x_s \cdot \nabla W + i \frac{\lambda_s}{\lambda}\Lambda W = O\left(\epsilon(s)\right) \text{ in } \dot{H}^1,
\end{equation}
where $\epsilon(s) := {d(u(t(s)))\left(d(u(t(s)))+|\theta_s|+|x_s|+\left|\frac{\lambda_s}{\lambda}\right|\right)}$.
Since $h \in G^\perp$, projecting \eqref{NLS_mod} in $\dot{H}^1$ onto $W$, $iW$, $\nabla W$ and $\Lambda W$ and integrating by parts (possible due to a standard regularization argument) yields
\begin{equation}
    |\alpha_s| + |\theta_s| + |x_s| + \left|\frac{\lambda_s}{\lambda}\right| = O(d+\epsilon(s)),
\end{equation}
which is enough to conclude \eqref{lem_equiv_2} and finish the proof of Lemma \ref{mod_2}.
\end{proof}

\begin{proof}[Proof of Lemma \ref{sub_mod_1}]
The proof is analogous to the proof of Lemma \ref{mod_1} and is omitted.\end{proof}

\begin{proof}[Proof of Lemma \ref{sub_mod_2}]

The orthogonality condition \eqref{sub_o2} implies $B(Q,h) = 0$. Since
$E(u) = E(Q)$, and $Q$ is a critical point for $E$, we have
\begin{equation}
    \alpha^2\Phi(Q)  + \Phi(h) = \Phi(\alpha Q + h) = O(|\alpha|^3+\|h\|_{H^1}^3).
\end{equation}
By coercivity, $\Phi(h)\approx \|h\|_{H^1}^2$, and hence, $|\alpha| \approx \|h\|_{H^1}$.
The relation $M(u) = M(Q)$ gives
\begin{equation}\label{sub_MuQ}
    \left|\alpha \int Q^2+ \int Q h_1\right| = \frac{1}{2}\int|\alpha Q + h|^2 = O(|\alpha|^2),
\end{equation}
and thus, 
\begin{equation}
    |\alpha| \approx \left|\int Qh_1\right|.
\end{equation}
Now, using \eqref{sub_o2}, 
\begin{equation}
    d(u) = \left|\int|\nabla u|^2-\int |\nabla Q|^2 \right| =  \left|2\left(\alpha\int|\nabla Q|^2-\int Qh_1 \right)+O(|\alpha|^2)\right|,
\end{equation}
which, together with \eqref{sub_MuQ}, gives
\begin{equation}
    d(u) = \left|2\alpha\left(\int|\nabla Q|^2-\int Q^2 \right)+O(|\alpha|^2)\right|.
\end{equation}
Since, by Pohozaev identities \eqref{pohozaev}, $\|\nabla Q\|_{L^2} \neq \|Q\|_{L^2}$ for any $N$ and any $p$ in the intercritical range, we conclude $d(u) \approx |\alpha|$, and hence, \eqref{sub_lem_equiv_1} holds. The rest of the proof goes along the same lines of the proof of Lemma \ref{mod_2} (without the need of self-similar variables), and is omitted.

\end{proof}

\subsection{Convergence to the standing wave above the ground state}

\begin{proof}[Proof of Lemma \ref{lem_int_exp_decay_d_2} (Critical case)]
\hspace{\fill}

\textit{Step 1. A general bound on {$A_R$} (recall \eqref{truncated_virial_error})}

By  the definition of $\phi_R$, we have the bounds $4\partial_r^2 \phi_R \leq 8$,  $ |\Delta^2 \phi_R| \lesssim 1$ and $|\Delta^2 \phi_R(r)| \lesssim \frac{1}{r^2}$. Therefore,
\begin{equation}\label{radial_AR}
    A_R(u(t)) \lesssim \int_{|x|\geq R}  |u(x,t)|^{2^*}+\frac{1}{R^2}|u(x,t)|^2\, dx.
\end{equation}
We now recall Strauss Lemma \cite{Strauss77} and make use of the decay given by radiality in $H^1$.

\begin{lemma}[\cite{Strauss77}]
There is a constant $C>0$ such that, for any radial function $f$ in $H^1(\Real^N)$ and any $R > 0$,
\begin{equation}
    \|f\|_{L^\infty_{\{|x|\geq R\}}} \leq \frac{C}{R^\frac{N-1}{2}}\|f\|_{L^2}^\frac{1}{2}\|\nabla f\|_{L^2}^\frac{1}{2}.
\end{equation}
\end{lemma}
We can now bound
\begin{equation}
    \int_{|x|\geq R} |u(x,t)|^{2^*}dx \leq \|u(t)\|^\frac{4}{N-2}_{L^\infty_{\{|x|\geq R\}}}\|u_0\|_{L^2}^2 \leq\frac{C}{R^\frac{2N-2}{N-2}}\|\nabla u(t)\|^\frac{2}{N-2}_{L^2}\|u_0\|_{L^2}^\frac{2N-2}{N-2},
\end{equation}

to obtain
\begin{equation}\label{bound_A_general}
    A_R(u(t)) \leq C_0\left[\frac{1}{R^2}+\frac{1}{R^\frac{2N-2}{N-2}}(d(t)+\|W\|_{\dot{H}^1})^\frac{1}{N-2}\right],
\end{equation}
where $C_0$ depends only on $M(u_0)$.

\textit{Step 2. A bound on {$A_R$} when $d(t)$ is small.}

Taking a small $\delta_1$, write the decomposition \eqref{decomp_alpha} as $u_{[x(t),\lambda(t),\theta(t)]} = W + v$, with $\|v\|_{\dot{H}1} \lesssim d(t)$, by Lemma \ref{mod_2}. We first prove 
\begin{equation}\label{lambda_m}
    \lambda_- := \inf \{ \lambda(t), \, t\geq 0, \, d(t) \leq \delta_1\} > 0.
\end{equation}
Indeed, by mass conservation,
\begin{equation}
    \|u_0\|_{L^2} \geq \int_{|x| \leq \lambda(t)} |u(t)|^2 = \frac{1}{\lambda^2(t)}\left(\int_{|x|\leq 1}W^2(x) dx -C d^2(t)\right).
\end{equation}
If $d(t) \leq \delta_1$ and $\delta_1$ is small enough, then \eqref{lambda_m} holds. We now give an estimate on $A_R$ when $d(t)$ is small. Since $W$ is a static solution to \eqref{NLS}, $\frac{d}{dt}\int \phi_R |W|^2 = 0$, so that $A_R(W) = 0$, for all $R > 0$, by \eqref{truncated_virial_2}. If we assume $R \geq 1$, by a change of variables, H\"older, Hardy and Sobolev inequalities, we can write \eqref{truncated_virial_error} as 
\begin{align}\label{bound_AR_small}
    |A_R(u(t))| &= |A_{R\lambda(t)}(W+v)| = |A_{R\lambda(t)}(W+v)-A_{R\lambda(t)}(W)| \\
    &\leq C \Biggl[ \int_{|x|\geq R\lambda(t)}|\nabla v|^2 + |\nabla W \cdot \nabla v| + W^{2^*-1}|v|+|v|^{2^*}+\frac{1}{(R\lambda(t))^2}(W|v|+|v|^2)\Biggr]\\
    &\leq C\left[ \|v\|^2_{\dot{H}^1}+\frac{1}{(R\lambda(t))^\frac{N-2}{2}}\|v\|_{\dot{H}^1}+ \|v\|^{2^*}_{\dot{H}^1}+\frac{1}{(R\lambda(t))^\frac{N+2}{2}}\|v\|_{\dot{H}^1}+
    \frac{1}{\lambda^2(t)}\|v\|^2_{\dot{H}^1}\right]\\
   \label{bound_A_small}&\leq C_1 \left[d^2(t)+\frac{1}{R^\frac{N-2}{2}}d(t)\right],
\end{align}

where we used the fact that $\|\nabla W\|_{L^2_{\{|x|\geq r\}}} \approx \|W\|_{L^{2^*}_{\{|x|\geq r\}}} \approx \frac{1}{r^\frac{N-2}{2}}$, which can be verified by explicit computation. Note that the constant $C_1$ depends only on $\lambda_-$, which in turn depends only on $M(u_0)$.

\textit{Step 3. Bounds on {$A_R$} prove bounds on {$d(t)$}.}

We now claim the bound
\begin{equation}\label{final_bound_AR}
    A_R(t) \leq \frac{8}{N-2}d(t).
\end{equation}

This follows from the bound \eqref{bound_A_small}, if $d(t) \leq \delta_1$ and $R\geq R_1$, where $R_1$ is a large constant depending only on $M(u_0)$. Now, if $d(t) > \delta_1$, consider the function
\begin{equation}
    \varphi_{R}(\delta) = \frac{C_0}{R^2}+\frac{C_0}{R^\frac{2N-2}{N-2}}(\delta+\|W\|_{\dot{H}^1})^\frac{1}{N-2}-\frac{8}{N-2}\delta.
\end{equation}

By direct computation, we see that $\varphi_{R}''(\delta)< 0$ for any $\delta > 0$. We can choose a large $R_2 \geq R_1$ (depending again only on $M(u_0)$) such that $\varphi_{R_2}(\delta_1) \leq 0$ and $\varphi_{R_2}'(\delta_1) \leq 0$, so that $\varphi_{R_2}(\delta) \leq 0$ for all $\delta \geq \delta_1$. The bound \eqref{final_bound_AR} is now proved.

Bound \eqref{final_bound_AR}, together with \eqref{truncated_virial_2}, gives, for $R\geq R_2$ and any $t \geq 0$,
\begin{equation}\label{sign_truncated_2der}
    F_R''(t) \leq -\frac{8}{N-2}d(t) <0.
\end{equation}

Note that we must have $F'_R(t) > 0$, for all $t\geq 0$, as \eqref{sign_truncated_2der} would otherwise contradict the positivity of $F_R$, thus, proving \eqref{positivity_2}. 

We now make use of the following claim, in the spirit of \cite{Banica}*{Lemma 2.1} and \cite{DR_Thre}*{Claim 5.4}.

\begin{claim}\label{cauchy_banica} Let $\phi \in C^1(\Real^N)$ and $f \in H^1(\Real^N)$. Assume that $\int |\nabla \phi|^2|f|^2 < +\infty$ and $E(f) = E(W)$. Then
\begin{equation}
    \left(\Im\int \nabla \phi \cdot \nabla f \,\bar{f} \right)^2 \lesssim d^2(f)\int|\nabla \phi |^2|f|^2.
\end{equation}
\end{claim}

By Claim \ref{cauchy_banica} and the fact that $F_R'(t) > 0$ and $F''(t) < 0$, we can write
\begin{equation} 
    \frac{F_R'(t)}{\sqrt{F_R(t)}}\lesssim -F_R''(t),
\end{equation}
so that 
\begin{equation}
\int_t^{+\infty}d(s) \, ds \lesssim e^{-ct},    
\end{equation}
which proves \eqref{int_exp_decay_d_2} and finishes the proof of Lemma \ref{lem_int_exp_decay_d_2}.
\end{proof}


\begin{proof}[Proof of Claim \ref{cauchy_banica}]
Let $\delta(f) = \int |\nabla W|^2-\int |\nabla f|^2$ and $\lambda \in \Real$. By Sobolev inequality
\begin{equation}
    \|\nabla(e^{i\lambda \varphi}f)\|_{L^2} \geq \frac{\|\nabla W\|_{L^2}}{\|W\|_{L^{2^*}}} \|f\|_{L^{2^*}}.
\end{equation}
Squaring the last inequality and expanding the term $\|\nabla(e^{i\lambda \varphi}f)\|_{L^2}$, we get
\begin{equation}
    \lambda^2 \int |\nabla \varphi|^2 |\nabla f|^2 + 2 \lambda \Im \int (\nabla \varphi \cdot \nabla f) \bar{f} + \int |\nabla f|^2 - \frac{\|\nabla W\|_{L^2}^2}{\|W\|_{L^{2^*}}^2} \|f\|_{L^{2^*}}^2 \geq 0.
\end{equation}
The discriminant of this quadratic form must be non-positive, thus,
\begin{equation}
    \left( \Im \int (\nabla \varphi \cdot \nabla f) \bar{f}\right)^2 \leq \left(\int |\nabla f|^2 - \frac{\|\nabla W\|_{L^2}^2}{\|W\|_{L^{2^*}}^2} \|f\|_{L^{2^*}}^2\right)\left(\int |\nabla \varphi|^2 |\nabla f|^2 \right).
\end{equation}
Since
\begin{equation}
    \int |\nabla f|^2 =  \int |\nabla W|^2-\delta(f),
\end{equation}
by $E(f) = E(W)$, we have,
\begin{equation}
   0 < \int |f|^{2^*} =  \int |W|^{2^*}-\frac{N}{N-2}\delta(f).
\end{equation}
Therefore,
\begin{align}
    \int |\nabla f|^2 - \frac{\|\nabla W\|_{L^2}^2}{\|W\|_{L^{2^*}}^2} \|f\|_{L^{2^*}}^2 &= \int |\nabla W|^2-\delta(f) -\frac{\|\nabla W\|_{L^2}^2}{\|W\|_{L^{2^*}}^2}\left( \int |W|^{2^*}-\frac{N}{N-2}\delta(f)\right)^\frac{N-2}{N}\\
    & = \int |\nabla W|^2-\delta(f) -\frac{\|\nabla W\|_{L^2}^2}{\|W\|_{L^{2^*}}^2}\left( \|W\|_{L^{2^*}}^2-\|W\|_{L^{2^*}}^{-\frac{4}{N-2}}\delta(f)+O(\delta(f)^2)\right)\\
    &=O(\delta(f)^2),
\end{align}
and Claim \ref{cauchy_banica} is proved.
\end{proof}

\begin{proof}[Proof of Proposition \ref{prop_supercrit}]

We first prove that 
\begin{equation}\label{d_to_0}
  \lim_{t \to +\infty} d(t) = 0. 
\end{equation}


Indeed, by Lemma \ref{lem_int_exp_decay_d_2}, there exists $t_n\to +\infty$ such that $d(t_n) \to 0$. Assume, by contradiction, that there exists a sequence $t_n' > t_n$ such that 
\begin{equation}\label{t_n_l}
    d(t_n') = \delta_0 \text{ and } 0 < d(t) < \delta_0 \quad \forall t \in (t_n, t_n'),
\end{equation}
where $\delta_0$ is given by Lemmas \ref{mod_1} and \ref{mod_2}. 
%
%
%
%
%
Recall the decomposition \eqref{decomp_alpha}:
\begin{equation}
    	u_{[\lambda(t),\theta(t)]}(t) = (1+\alpha(t))W+ h(t) \text{ with } h \in G^\perp.
\end{equation}

By taking subsequences, if necessary, we can assume
\begin{equation}
    \lim\lambda(t_n) = \lambda_\infty \in (0,+\infty].
\end{equation}

We now prove that $\lambda_\infty < +\infty$.

If $\lambda_\infty= +\infty$, as $u_{[\lambda(t_n),\theta(t_n)]}$ converges to $W$ in $\dot{H}^1$, we have, for any $C > 0$, 
\begin{equation}
    \int_{\left|x\right|\geq C}|u(t_n)|^{2^*} \to 0.
\end{equation}
For any $\epsilon>0$ we have, by Hölder inequality,
\begin{equation}
    |F_R(t_n)| \lesssim \epsilon \|u(t_n)\|_{\dot{H}^1}+\int_{\left|x \right|\geq C_\epsilon}|u(t_n)|^{2^*},
\end{equation}
so that
\begin{equation}
    \lim F_R(t_n) = 0.
\end{equation}
However, by \eqref{positivity_2}, $F_R'(t) > 0$. This implies $F_R(t) < 0$ for all $t \geq 0$, contradicting the fact that $F_R$ is positive. Therefore, $\lambda(t_n)$ must be bounded.

%
%
%
%

Now, by Lemma \ref{mod_2}, we have $\left|\frac{\lambda'(t)}{\lambda^3(t)}\right| \lesssim d(t)$. This implies, if $t \in (t_n, t_n')$,
\begin{equation}
    \left|\frac{1}{\lambda^2(t)}-\frac{1}{\lambda^2(t_n)}\right| \lesssim e^{-ct_n}.
\end{equation}
Therefore, $\lambda(t) \leq 2\lambda_\infty$ on $\displaystyle\cup_n(t_n,t_n')$, for large $t$.
Turning to the bound on $\alpha'$ in Lemma \ref{mod_2},
\begin{equation}
    |\alpha'(t)| \lesssim \lambda^2(t) d(t) \lesssim d(t).
\end{equation}
This implies $|\alpha(t_n) - \alpha(t_n')| \to 0$. Moreover, again by Lemma \ref{mod_2}, $|\alpha(t)| \approx d(t)$, which contradicts \eqref{t_n_l} and proves \eqref{d_to_0}.

To finish the proof of Proposition \ref{prop_supercrit}, we must refine the estimate on $d(t)$. Since $d(t) \to 0$ as $t \to +\infty$, the decomposition \eqref{decomp_alpha} is well-defined for all large times. Therefore, by \eqref{d_to_0} and \eqref{lambda_m}, we have
\begin{equation}
    \lim_{t\to+\infty}\lambda(t) = {\lambda_\infty} \in (0,+\infty),\quad
    \lim_{t\to +\infty}\alpha(t) = \lim_{t\to +\infty}d(t) =  0, 
\end{equation}
and
\begin{equation}
    \|h(t)\|_{\dot{H}^1} \approx|\alpha(t)| = \int_t^{+\infty}|\alpha'(t)|ds \lesssim \int_t^{+\infty}\lambda^2(s)d(s)ds \lesssim e^{-ct}.
\end{equation}
Furthermore, the bound $|\theta'(t)| \lesssim \lambda^2(t)d(t)$ implies that there exists $\theta_\infty$ such that
\begin{equation}
    \lim_{t\to+\infty} |\theta(t) -\theta_\infty| =0.
\end{equation}
Therefore, \eqref{exp_conv_W_super} is proven. 

It remains to prove the finite-time blow-up. This is a corollary of Lemma \ref{lem_int_exp_decay_d_2}, applied to the time-reversed solution $v(t) := \bar{u}(-t)$. If $u$ is defined on $\Real$, by \eqref{positivity_2}, we have
\begin{equation}
    \Im \int \nabla \phi \cdot \nabla u_0 \,\bar{u}_0 >0, 
\end{equation}
and
\begin{equation}
    \Im \int \nabla \phi \cdot \nabla v_0 \,\bar{v}_0 >0, 
\end{equation}
which clearly contradicts the fact that
\begin{equation}
    \Im \int \nabla \phi \cdot \nabla u_0 \,\bar{u}_0 = - \Im \int \nabla \phi \cdot \nabla v_0 \,\bar{v}_0. 
\end{equation}
This finishes the proof of Proposition \ref{supercrtit}.
\end{proof}

\begin{proof}[Proof of Proposition \ref{sub_prop_supercrit} (Intercritical case)]
We divide the proof in two cases: the finite-variance case, and the radial case. Using the same argument of the finite-time blow-up as in the energy-critical case, and in view of Lemmas \ref{sub_lem_exp_delta}, \ref{sub_lem_int_exp_decay_d} and \ref{sub_lem_int_exp_decay_d_2} in the next subsections, Proposition \ref{sub_prop_supercrit} follows.
\end{proof}
\subsubsection{Finite-variance solutions}
\begin{lemma}\label{sub_lem_int_exp_decay_d}
Let $u$ be a solution to \eqref{sub_NLS} defined on $[0,+\infty)$ and satisfying \eqref{sub_supercrtit} and  $\||x|u_0\|_{L^2}< +\infty$. Then, for all $t$ in the interval of existence of $u$,
\begin{equation}\label{sub_positivity}
\Im \int x\cdot\nabla u(t) \overline{u(t)} > 0,    
\end{equation}
and there exists $c > 0$ such that 
\begin{equation}\label{sub_int_exp_decay_d}
    \int_t^{+\infty}d(s)ds \lesssim e^{-ct},\quad \forall t \geq 0.
\end{equation}
\end{lemma}

\begin{proof}
Let $F(t) = \int|x|^2|u(x,t)|^2\, dx$. Then, by the virial identities, we have, for all $t\geq0$,
\begin{equation}
    F'(t) = 4\Im \int x\cdot\nabla u(t) \overline{u}(t).
\end{equation}
Note that, by Cauchy-Schwarz, $F'(t)$ is well-defined. Furthermore,
since $E(u)=E(Q)$,
\begin{equation}
    F''(t) = -[2N(p-1)-8] \left(\int |\nabla u|^2-\int |u|^{p+1}\right) = -[2N(p-1)-8]d(u(t)).
\end{equation}
Now, if \eqref{sub_positivity} does not hold, there exists $t_1$ such that $F'(t_1) \leq 0$. Since $F'' \leq 0$, for any $t_0 > t_1$,
\begin{equation}
    F'(t) \leq F'(t_0) < 0, \quad \forall t \geq t_0.
\end{equation}
This implies that $F(t) < 0$ for large $t$, contradicting the definition of $F$.

We now claim that
\begin{equation}\label{sub_diff_eq_F}
    [F'(t)]^2 \lesssim F(t) [F''(t)]^2,
\end{equation}
which is a consequence of the following claim, which is proved similarly to Claim \ref{cauchy_banica}.

\begin{claim}\label{sub_cauchy_banica} Let $\phi \in C^1(\Real^N)$ and $f \in H^1(\Real^N)$. Assume that $\int |\nabla \phi|^2|f|^2 < +\infty$, $M(f) = M(Q)$ and $E(f) = E(Q)$. Then
\begin{equation}
    \left(\Im\int \nabla \phi \cdot \nabla f \,\bar{f} \right)^2 \lesssim d^2(f)\int|\nabla \phi |^2|f|^2.
\end{equation}
\end{claim}

Taking $\phi(x) = |x|^2$ yields \eqref{sub_diff_eq_F} is proved.

Since $F'(t) > 0$ and $F''(t) < 0$ for all $t \geq 0$, the equation \eqref{sub_diff_eq_F} can be rewritten as
\begin{equation}
    \frac{F'(t)}{\sqrt{F(t)}} \lesssim -F''(t).
\end{equation}
Integrating from $0$ to $t \geq 0$,
\begin{equation}
    \sqrt{F(t)}-\sqrt{F(0)} \lesssim -(F'(t)-F'(0)) \leq F'(0).
\end{equation}
From \eqref{sub_diff_eq_F}, we deduce
\begin{equation}
    F'(t) \lesssim -\left(\sqrt{F(0)}+F'(0)\right) F''(t) \lesssim -F''(t),
\end{equation}
which shows
\begin{equation}
    F'(t) \lesssim e^{-ct}.
\end{equation}
Finally,
\begin{equation}
F'(t) = - \int_t^{+\infty}F''(s) ds = 4\int_t^{+\infty}d(s)ds,    
\end{equation}
producing \eqref{sub_int_exp_decay_d}, which proves Lemma \ref{sub_lem_int_exp_decay_d}.
\end{proof}
\subsubsection{Radial solutions}

We now work with a truncated variance. Consider a radial function $\phi \in C^\infty_0(\Real^N)$ such that
\begin{equation}
  \quad \phi(r) \geq 0  \quad \forall r \geq 0,
\end{equation}
\begin{equation}
    \phi(r) = 
    \begin{cases} 
    r^2, &r \leq 1,\\
    0, & r \geq 3,
    \end{cases}
\end{equation}

and
\begin{equation}
    \frac{d^2\phi}{dr^2}(r)\leq 2, r \geq 0.
\end{equation}

Define $\phi_R(x) = R^2\phi(\frac{x}{R})$ and 
\begin{equation}
    F_R(t) = \int \phi_R |u(t)|^2.
\end{equation}
By virial identities, if $M(u_0) = M(Q)$ and $E(u_0) = E(Q)$, we have
\begin{align}
\label{sub_truncated_virial_1}F_R'(t) &= 2 \Im \int \nabla \phi_R \cdot \nabla u \,\bar{u},\\
\label{sub_truncated_virial_2}F''_R(t)&= -[2N(p-1)-8]d(t)+A_R(t),
\end{align}
where
\begin{equation}\label{sub_truncated_virial_error}
    A_R(u(t)) = \int_{|x|\geq R}|\nabla u(t)|^2\left(4\partial_r^2\phi_R-8\right)+\frac{2(p-1)}{(p+1)}\int_{|x|\geq R}|u|^{p+1}\left(2N-\Delta \phi_R\right)-\int_{|x|\geq R} |u|^2 \Delta^2\phi_R.
\end{equation}

The following lemma holds.

\begin{lemma}\label{sub_lem_int_exp_decay_d_2}
Let $u$ be radial a solution to \eqref{sub_NLS} defined on $[0,+\infty)$ and satisfying \eqref{sub_supercrtit}. Then, there exists a constant $R_0 >0$ such that, for all $t$ in the interval of existence of $u$ and all $R\geq R_0$,
\begin{equation}\label{sub_positivity_2}
F_R'(t) > 0,    
\end{equation}
and there exists $c > 0$ such that 
\begin{equation}\label{sub_int_exp_decay_d_2}
    \int_t^{+\infty}d(s)ds \lesssim e^{-ct},\quad \forall t \geq 0.
\end{equation}
Moreover, $u_0$ has finite variance.
\end{lemma}
\begin{proof}
The proof of Lemma \ref{sub_lem_int_exp_decay_d_2} is essentially the same as in the energy-critical case, and is omitted, except for the finite-variance part.

By hypothesis, there is a sequence $t_n \to +\infty$ such that $d(t_n) \to 0$. By \eqref{sub_var_car_H1}, extracting a subsequence, if necessary, we have $u_n \to e^{i\theta_0}Q$ in $H^1$ for some $\theta_0 \in \Real$. Since $F_R$ is increasing by \eqref{sub_positivity_2}, we have
\begin{equation}
    \int \phi_R|u_0|^2 = F_R(0) \leq \int \phi_R Q^2.
\end{equation}

Thus, we can make $R\to +\infty$, which proves the finite variance of $u_0$.
\end{proof}

\begin{proof}[Proof of Lemma \ref{sub_cauchy_banica}]

The proof is analogous to the proof of Lemma \ref{cauchy_banica} and, thus, omitted.
\end{proof}

\subsection{Convergence to the standing wave below the ground state}
As in the proof of Proposition \ref{prop_supercrit}, we need to show that
\begin{equation}
    d(t) = \|W\|_{\dot{H}^1} - \|u(t)\|_{\dot{H}^1} \to 0 \text{ as } t\to +\infty.
\end{equation}

We start by stating the following monotonicity results.
\begin{lemma}\label{monot0}
Consider $\{t_n\}_n$ and $\{t'_n\}_n$, $t_n < t'_n$  two real sequences, and $\{u_n\}_n$ a sequence of radial solutions to \eqref{NLS} on $[t_n,t'_n]$ satisfying \eqref{subcrtit}. Assume that there exist $\{\lambda_n(t)\}_n \subset R_+^*$ such that the set
\begin{equation}
    K =\left\{ (u_n(t))_{[\lambda_n(t)]}: \, n \in \mathbb{N}, \, t \in [t_n,t'_n]\right\}
\end{equation}
is relatively compact in $\dot{H}^1$. If
\begin{equation}
    \lim_n d(u_n(t_n)) + d(u_n(t'_n)) = 0,
\end{equation}
then
\begin{equation}\label{sup_to_0}
    \lim_n \left\{ \sup_{t \in [t_n^{},t'_n]}d(u_n(t))\right\} = 0.
\end{equation}
\end{lemma}
\begin{lemma}\label{bound_mu} Under the assumptions of Lemma \ref{monot0}, if $n$ is large enough so that $d(u_n(t)) \leq \delta_0$ on $[t^{}_n, t'_n]$ and if $\theta_n$, $\mu_n$ and $\alpha_n$ are the parameters of the decomposition \eqref{decomp_alpha}, then
\begin{equation}\label{sup_inf}
    \lim_n \frac{\displaystyle\sup_{t \in [t_n^{},t_n']} \mu_n(t)}{\displaystyle\inf_{t \in [t_n^{},t_n']} \mu_n(t)} = 1.
\end{equation}

\end{lemma}
\begin{re}\label{inf_lambda}
In Lemmas \ref{monot0} and \ref{bound_mu}, it is sufficient to assume
\begin{equation}
    \inf_{t \in [t_n^{},t_n']} \lambda(t) = 1 \quad \forall n \in \mathbb{N}.
\end{equation}
In fact, if $\tilde{\lambda}_n := \inf_{t \in [t_n^{},t_n']} \lambda(t)$, then
\begin{equation}
    u^*_n (x,t) := \frac{1}{\tilde{\lambda}_n^{\frac{N-2}{2}}} u_n\left(\frac{x}{\tilde{\lambda}_n^{}},\frac{t}{\tilde{\lambda}_n^2} \right), \quad \lambda^*_n(t) := \frac{\lambda_n(t)}{\tilde{\lambda}_n}, \quad t_n^* := \frac{t_n}{\tilde{\lambda}^2_n}, \quad t'^{*}_n := \frac{t'_n}{\tilde{\lambda}^2_n},
\end{equation}
\begin{equation}
    K^*  :=\left\{ (u_n^*(t))_{[\lambda^*_n(t)]}: \, n \in \mathbb{N}, \, t \in [t_n^*,t'^{*}_n]\right\}
\end{equation}
satisfy the assumptions of Lemma \ref{monot0}. Moreover, the conclusions of the Lemmas are unchanged under these transformations.
\end{re}

Before proving Lemmas \ref{monot0} and \ref{bound_mu}, we prove two auxiliary lemmas.
\begin{lemma}\label{monot1}
Consider $\{t_n\}_n$ and $\{t'_n\}_n$, $t_n < t'_n$  two real sequences, and $\{u_n\}_n$ a sequence of radial solutions to \eqref{NLS} on $[t_n,t'_n]$ satisfying \eqref{subcrtit}. Assume that there exist $\{\lambda_n(t)\}_n \subset R_+^*$ such that the set
\begin{equation}
    K =\left\{ (u_n(t))_{[\lambda_n(t)]}: \, n \in \mathbb{N}, \, t \in [t_n,t'_n]\right\}
\end{equation}
is relatively compact in $\dot{H}^1$. Assume furthermore that 
\begin{equation}\label{inf_lambda1}
    \inf_{t \in [t_n,t'_n]} \lambda(t) = 1 \quad \forall n \in \mathbb{N}.
\end{equation}
Then, for all $n \in \mathbb{N}$,
\begin{equation}\label{subcrit_bound_int}
    \int_{t_n}^{t'_n}d(u(t)) \lesssim d(u(t_n))+d(u(t'_n)).
\end{equation}
\end{lemma}

\begin{proof}[Proof of Lemma \ref{monot1}]
For $R>0$, consider the function
\begin{equation}
    F_{R,n}(t) = \int \phi_R |u_n(t)|^2. 
\end{equation}
By H\"older and Sobolev inequalities, and recalling that $\|u(t)\|_{\dot{H}^1} \leq \|W\|_{\dot{H}^1}$, we have
\begin{equation}
    F_{R,n}(t) \lesssim R^4.
\end{equation}
By Lemma \ref{cauchy_banica},
\begin{equation}\label{subcrit_bound_Fl}
    |F'_{R,n}(t)|\lesssim d(u_n(t)) \sqrt{F_{R,n}(t)} \lesssim R^2 d(u_n(t)).
\end{equation}

By \eqref{inf_lambda1}, $\lambda(t) \geq 1$ on $[t_n, t'_n]$. We claim that, whenever defined, $\mu_n$ is bounded away from zero. In fact, by the precompactness of $K$ and decomposition \eqref{decomp_alpha}, we have
\begin{equation}
    (u_n(t))_{[\lambda_n(t)]} = (1+\alpha_n(t))W_{[\lambda_n(t)/\mu_n(t)]}+(h_n(t))_{[\lambda_n(t)/\mu_n(t)]}
\end{equation}
 with $(h_n(t))_{[\lambda_n(t)/\mu_n(t)]} \perp W_{[\lambda_n(t)/\mu_n(t)]}$ and $\alpha_n(t) \leq \|u_{[\lambda_n(t)]}\|_{\dot{H}^1}+1$. Therefore, the set
\begin{equation}\label{W_compact}
    \displaystyle\bigcup_n \left\{ W_{[\lambda_n(t)/\mu_n(t)]}:\, t \in [t_n,t'_n], d(u_n(t)) \leq \delta_0\right\}
\end{equation}
must be precompact. Since $W$ does not depend on time, we get
\begin{equation}\label{lambda_mu_equiv}
    \lambda_n(t) \approx \mu_n(t) \text{ on } \{ t \in [t_n,t'_n], d(u_n(t)) \leq \delta_0 \}.
\end{equation}
(Note that the constant does not depend on $n$). Thus, defining $\mu_- = \displaystyle\inf_{\substack{ t \in [t_n,t'_n],\\ d(u_n(t)) \leq \delta_0 }}\mu_n(t)\gtrsim 1$.

We will now give a lower bound to $F''_{R,n}(t)$. Recalling \eqref{bound_AR_small}, if $d(u_n(t)) \leq \delta_0$ and $R \geq \frac{1}{\mu_-}$, we have
\begin{equation}
    |A_R(u_n(t))| \lesssim\left[d^2(u_n(t))+\frac{1}{(R\mu_-)^\frac{N-2}{2}}d(u_n(t))\right].
\end{equation}
Therefore, there exist $\delta_1 > 0$ and $R_1 > 0$ such that, if $d(u_n(t)) \leq \delta_1$, then
\begin{equation}
    |A_R(u_n(t))| \leq \frac{8}{N-2}d(u_n(t)).
\end{equation}
Now, by almost periodicity modulo symmetries and \eqref{inf_lambda1}, if $\eta >0$ and $R \geq C(\eta)$, then 
\begin{equation}
     |A_R(u_n(t))| \lesssim \eta.
\end{equation}
Thus, we can choose $\eta_1 = \eta_1(\delta_1)$ such that, if $d(u_n(t)) \geq \delta_1$ and $R\geq C(\eta_1)$, then
\begin{equation}
     |A_R(u_n(t))| \leq \frac{8}{N-2}\delta_1 \leq \frac{8}{N-2}d(u_n(t)).
\end{equation}
Finally, since
\begin{equation}
    F_{R,n}''(t) = \frac{16}{N-2}d(u_n(t)) + A_R(u_n(t)),
\end{equation}
we get, if $R \geq \max\{R_1,C(\eta_1)\}$,
\begin{equation}\label{subcrit_bound_Fll}
    F_{R,n}''(t) \geq \frac{8}{N-2}d(u_n(t)).
\end{equation}
Integrating \eqref{subcrit_bound_Fll} and using \eqref{subcrit_bound_Fl}, we obtain \eqref{subcrit_bound_int}.
\end{proof}
\begin{lemma}\label{monot2} Under the assumptions of Lemma \ref{monot0} and Remark \ref{inf_lambda}, if $s_n \in [t_n,t'_n]$ and the sequence $\lambda_n(s_n)$ is bounded, then
\begin{equation}\label{d_not_0}
    \lim_n d(u_n(s_n)) = 0.
\end{equation}
\end{lemma}
\begin{proof}
By Remark \ref{inf_lambda}, we have $\lambda_n(s_n) \approx 1$, hence, we can assume that the sequence $\{u_n(s_n)\}_n$ converges to some $v_0 \in \dot{H}^1$. If \eqref{d_not_0} does not hold, then
\begin{equation}\label{contrad_v0}
    d(v_0) > 0 \text{ and } \|v_0\|_{\dot{H}^1} < \|W\|_{\dot{H}^1}.
\end{equation}
By strong convergence, we have $E(v_0) = E(W)$. Let $v$ be the solution to \eqref{NLS} with initial condition $v_0$. By Proposition \ref{sub_global}, $v(t)$ is defined for any $ t \in \Real$. 

We claim that, for large $n$, $s_n+1 \leq t'_n$. Indeed, if $t'_n \in (s_n,s_n+1)$ for an infinite number of $n$, after extracting a subsequence, we have that $t'_n-s_n$ converges to some $\tau \in [0,1]$. By continuity of the flow, $u_n(t'_n) \to v(\tau)$. But since $d(u_n(t'_n)) \to 0$, $d(v(\tau)) = 0$, which implies that $v = W_{[\lambda_0,\theta_0]}$, for some fixed $\lambda_0,\theta_0$. Uniqueness of solutions to \eqref{NLS} then contradicts \eqref{contrad_v0}. Therefore, for large $n$, $t_n \leq s_n \leq s_n+1 \leq t_n'$. Again by continuity of the flow,
\begin{equation}
    \lim_n \int_{s_n}^{s_n+1}d(u_n(t)) \, dt = \int_0^1 d(v(t)) \, dt > 0.
\end{equation}
However, Lemma \ref{monot1} gives
\begin{equation}
    \lim_n \int_{s_n}^{s_n+1}d(u_n(t))\, dt \leq \lim_n \int_{t_n}^{t'_n}d(u_n(t))\, dt \lesssim \lim_n d(u_n(t_n)) + d(u_n(t'_n)) = 0,
\end{equation}
completing Lemma \ref{monot2}.
\end{proof}
We now prove Lemmas \ref{monot0} and \ref{bound_mu}.
\begin{proof}[Proof of Lemma \ref{monot0}]
By Remark \ref{inf_lambda}, we can choose, for every $n$, $b_n \in [t_n,t'_n]$ such that 
\begin{equation}
    \lim_n \lambda_n(b_n) = 1.
\end{equation}
This implies, by Lemma \ref{monot2}, that 
\begin{equation}
    \lim_n d(u_n(b_n)) = 0.
\end{equation}
Assume, by contradiction, that \eqref{sup_to_0} does not hold. Without loss of generality, there exists $\delta_1  > 0$ such that
\begin{equation}\label{an}
    \sup_{t \in [t_n,b_n]} d(u_n(t)) \geq \delta_1,\quad \forall n \in \mathbb{N}.
\end{equation}
Choosing $\delta_2 < \min\{ \delta_0,\delta_1\}$, by continuity there exists $a_n \in [t_n,b_n)$ such that 
\begin{equation}
    d(u_n(t)) < \delta_2 \text{ on }(a_n,b_n) \text{ and } d(u_n(a_n)) = \delta_2.
\end{equation}
Since $\delta_2 < \delta_0$, the modulation parameter $\mu_n$ is well-defined. Recalling that the set defined by \eqref{W_compact} is precompact, we must have $\lambda_n \approx \mu_n$, where the constants do not depend on $n$. Thus, up to a subsequence, we can assume 
\begin{equation}
    \mu_n(b_n) \to \mu_0 \in (0,+\infty) \text{ as } n \to +\infty.
\end{equation}
We will now show that the $\mu_n$ are uniformly bounded on $\cup_n[a_n,b_n]$. Suppose, by contradiction, that there exists $c_n \in [a_n,b_n)$ such that, for large $n$,
\begin{equation}\label{mu_cn}
    \mu_n(t) < 2\mu_0 \text{ on } (c_n,b_n) \text{ and } \mu_n(c_n) = 2\mu_0.
\end{equation}
Since $\mu_n(c_n)$ is bounded, so is $\lambda_n(c_n)$. Therefore, by Lemma \ref{monot2}, $\displaystyle\lim_n d(u_n(c_n)) = 0$. Recalling Lemma \ref{mod_2}, we have
\begin{equation}
    \left| \frac{1}{\mu_n^2(c_n)}-\frac{1}{\mu_n^2(b_n)}\right| \leq \int_{c_n}^{b_n} \left|\frac{\mu'_n(t)}{\mu_n^3(t)} \right| \lesssim \int_{c_n}^{b_n} d(u_n(t)) \, dt.
\end{equation}
By Lemma \ref{monot1}, the last integral converges to $0$, contradicting \eqref{mu_cn}. Therefore, 
\begin{equation}
    \sup_{\substack{t \in [a_n, b_n]\\n \in \mathbb{N}}} \mu_n(t) < +\infty.
\end{equation}
We conclude that $\mu_n(a_n)$ must be bounded, and so must be $\lambda_n(a_n)$. Invoking again Lemma \ref{monot2}, we have $\displaystyle\lim_n d(u_n(a_n)) = 0$, contradicting \eqref{an}. Lemma \ref{monot0} is proven.
\end{proof}
\begin{proof}[Proof of Lemma \ref{bound_mu}]
As in the proof of the previous Lemma, by Remark \ref{inf_lambda} and Lemmas \ref{monot1} and \ref{monot2}, we have that $\mu_n \approx 1$, where the constant does not depend on $n$. Let $a_n$ and $b_n$ be such that
\begin{equation}
    \mu_n(a_n) = \inf_{t \in [t_n,t_n']} \mu_n(t) \text{ and }  \mu_n(b_n) = \sup_{t \in [t_n,t_n']}\mu_n(t).
\end{equation}
Let $\bar{a}_n = \min\{a_n,b_n\}$ and $\bar{b}_n = \max\{a_n,b_n\}$. Then,
\begin{equation}
    \left| \frac{1}{\mu_n^2(a_n)}-\frac{1}{\mu_n^2(b_n)}\right| \leq \int_{\bar{a}_n}^{\bar{b}_n} \left|\frac{\mu'_n(t)}{\mu_n^3(t)} \right| \lesssim \int_{\bar{a}_n}^{\bar{b}_n} d(u_n(t)) \, dt \to 0  \text{ as } n \to +\infty.
\end{equation}
Since $\mu_n(b_n)$ is bounded, we get \eqref{sup_inf}, proving Lemma \ref{bound_mu}.
\end{proof}
We now have all the tools to prove Proposition \ref{prop_subcrit}.
\begin{proof}[Proof of Proposition \ref{prop_subcrit}]
By Corollary \ref{cor_mean_subcrit}, there exists a sequence $t_n \to +\infty$ such that
\begin{equation}
    \lim_n d(u(t_n)) = 0.
\end{equation}
By Lemma \eqref{monot0}, with $u_n = u$, $\lambda_n = \lambda$ (where $\lambda$ is the frequency scale obtained from Proposition \ref{sub_global}) and $t_n' = t_{n+1}$,  this implies 
\begin{equation}\label{d_to_0_final}
    \lim_{t \to +\infty} d(t) = 0.
\end{equation}
Therefore, the modulation parameters $\alpha(t)$, $\mu(t)$, $\theta(t)$ are defined for large $t$.
We now prove that 
\begin{equation}\label{mu_bounded}
    \lim_{t \to +\infty} \mu(t) = \mu_{\infty} \in (0, +\infty).
\end{equation}
Indeed, if not, then as $t \to +\infty$, $\log(\mu(t))$ does not satisfy the Cauchy criterion. Therefore, there must exist sequences $\{T_n\}$ and $\{T_n'\}$ such that $T_n < T_n'$ and 
\begin{equation}\label{mu_Tn_1}
    \lim_n \frac{\mu(T_n')}{\mu(T_n)} \neq 1.
\end{equation}
But $d(T_n) + d(T_n') \to 0$ by \eqref{d_to_0_final}. By Lemma \eqref{bound_mu}, with $u_n = u$, $\lambda_n = \lambda$, $t_n = T_n$ and $t_n' = T_n'$,  we have
\begin{equation}
    \lim_n \frac{\sup_{t \in [T_n,T_n']}\mu(t)}{\inf_{t \in [T_n,T_n']}\mu(t)} = 1,
\end{equation}
contradicting \eqref{mu_Tn_1}.
Turning to the proof of \eqref{exp_conv_sub}, we claim the following inequality
\begin{equation}\label{int_exp}
    \int_t^{+\infty}d(u(s)) \, ds \lesssim d(u(t)).
\end{equation}
Suppose by contradiction that \eqref{int_exp} does not hold. Then there exists a sequence $T_n \to +\infty$ such that 
\begin{equation}
    \int_{T_n}^{+\infty} d(u(s)) \, ds \geq 2n \, d(u(T_n)).
\end{equation}
Moreover, there exists a sequence $\{S_n\}_n$ such that $S_n>T_n$ for all $n$, and
\begin{equation}
    \int_{T_n}^{S_n} d(u(s)) \, ds \geq n \, d(u(T_n)).
\end{equation}
By \eqref{mu_bounded}, for any sequence $\{T_n'\}_n$ such that $T_n' \geq S_n$ for all $n$, we are under the assumptions of Lemma \ref{monot1}, with $u_n = u$, $\lambda_n = \lambda$, $T_n= t_n$ and $t_n' = T_n'$,  Hence, 
\begin{equation}
     n \, d(u(T_n)) \leq \int_{T_n}^{T_n'} d(u(s)) \, ds \lesssim \, d(u(T_n))+d(u(T_n')). 
\end{equation}
Since $T_n'$ can be taken arbitrarily large, and the implicit constant is independent of the choice of a particular $\{T_n'\}_n$ (given the function $u$ itself does not change), we have a contradiction.

Note that \eqref{int_exp} is equivalent to the existence of $c>0$ such that
\begin{equation}
    \int_t^{+\infty} d(u(s)) \, ds \lesssim e^{-ct}.
\end{equation}
By Lemma \ref{mod_2}, since $|\alpha(t)| \approx d(u(t))$ and $\mu$ is bounded, there exist $\theta_{\infty}$ such that
\begin{equation}
    |\alpha(t)| + |\theta(t)-\theta_\infty|+\|h(t)\|_{\dot{H}^1} \lesssim e^{-ct}.
\end{equation}
Therefore, the bound \eqref{exp_conv_sub} is proven. The assertion about scattering for negative times is a corollary of Lemma \ref{monot1}. Indeed, if
\begin{equation}
    \|u\|_{S(0,+\infty)} = \|u\|_{S(-\infty,0)} = +\infty,
\end{equation}
by time-reversal and \eqref{exp_conv_sub}, we see that the set 
\begin{equation}
    \left\{u(t): \, t \in \Real\right\} 
\end{equation}
is relatively compact and that
\begin{equation}
\lim_{t \to \pm\infty} d(t) = 0.
\end{equation}
Therefore, by Lemma \ref{monot1}, with $u_n = u$, $\lambda_n = 1$, $t_n = -n$ and $t_n' = n$, we have
\begin{equation}
    \int_{-\infty}^{+\infty}d(t)\, dt = \lim_{n \to +\infty} \int_{-n}^n d(t) \, dt \lesssim d(-n)+d(n) = 0. 
\end{equation}
Therefore, $d(u_0)=0$, contradicting \eqref{subcrtit}. Proposition \ref{prop_subcrit} is proven.
\end{proof}

For the intercritical case, as in the proof of Proposition \ref{sub_prop_supercrit}, we need to show that
\begin{equation}\label{sub_exp_conv_sub}
    \int_{t}^{+\infty} d(s)\, ds \lesssim e^{-ct}, \, \forall t \geq 0.
\end{equation}

We start with the following lemmas.

\begin{lemma}\label{sub_mean_subcrit}
Let $u$ be a solution to \eqref{sub_NLS} satisfying \eqref{sub_subcrtit} and \eqref{sub_noscatter}. Then
\begin{equation}
    \lim_{T\to +\infty} \frac{1}{T}\int_0^T d(t) \, dt = 0.
\end{equation}
\end{lemma}

We next state a key result to prove Proposition \ref{sub_exp_conv_sub}. 
\begin{lemma}\label{sub_monot1}
Let $u$ be a solution to \eqref{sub_NLS} satisfying \eqref{sub_subcrtit} and \eqref{sub_noscatter}, and $x(t)$ as in Proposition \ref{sub_sub_global} and Remark \ref{sub_re_x}. Then, for any $0 \leq \sigma \leq \tau$,
\begin{equation}\label{sub_subcrit_bound_int}
    \int_{\sigma}^{\tau}d(u(t)) \lesssim \left[1+\sup_{\sigma \leq t \leq \tau} |x(t)|\right](d(u(\sigma))+d(u(\tau))),
\end{equation}
and, if $\tau \geq \sigma +1$,
\begin{equation}\label{sub_subcrit_bound_x}
    |x(\tau)-x(\sigma)| \lesssim \int_{\sigma}^{\tau}d(u(t)). 
\end{equation}
\end{lemma}

The proof of \eqref{sub_subcrit_bound_int} is similar to the energy-critical setting (it is in fact easier, since there is no scaling involved). We refer to \cite{DR_Thre}*{Lemma 6.7} for the argument in the 3d cubic case. The proof of \eqref{sub_subcrit_bound_x} follows verbatim from the proof in \cite{DR_Thre}*{Lemma 6.8}.

\begin{proof}[Proof of Lemma \ref{sub_mean_subcrit}]
Let $R> 0$ to be chosen later and let $\phi_R$ and $F_R$ be as in the previous section. Then, by H\"older and inequatity,
\begin{equation}\label{sub_bound_FR}
    |F'_R(t)| \lesssim R.
\end{equation}
Moreover, we have
\begin{equation}\label{sub_virial_compl_nonrad}
    F_R''(t) = [2N(p-1)-8]d(t) + A_R(u(t)),
\end{equation}
where $A_R$ is given by \eqref{sub_truncated_virial_error}.

Fix $\eta > 0$. By definition of $\phi_R$ and almost periodicity modulo symmetries, if $R \geq C(\eta)$, we have
\begin{align}\label{sub_virial_nonrad}
    |A_R(u(t))| &\lesssim \int_{|x| \geq R} |\nabla u(x,t)|^2 + |u(x,t)|^{p+1}+\frac{1}{|x|^2}|u(x,t)|^2\, dx. 
\end{align}

Choose $T_0(\eta) \geq 0$ such that, for any $t \geq T_0$,
\begin{equation}
    |x(t)| \leq \eta t.
\end{equation}

For $T \geq T_0$, choose $R = \eta T +C(\eta)+1$. With this choice of $R$, we have
\begin{align}
    |A_R(u(t))| &\lesssim \int_{|x-x(t)|+|x(t)| \geq R} |\nabla u(x,t)|^2 + |u(x,t)|^{p+1}+|u(x,t)|^2\, dx \\
    &\lesssim \int_{|x-x(t)| \geq C(\eta)} |\nabla u(x,t)|^2 + |u(x,t)|^{p+1}+|u(x,t)|^2\, dx \\
    &\lesssim \eta\,.
\end{align}

By \eqref{sub_bound_FR}, \eqref{sub_virial_compl_nonrad}, and \eqref{sub_virial_nonrad},
\begin{align}
    [2N(p-1)-8]\int_{T_0}^Td(t) \, dt &\lesssim |F'_R(T)| + |F'_R(T_0)| + \eta(T-T_0)\\
    &\lesssim R + \eta(T-T_0)\\
    &=  \eta T+ \eta(T-T_0).
\end{align}
Letting $T \to +\infty$, we deduce
\begin{equation}
    \limsup_{T \to +\infty} \frac{1}{T}\int_0^Td(t) \, dt \lesssim \eta.
\end{equation}
Since $\eta$ is arbitrary, we conclude the proof of Lemma \ref{sub_mean_subcrit}.
\end{proof}

We are now able to prove Proposition \ref{sub_prop_subcrit}, following the proof in \cite{DR_Thre}.
\begin{proof}[Proof of Proposition \ref{sub_prop_subcrit}]
We first show that $x(t)$ is bounded. By Lemma \ref{sub_mean_subcrit}, there exists a sequence $\{t_n\}_n$ such that $t_{n+1} \geq t_n +1$ for all $n$, and $d(u(t_n)) \to 0$. By Lemma \ref{sub_monot1}, there exists $C_0 > 0$ such that, if $n > n_0$ and $1+t_{n_0} \leq t \leq t_n$, then
\begin{equation}
    |x(t)-x(t_{n_0})| \leq C_0 \left[1+\sup_{t_{n_0}\leq s \leq t_n}|x(s)| \right](d(u(t_n))+d(u(t_{n_0}))).
\end{equation}

If $n_0$ is large enough so that $d(u(t_n))+d(u(t_{n_0})) \leq 1/(2C_0)$, and $t$ is chosen in $[t_{n_0}+1,t_n]$ so that $\sup_{t_{n_0}+1\leq s \leq t_n}|x(s)| = |x(t)|$, then
\begin{equation}
    \sup_{t_{n_0}+1\leq s \leq t_n}|x(s)| \leq C(n_0) + \frac{1}{2}\sup_{t_{n_0}+1\leq s \leq t_n}|x(s)|,
\end{equation}

where $C(n_0) =  |x(t_{n_0})| + \frac{1}{2} \sup_{t_{n_0}\leq s \leq t_{n_0}+1}|x(t)|+\frac{1}{2}$. Therefore, $x(t)$ is bounded on $[t_{n_0}+1,+\infty)$, and hence, by continuity, on $[0,+\infty)$.

By the boundedness of $x(t)$ and \eqref{sub_subcrit_bound_int}, we have
\begin{equation}
    \int_{\sigma}^{\tau}d(u(t)) \lesssim d(u(\sigma))+d(u(\tau)).
\end{equation}
For a fixed $\sigma \geq 0$ and choosing $\tau = t_n$, we let $n \to +\infty$ to obtain
\begin{equation}
    \int_{\sigma}^{\infty}d(u(t)) \lesssim d(u(\sigma)).
\end{equation}
By Gronwall's Lemma, we have \eqref{sub_exp_conv_sub} and, by Lemma \ref{sub_lem_exp_delta}, we finish the proof of Proposition \ref{sub_prop_subcrit}.
\end{proof}

\subsection{Results for the linearized equation}

\begin{proof}[Proof of Lemma \ref{lem_self}] By Lemma \ref{lem_stric_exp}, we can assume that
\begin{equation}\label{local_boot}
    \|h(t)\|_{S(L^2,\,[t,+\infty))} \lesssim e^{-c_0t}.
\end{equation} We first normalize the eigenfunctions of $\mathcal{L}$. Define
\begin{equation}
	f_0 = \frac{iW}{\|W\|_{\dot{H}^1}}, \quad f_k = \frac{\partial_k W}{\|\partial_k W\|_{\dot{H}^1}}, 1 \leq k \leq N,  
	\quad \text{and}\quad f_{N+1} = \frac{\Lambda W}{\|\Lambda W\|_{\dot{H}^1}}.
\end{equation}
We have
\begin{equation}
	B(f_k, h) = 0, \quad \|f_k\|_{\dot{H}^1}, \quad \forall 0 \leq k \leq N+1, \, \forall h \in \dot{H}^1.
\end{equation}
Recall that $B(\mathcal{Y}_+, \mathcal{Y}_-)\neq 0$. Normalize $\mathcal{Y}_+$, $\mathcal{Y}_-$ such that $B(\mathcal{Y}_+, \mathcal{Y}_-)=1$. Next, write
\begin{gather}
\label{def_h}h(t) = \alpha_+(t) \mathcal{Y}_+ + \alpha_-(t) \mathcal{Y}_- + \sum_k \beta_k(t) f_k + g(t), \quad g(t)\in \tilde{G}^\perp,
\end{gather}
where, recalling that $\mathcal{L}_{|\text{span}\{f_k,\, k \leq N+1\}} = 0$ and that $\Phi(\mathcal{Y}_+) = \Phi(\mathcal{Y}_-) = 0$,
\begin{gather}
	\label{def_alpha}\alpha_+(t) = B(h(t), \mathcal{Y}_-), \quad \alpha_-(t) = B(h(t), \mathcal{Y}_+),\\
	\label{def_beta}\beta_k(t) = (h(t),f_k)_{\dot{H}^1}-\alpha_+(t)(\mathcal{Y}_+,f_k)_{\dot{H}^1}-\alpha_-(t)(\mathcal{Y}_-,f_k)_{\dot{H}^1}, \quad \forall k \leq N+1.
\end{gather}
\textit{Step 1. Differential inequalities on the coefficients}. We will show 
\begin{gather}
\label{diff_alpha}\frac{d}{dt}\left(e^{e_0t}\alpha_+(t)\right) = e^{e_0 t} B(\mathcal{Y}_-,\epsilon), 
\quad \frac{d}{dt}\left(e^{-e_0t}\alpha_-(t)\right) = e^{-e_0 t} B(\mathcal{Y}_+,\epsilon),\\
\label{diff_beta} \frac{d}{dt}\left(e^{-e_0t}\beta_k(t)\right) =  (f_k, \epsilon)_{\dot{H}^1}-(\mathcal{Y}_+,f_k)_{\dot{H}^1}B(\mathcal{Y}_-,\epsilon)-(\mathcal{Y}_-,f_k)_{\dot{H}^1}B(\mathcal{Y}_+,\epsilon)-(\mathcal{L}g,f_k)_{\dot{H}^1},\\
\label{diff_phi}\frac{d\Phi(h(t))}{dt}= 2B(h, \epsilon).
\end{gather}

By the equation \eqref{linear_exp},
\begin{align}\label{alpha_estim}
	\alpha_+'(t) &= B(\partial_t h, \mathcal{Y}_-) = B(-\mathcal{L}h+\epsilon,\mathcal{Y}_-)\\
	&= B(h, \mathcal{L}\mathcal{Y}_-)+B(\epsilon, \mathcal{Y}_-)=- e_0\alpha_+(t)+B(\epsilon, \mathcal{Y}_-).
\end{align}

This yields the first equation in \eqref{diff_alpha}. The second equation follows similarly.

Now, differentiating \eqref{def_beta}, we obtain
\begin{equation}
	\beta_k' = (-\mathcal{L}h + \epsilon -\alpha_+' \mathcal{Y}_+-\alpha_-'\mathcal{Y}_-, f_k)_{\dot{H}^1}.
\end{equation}

Note that $\mathcal{L}h = e_0\alpha_+\mathcal{Y}_+ - e_0 \alpha_-\mathcal{Y}_-+\mathcal{L}g$, by \eqref{def_h}, which proves \eqref{diff_beta}, in view of \eqref{diff_alpha}.

Finally, differentiating $\Phi(h(t))$,
\begin{equation}
	\frac{d}{dt}\Phi(t) = 2B(h, \partial_t h) = -2B(h, \mathcal{L}h)+2B(h,\epsilon) = 2B(h,\epsilon),
\end{equation}
by the skew-symmetry of $\mathcal{L}$ in \eqref{prop_B}. Equation \eqref{diff_phi} is then proved.

\textit{Step 2. Estimates on} $\alpha_\pm$. We claim
\begin{align}
\label{decay_alpha_1}|\alpha_-(t)| &\lesssim e^{-c_1 t},\\
\label{decay_alpha_2}|\alpha_+(t)| &\lesssim 
\begin{cases}
	e^{-c_1 t} &\text{if } e_0 < c_0,\\
	e^{-e_0t}+e^{-c_1^- t}&\text{if } e_0 \geq c_0.\\
\end{cases}
\end{align}

We need the following claim, which is an immediate application of Hölder's inequality.
\begin{claim}\label{claim_B} If $I$ is a finite time interval, $f \in L^\infty_I L^\frac{2N}{N-2}$, $g \in L^\infty_I L_x^{\frac{2N}{N+2}}$ are such that $\nabla f \in L^2_I L_x^\frac{2N}{N-2}$, $\nabla g \in L^2_I L_x^{\frac{2N}{N+2}}$, then
\begin{equation}
	\int_I|B(f(t),g(t))|dt \lesssim \|\nabla f\|_{L^2_I L_x^\frac{2N}{N-2}}\|\nabla g\|_{L^2_I L_x^{\frac{2N}{N+2}}} +  |I|\|f\|_{L^\infty_I L_x^\frac{2N}{N-2}}\|g\|_{L^\infty_I L_x^{\frac{2N}{N+2}}}.
\end{equation}
\end{claim}

The above claim, \eqref{local_boot} and \eqref{diff_alpha} yield
\begin{equation}
\int_t^{t+1}|e^{-e_0s}B(\mathcal{Y}_+,\epsilon(s))|ds \leq 	e^{-e_0t}\int_t^{t+1}|B(\mathcal{Y}_+,\epsilon(s))|ds 
\lesssim e^{-(e_0+c_1)t}.
\end{equation}
By triangle inequality, integrating the second equation in \eqref{diff_alpha} gives
\begin{equation}
	|\alpha_-(t)| \lesssim e^{e_0t} \int_t^{+\infty}|e^{-e_0s}B(\mathcal{Y}_+,\epsilon(s))|ds \lesssim e^{-c_1t},
\end{equation}
which proves \eqref{decay_alpha_1}.

To prove \eqref{decay_alpha_2}, consider first the case $c_0 > e_0$. Then, by \eqref{local_boot} and \eqref{def_alpha}, $e^{e_0t}\alpha_+(t)$ vanishes as $t \to +\infty$. By Claim \ref{claim_B}
\begin{equation}
	\int_t^{t+1}|e^{e_0s}B(\mathcal{Y}_+,\epsilon(s))|ds \lesssim e^{e_0t}\int_t^{t+1}|B(\mathcal{Y}_+,\epsilon(s))|ds 
	\lesssim e^{(e_0-c_1)t},
\end{equation}
integrating the equation on $\alpha_+$ in \eqref{diff_alpha}, recalling that $c_1>c_0$, and using triangle inequality, we get \eqref{decay_alpha_2} if $c_0 > e_0$.

Assume now that $c_0 \leq e_0$. Integrating \eqref{diff_alpha},
\begin{equation}
	|\alpha_+(t)-e^{-e_0t}\alpha_+(0)| \leq e^{-e_0t}\int_0^te^{e_0s}|B(\mathcal{Y}_-,\epsilon(s))|ds 
	 \lesssim e^{-c_1^- t},
\end{equation}

and the proof of \eqref{decay_alpha_2} is finished.

\textit{Step 3. Bounds on $g$ and $\beta_k$.} We will prove
\begin{equation}\label{decay_beta}
	\|g(t)\|_{\dot{H}^1} + \sum_k |\beta_k(t)| \lesssim e^{-\frac{(c_0+c_1)}{2}t}.
\end{equation}

Again by Claim \ref{claim_B}, $\int_t^{t+1}|B(h(s),\epsilon(s))|ds \lesssim e^{-(c_0+c_1)t}$. By triangle inequality, integrating \eqref{diff_phi}, we get
\begin{equation}
	\Phi(h(t)) \lesssim e^{-(c_0+c_1)t}.
\end{equation}
Therefore,
\begin{equation}
	|2\alpha_+\alpha_-B(\mathcal{Y}_+,\mathcal{Y}_-)+\Phi(g)| = |\Phi(h)| \lesssim e^{-(c_0+c_1)t}. 
\end{equation}
By Step 2,
\begin{equation}
	|\Phi(g)| \lesssim \begin{cases}
	e^{-(c_0+c_1)t}+ e^{-2c_1t} &\text{if } c_0 > e_0,\\
	e^{-(c_0+c_1)t}+ e^{-c_1t}(e^{-e_0t}+e^{-c_1^-t}) &\text{if } c_0 \leq e_0.\\
	\end{cases}
\end{equation}
In any case, $|\Phi(g)| \leq e^{-(c_0+c_1)t}$. Using the coercivity of $\Phi$, given by Lemma \ref{lem_coerc}, estimate for $g$ in \eqref{decay_beta} is proven.

Consider now \eqref{diff_beta}. By \eqref{local_boot},
\begin{equation}
	\beta_k(t+1)-\beta_k(t) \lesssim e^{-c_1t} + \int_t^{t+1}|(f_k,\mathcal{L}g(s))_{\dot{H}^1}| ds =
	e^{-c_1t} + \int_t^{t+1}\left|\Re\int\mathcal{L}^*(\Delta f_k) \overline{g}(s)\right|d(s),
\end{equation}
where $\mathcal{L}^* = \begin{pmatrix}
	0 &L_+\\-L_- & 0
\end{pmatrix}$ 
is the $L^2$-adjoint of $\mathcal{L}$.

One can check explicitly that, for any $0 \leq k \leq N+1$, $|\mathcal{L}^*(\Delta f_k)| \lesssim \frac{1}{1+|x|^{N+4}}$. Therefore, $\mathcal{L}^*(\Delta f_k) \in L^\frac{2N}{N+2}(\Real^N)$, so that, by the estimate on $g$ in \eqref{decay_beta}, we obtain
\begin{equation}
	\left|\Re\int\mathcal{L}^*(\Delta f_k) \overline{g}(t)\right|\lesssim \|g(t)\|_{\frac{2N}{N-2}} \lesssim \|g\|_{\dot{H}^1} \lesssim e^{-\frac{(c_0+c_1)}{2}t}.
\end{equation}

\textit{Step 4. Closure}

By the decomposition \eqref{def_h}, as well as Steps $2$ and $3$, so far we have
\begin{equation}
	\|h(t)\|_{\dot{H}^1} \lesssim 
	\begin{cases} 
		e^{-\frac{(c_0+c_1)}{2}t} &\text{if } c_0 > e_0\\
		e^{-e_0t} + e^{-\frac{(c_0+c_1)}{2}t} &\text{if } c_0 \leq e_0.\\
		\end{cases}
\end{equation}

Now, if $e_0 \notin [c_0,c_1)$, by iterating the argument, we obtain
\begin{equation}
	\|h(t)\|_{\dot{H}^1} \lesssim e^{-c_1^- t},
\end{equation}
which proves \eqref{lem_self_1}. 

Assume now $e_0 \in [c_0,c_1)$. Then, the estimate \eqref{diff_alpha} on $\alpha_+$ ensures the existence of a limit $A$ to $e^{e_0t}\alpha_+(t)$ as $t \to +\infty$. Integrating \eqref{diff_alpha} from $t$ to $+\infty$, we get
\begin{equation}
	|A - e^{e_0t}\alpha_+| \leq e^{e_0t}\int_t^{+\infty}|B(\mathcal{Y}_+,\epsilon(s))| ds  \lesssim e^{(e_0-c_1)t}.
\end{equation}

In view of the decomposition \eqref{def_h} and estimates \eqref{decay_alpha_1}, \eqref{decay_alpha_2} and \eqref{decay_beta}, we get
\begin{equation}
	\|h(t) - A e^{-e_0t}\mathcal{Y}_+\|_{\dot{H}^1} \lesssim e^{-\frac{(c_0+c_1)}{2}t}. 
\end{equation}

Since $\mathcal{L}\mathcal{Y}_+ = e_0 \mathcal{Y}_+$, we see that $\tilde{h}(t) := h(t) - A e^{-e_0t}\mathcal{Y}_+$ satisfies the differential equation \eqref{linear_exp} with the same $\epsilon$, and with $c_0$ replaced by $\frac{c_0+c_1}{2}>c_0$ in the condition \eqref{exp_decay}. By iterating the argument a finite number of times, we end up under the condition \eqref{lem_self_1}, which implies condition \eqref{lem_self_2} for the original $h$, and finishes the proof of Lemma \ref{lem_self}.
\end{proof}

\begin{proof}[Proof of Lemma \ref{sub_lem_self}]
We first normalize the eigenfunctions of $\mathcal{L}$. Define
\begin{equation}
 \quad f_{0} := \frac{iQ}{\|Q\|_{L^2}},\quad f_k := \frac{\partial_k Q}{\|\partial_k Q\|_{L^2}}, 1 \leq k \leq N.
\end{equation}
We have
\begin{equation}
	B(f_k, h) = 0, \quad \|f_k\|_{L^2} = 1, \quad \forall 0 \leq  k \leq N, \, \forall h \in H^1.
\end{equation}
Recall that $B(\mathcal{Y}_+, \mathcal{Y}_-)\neq 0$. Normalize $\mathcal{Y}_+$, $\mathcal{Y}_-$ such that $B(\mathcal{Y}_+, \mathcal{Y}_-)=1$. Next, write
\begin{gather}
\label{sub_def_h}h(t) = \alpha_+(t) \mathcal{Y}_+ + \alpha_-(t) \mathcal{Y}_- + \sum_k \beta_k(t) f_k + g(t), \quad g(t)\in \tilde{G}^\perp,
\end{gather}
where, recalling that $\mathcal{L}_{|\text{span}\{f_k,\, k \leq N\}} = 0$ and that $\Phi(\mathcal{Y}_+) = \Phi(\mathcal{Y}_-) = 0$, we have
\begin{gather}
	\label{sub_def_alpha}\alpha_+(t) = B(h(t), \mathcal{Y}_-), \quad \alpha_-(t) = B(h(t), \mathcal{Y}_+),\\
	\label{sub_def_beta}\beta_k(t) = (h(t),f_k)_{H^1}-\alpha_+(t)(\mathcal{Y}_+,f_k)_{H^1}-\alpha_-(t)(\mathcal{Y}_-,f_k)_{H^1}, \quad \forall k \leq N.
\end{gather}
\textit{Step 1. Differential inequalities on the coefficients}. We show 
\begin{gather}
\label{sub_diff_alpha}\frac{d}{dt}\left(e^{e_0t}\alpha_+(t)\right) = e^{e_0 t} B(\mathcal{Y}_-,\epsilon), 
\quad \frac{d}{dt}\left(e^{-e_0t}\alpha_-(t)\right) = e^{-e_0 t} B(\mathcal{Y}_+,\epsilon),\\
\label{sub_diff_beta} \frac{d}{dt}\left(e^{-e_0t}\beta_k(t)\right) =  (f_k, \epsilon)_{H^1}-(\mathcal{Y}_+,f_k)_{H^1}B(\mathcal{Y}_-,\epsilon)-(\mathcal{Y}_-,f_k)_{H^1}B(\mathcal{Y}_+,\epsilon)-(\mathcal{L}g,f_k)_{H^1},\\
\label{sub_diff_phi}\frac{d\Phi(h(t))}{dt}= 2B(h, \epsilon).
\end{gather}

By the equation \eqref{sub_linear_exp},
\begin{align}\label{sub_alpha_estim}
	\alpha_+'(t) &= B(\partial_t h, \mathcal{Y}_-) = B(-\mathcal{L}h+\epsilon,\mathcal{Y}_-)\\
	&= B(h, \mathcal{L}\mathcal{Y}_-)+B(\epsilon, \mathcal{Y}_-)=- e_0\alpha_+(t)+B(\epsilon, \mathcal{Y}_-).
\end{align}

This yields the first equation in \eqref{sub_diff_alpha}. The second equation follows similarly.

Now, differentiating \eqref{sub_def_beta}, we obtain
\begin{equation}
	\beta_k' = (\mathcal{L}h + \epsilon -\alpha_+' \mathcal{Y}_+-\alpha_-'\mathcal{Y}_-, f_k)_{H^1}.
\end{equation}

Note that $\mathcal{L}h = e_0\alpha_+\mathcal{Y}_+ - e_0 \alpha_-\mathcal{Y}_-+\mathcal{L}g$, by \eqref{sub_def_h}, which proves \eqref{sub_diff_beta}, in view of \eqref{sub_diff_alpha}.

Finally, differentiating $\Phi(h(t))$,
\begin{equation}
	\frac{d}{dt}\Phi(t) = 2B(h, \partial_t h) = -2B(h, \mathcal{L}h)+2B(h,\epsilon) = 2B(h,\epsilon),
\end{equation}
by the skew-symmetry of $\mathcal{L}$ in \eqref{prop_B}. The equation \eqref{sub_diff_phi} is then proved.

\textit{Step 2. Estimates on} $\alpha_\pm$. We claim
\begin{align}
\label{sub_decay_alpha_1}|\alpha_-(t)| &\lesssim e^{-c_1 t},\\
\label{sub_decay_alpha_2}|\alpha_+(t)| &\lesssim 
\begin{cases}
	e^{-c_1 t} &\text{if } c_0 \leq e_0,\\
	e^{-e_0t}+e^{-c_1^- t}&\text{if } c_0 > e_0.\\
\end{cases}
\end{align}

We will need the following inequality, which is an immediate application of Hölder's inequality.
\begin{equation}\label{sub_claim_B}
    \int_I|B(f(t),g(t))| dt \lesssim \|\langle \nabla \rangle f\|_{S(L^2,\,I)}\|\langle \nabla \rangle g\|_{S'(L^2,\,I)}.
\end{equation}

The above inequality, assumption \eqref{sub_exp_decay} and \eqref{sub_diff_alpha} yield
\begin{equation}
\int_t^{+\infty}|e^{-e_0s}B(\mathcal{Y}_+,\epsilon(s))|ds \leq 	e^{-e_0t}\int_t^{+\infty}|B(\mathcal{Y}_+,\epsilon(s))|ds 
\lesssim e^{-(e_0+c_1)t}.
\end{equation}
By integrating the second equation in \eqref{sub_diff_alpha} gives
\begin{equation}
	|\alpha_-(t)| \lesssim e^{e_0t} \int_t^{+\infty}|e^{-e_0s}B(\mathcal{Y}_+,\epsilon(s))|ds \lesssim e^{-c_1t},
\end{equation}
which proves \eqref{sub_decay_alpha_1}.

To prove \eqref{sub_decay_alpha_2}, consider first the case $c_1 > c_0 > e_0$. Then, by assumption \eqref{sub_exp_decay} and \eqref{sub_def_alpha}, $e^{e_0t}\alpha_+(t)$ vanishes as $t \to +\infty$. By \eqref{sub_claim_B}, integrating the equation on $\alpha_+$ in \eqref{sub_diff_alpha},
\begin{equation}
	|e^{e_0t}\alpha_+(t) | \lesssim \int_t^{+\infty}|e^{e_0s}B(\mathcal{Y}_+,\epsilon(s))|ds
	\lesssim e^{(e_0-c_1)t},
\end{equation}
and we get \eqref{sub_decay_alpha_2} if $c_0 > e_0$.

Assume now that $c_0 \leq e_0$. By \eqref{sub_diff_alpha}, we have
\begin{equation}
	|\alpha_+(t)-e^{-e_0t}\alpha_+(0)| \leq e^{-e_0t}\int_0^te^{e_0s}|B(\mathcal{Y}_-,\epsilon(s))|ds 
	 \lesssim e^{-c_1^- t},
\end{equation}

and the proof of \eqref{sub_decay_alpha_2} is finished.

\textit{Step 3. Bounds on $g$ and $\beta_k$.} We prove
\begin{equation}\label{sub_decay_beta}
	\|g(t)\|_{H^1} + \sum_k |\beta_k(t)| \lesssim e^{-\frac{(c_0+c_1)}{2}t}.
\end{equation}

Again by \eqref{sub_claim_B}, $\int_t^{+\infty}|B(h(s),\epsilon(s))|ds \lesssim e^{-(c_0+c_1)t}$. By integrating \eqref{sub_diff_phi}, we get
\begin{equation}
	\Phi(h(t)) \lesssim e^{-(c_0+c_1)t}.
\end{equation}
Therefore,
\begin{equation}
	|2\alpha_+\alpha_-B(\mathcal{Y}_+,\mathcal{Y}_-)+\Phi(g)| = |\Phi(h)| \lesssim e^{-(c_0+c_1)t}. 
\end{equation}
By Step 2,
\begin{equation}
	|\Phi(g)| \lesssim \begin{cases}
	e^{-(c_0+c_1)t}+ e^{-2c_1t} &\text{if } c_0 > e_0,\\
	e^{-(c_0+c_1)t}+ e^{-c_1t}(e^{-e_0t}+e^{-c_1^-t}) &\text{if } c_0 \leq e_0.\\
	\end{cases}
\end{equation}
In any case, $|\Phi(g)| \leq e^{-(c_0+c_1)t}$. Using the coercivity of $\Phi$, given by Lemma \ref{lem_coerc}, estimate for $\Phi$ in \eqref{sub_decay_beta} is proven.

Consider now \eqref{sub_diff_beta}. By the assumption \eqref{sub_exp_decay},
\begin{equation}
	\beta_k(t+1)-\beta_k(t) \lesssim e^{-c_1t} + \int_t^{t+1}|(f_k,\mathcal{L}g(s))_{H^1}| ds =
	e^{-c_1t} + \int_t^{t+1}\left|\Re\int\mathcal{L}^*(\Delta f_k) \overline{g}(s)\right|d(s),
\end{equation}
where $\mathcal{L}^* = \begin{pmatrix}
	0 &L_+\\-L_- & 0
\end{pmatrix}$ 
is the $L^2$-adjoint of $\mathcal{L}$.

One can check that, for any $0 \leq k \leq N$, $|\mathcal{L}^*(\Delta f_k)| \lesssim e^{-|x|}$. Therefore, $\mathcal{L}^*(\Delta f_k) \in L^2$ so that, by the estimate on $g$ in \eqref{sub_decay_beta},
\begin{equation}
	\left|\Re\int\mathcal{L}^*(\Delta f_k) \overline{g}(t)\right|\lesssim \|g(t)\|_{L^2} \leq \|g\|_{H^1} \lesssim e^{-\frac{(c_0+c_1)}{2}t}.
\end{equation}

\textit{Step 4. Closure}

By the decomposition \eqref{sub_def_h}, as well as Steps $2$ and $3$, so far we have
\begin{equation}
	\|h(t)\|_{H^1} \lesssim 
	\begin{cases} 
		e^{-\frac{(c_0+c_1)}{2}t} &\text{if } c_0 > e_0\\
		e^{-e_0t} + e^{-\frac{(c_0+c_1)}{2}t} &\text{if } c_0 \leq e_0.\\
		\end{cases}
\end{equation}

Now, if $e_0 \notin [c_0,c_1)$, by iterating the argument, we obtain
\begin{equation}
	\|h(t)\|_{H^1} \lesssim e^{-c_1^- t},
\end{equation}
which proves \eqref{sub_lem_self_1}. 

Assume now $e_0 \in [c_0,c_1)$. Then, the estimate \eqref{sub_diff_alpha} on $\alpha_+$ ensures the existence of a limit $A$ to $e^{e_0t}\alpha_+(t)$, as $t \to +\infty$. Integrating \eqref{sub_diff_alpha} from $t$ to $+\infty$,
\begin{equation}
	|A - e^{e_0t}\alpha_+| \leq e^{e_0t}\int_t^{+\infty}|B(\mathcal{Y}_+,\epsilon(s))| ds  \lesssim e^{(e_0-c_1)t}.
\end{equation}

In view of decomposition \eqref{sub_def_h} and estimates \eqref{sub_decay_alpha_1}, \eqref{sub_decay_alpha_2} and \eqref{sub_decay_beta}, we get
\begin{equation}
	\|h(t) - A e^{-e_0t}\mathcal{Y}_+\|_{H^1} \lesssim e^{-\frac{(c_0+c_1)}{2}t}. 
\end{equation}

Since $\mathcal{L}\mathcal{Y}_+ = e_0 \mathcal{Y}_+$, we see that $\tilde{h}(t) := h(t) - A e^{-e_0t}\mathcal{Y}_+$ satisfies the differential equation \eqref{sub_linear_exp} with the same $\epsilon$, and with $c_0$ replaced by $\frac{c_0+c_1}{2}>c_0$ in the condition \eqref{sub_exp_decay}. By iterating the argument a finite number of times, we end up under the condition \eqref{sub_lem_self_1}, which implies condition \eqref{sub_lem_self_2} for the original $h$, and finishes the proof of Lemma \ref{sub_lem_self}.
\end{proof}




\addtocontents{toc}{\protect\vspace*{\baselineskip}}






\newcommand{\Addresses}{{
  \bigskip
  \footnotesize

  L. Campos, \textsc{Department of Mathematics, UFMG, Brazil}\par\nopagebreak
  \textit{E-mail address:} \texttt{luccasccampos@gmail.com}
  
  \vspace{3mm}
   L. G. Farah, \textsc{Department of Mathematics, UFMG, Brazil}\par\nopagebreak
  \textit{E-mail address:} \texttt{farah@mat.ufmg.br}
  
  \vspace{3mm}
   S. Roudenko, \textsc{Department of Mathematics and Statistics, FIU, USA}\par\nopagebreak
  \textit{E-mail address:} \texttt{sroudenko@fiu.edu}
}}
\setlength{\parskip}{0pt}
\Addresses
\batchmode
\end{document}